\documentclass[11pt]{amsart}
\usepackage{amsmath}
\usepackage{cases}
\usepackage{amssymb, verbatim}
\usepackage{comment}

\title[Poisson metrics on flat vector bundles]{Poisson metrics on flat vector bundles over non-compact curves}
\author[T.C. Collins]{Tristan C. Collins}
\address{Department of Mathematics, Columbia University, 2990 Broadway, New York, NY 10027}
 \email{tcollins@math.columbia.edu}
\author[A. Jacob]{Adam Jacob*}
\thanks{$^{*}$Supported in part by NSF Grant No. DMS-1204155. .}
 \address{Department of Mathematics, Harvard University, 1 Oxford St., Cambridge, MA 02138}
 \email{ajacob@math.harvard.edu}
\author[S.-T. Yau]{Shing-Tung Yau}
 \email{yau@math.harvard.edu}

\theoremstyle{plain}
\newtheorem{thm}{Theorem}[section]
\newtheorem{prop}[thm]{Proposition}
\newtheorem{defn}[thm]{Definition}
\newtheorem{lem}[thm]{Lemma}
\newtheorem{cor}[thm]{Corollary}

\theoremstyle{definition}
\newtheorem{ex}[thm]{Example}
\newtheorem{rk}[thm]{Remark}

\numberwithin{equation}{section}

\newcommand{\Tr}{{\rm Tr}}

\newcommand{\del}{\partial}

\newcommand{\dbar}{\overline{\del}}

\newcommand{\ddr}{\frac{\del}{\del r}}

\newcommand{\D}{\mathcal{D}}

\newcommand{\ti}[1]{\tilde{#1}}

\newcommand{\pl}{\del}
\newcommand{\be}{\begin{equation*}}
\newcommand{\ee}{\end{equation*}}

\renewcommand{\leq}{\leqslant}
\renewcommand{\geq}{\geqslant}

\renewcommand{\epsilon}{\varepsilon}
\renewcommand{\phi}{\varphi}

\def\t{\theta}

\begin{document}
\maketitle
\begin{abstract}
Let $(E,\nabla, \Pi)\rightarrow (M,g)$ be a flat vector bundle with a parabolic structure over a punctured Riemann surface.  We consider a deformation of the harmonic metric equation which we call the Poisson metric equation.  This equation arises naturally as the dimension reduction of the Hermitian-Yang-Mills equation for holomorphic vector bundles on $K3$ surfaces in the large complex structure limit.  We define a notion of slope stability, and show that if the flat connection $\nabla$ has regular singularities, and the Riemannian metric $g$ has finite volume then $E$ admits a Poisson metric with asymptotics determined by the parabolic structure if and only if $(E,\nabla,\Pi)$ is  slope polystable.
\end{abstract}

\section{Introduction}
In this paper we study a class of canonical metrics on complex vector bundles over a non-compact curve.  These canonical metrics arise naturally in both mathematics and physics.  Our primary motivation is to understand the notion of stable vector bundles on a $K3$ surface, in the large complex structure limit.  Let $(X,g)$ be a compact K\"ahler manifold, and let $E\rightarrow X$ be a holomorphic vector bundle.   A smooth hermitian metric $H$ on $E$ gives rise to the unitary Chern connection $\nabla$ by requiring that $\dbar$ is compatible with the inner product induced by $H$.  The metric  is said to be Hermitian-Yang-Mills if the ${\rm End}(E)$ valued $(1,1)$-form $F_{\bar{k}j} := [\nabla_{j}, \nabla_{\bar{k}}]$ satisfies
\begin{equation*}
g^{j\bar{k}}F_{\bar{k}j}  = \mu(E) I, \qquad \mu(E) = \frac{\deg(E)}{rk(E)Vol(X,g)}.
\end{equation*}
The existence of a Hermitian-Yang-Mills metric is not automatic, and is equivalent to the algebro-geometric notion of Mumford-Takemoto stability.  This deep correspondence was first elucidated by Narasimhan-Seshadri \cite{NS} when $X$ is a curve, by Donaldson \cite{Don1} when $X$ is a projective surface, and for general K\"ahler manifolds by Uhlenbeck-Yau \cite{UY}.  Notice that the notion of a Hermitian-Yang-Mills metric, and hence stability, depends on the K\"ahler class of the metric $g$.

Suppose now that $X$ is a K\"ahler, Calabi-Yau manifold of real dimension $2n$.  A large complex structure limit of $X$ is, in essence, the worst degeneration of complex structures on $X$.  These degenerations play a fundamental role in mirror symmetry.   The Strominger-Yau-Zaslow conjecture \cite{SYZ} states that near a large complex structure limit point, $X$ admits a fibration by special Lagrangian $n$-tori, which we denote by $\pi: X \rightarrow B$.  The base $B$ is an affine manifold of real dimension $n$, away from a singular set of dimension $n-2$.  The SYZ conjecture states that, roughly,  the Calabi-Yau manifold $X$ is obtained by taking a fiberwise quotient of $TB \rightarrow B$ by a lattice $\Lambda$, and the mirror manifold $\check{\pi}: \check{X}\rightarrow B$ is obtained by taking the quotient by the dual lattice $\Lambda^{*}$.  We refer the reader to Kontsevich-Soibelman \cite{KS} and Gross-Wilson \cite{GW} for precise statements, and important refinements of this conjecture.  Recipes for constructing SYZ mirror symmetry have been developed by Auroux \cite{Aur} using symplectic techniques, and Chan-Lau-Cheung \cite{CLC} for toric Calabi-Yau manifolds.  It should also be mentioned that Gross-Siebert have developed a new approach to the SYZ conjecture based on tropical geometry and the wall-crossing machinery of Kontsevich-Soibelman \cite{KS}, and have made deep and fundamental contributions using these new ideas; see, for example, \cite{GS, GS1, GS2}.  We note that the notion of a special Lagrangian depends on two choices; the choice of a holomorphic $(n,0)$ form on $X$, and the choice of a Ricci-flat K\"ahler metric $\omega$, whose existence is guaranteed by the third-named author's solution of the Calabi conjecture \cite{Y}.  Furthermore, it is a consequence of the SYZ conjecture that mirror symmetry is deeply related to the limits of Ricci-flat K\"ahler metrics \cite{GW}.  Under the SYZ correspondence holomorphic vector bundles $E \rightarrow X$ correspond to Lagrangian submanifolds in the mirror $\check{X}$, and stable bundles correspond to special Lagrangian submanifolds in $\check{X}$ with flat $U(1)$ connections \cite{V} .

It turns out (see Section~\ref{geom motiv})  that if $F \rightarrow B$ is a flat vector bundle over the base, then $F$ gives rise naturally to holomorphic vector bundles $E \rightarrow X$ and $\check{E} \rightarrow \check{X}$, at least away from the singular fibers of $\pi, \check{\pi}$.  Many examples of bundles which fit into this framework were constructed by Friedman-Morgan-Witten \cite{FMW}.  In the current work, we are motivated by the following question: Is there a condition on $F \rightarrow B$ that guarantees that both $E$ and $\check{E}$ are stable?  Note that, in order for this question to make sense, it is necessary to choose K\"ahler metrics on both $X$ and $\check{X}$.

Let $\pi: X \rightarrow \mathbb{P}^{1}$ be an elliptic $K3$ surface with $24$ singular fibers of type $I_{1}$. In this setting Greene-Shapere-Vafa-Yau constructed a Monge-Amp\`ere metric on the base of the fibration \cite{BSVY}. Subsequently, Gross-Wilson studied limits of Ricci-flat metrics on $X$ when the volume of the fibers of $\pi$ tends to zero \cite{GW}.   In this case, the special Lagrangian fibration is obtained by first finding an elliptic fibration of $X$, and then performing a hyperk\"ahler rotation of the complex structure.  Gross-Wilson show that if the Ricci-flat metrics are rescaled to have bounded diameter, then the Calabi-Yau metrics converge in the Gromov-Hausdorff sense to a Hessian metric of Monge-Amp\`ere type on the punctured Riemann sphere with prescribed singularities corresponding to the singular fibers of $\pi$.  This result was extended to general projective Calabi-Yau's admitting abelian fibrations by Gross-Tosatti-Zhang \cite{GTZ}. Many examples of Monge-Amp\`ere type metrics on the punctured Riemann sphere were constructed by Loftin \cite{Loft1}, with the same type of singularity as those found in \cite{GW}.  These Monge-Amp\`ere metrics also pull-back to define semi-flat K\"ahler metrics, away from the singular fibers of $\pi$, $\check{\pi}$, which are close approximations of the Ricci-flat metrics near the large complex structure limit.  It is therefore natural, in the question posed above, to equip $X$, $\check{X}$, and $B$ with these singular, semi-flat Calabi-Yau metrics.

Our approach to this problem is to study the dimension reduction of the Hermitian-Yang-Mills equation to the base of the special Lagrangian fibration, $B$.  We consider the case when $X$ is a $K3$ surface, and so explicitly, $B = \mathbb{P}^{1} \backslash \{p_{1}, \dots, p_{j}\}$, equipped with a Hessian metric of Monge-Amp\`ere type, $\phi_{ij}$, singular near the punctures.  The dimension reduction associates to $E \rightarrow X$ a flat vector bundle $(F,\nabla) \rightarrow B$.  We assume that the flat connection $\nabla$ has regular singularities near the punctures.  The dimension reduction of the Hermitian-Yang-Mills equation on $E \rightarrow X$ is given in affine coordinates $(x^1,x^2)$ for $B$ and a flat frame for $F$, by
\begin{equation}\label{eq: AHYM intro}
-\frac{1}{4}\phi^{ij} \frac{\del}{\del x^{i}} \left( H^{-1} \frac{\del}{\del x^{j}} H\right) = \lambda I.
\end{equation}
In order to compute the constant $\lambda$ on the right hand side, it is necessary to fix the asymptotics of the Hermitian metric $H$ near the punctures $p_{j}$.  This data is contained in a parabolic structure on $(F,\nabla)$.  Then a corollary of our main theorem is
\begin{cor}
Suppose $(F,\nabla) \rightarrow B := \mathbb{P}^{1}\backslash\{p_{1},\dots,p_{j}\}$ is a flat vector bundle with regular singularities at the punctures.  Let $\Pi$ denote a fixed parabolic structure on $(F,\nabla)$.  Let $\omega_{FS}$ denote the Fubini-Study metric, and suppose that $B$ admits a Hessian metric of Monge-Amp\`ere type, $\phi_{ij} = e^{\psi}\omega_{FS}$ with $e^{\psi} \in L^{1}(\mathbb{P}^{1},\omega_{FS})$.  Then there exists a conformally strongly tame hermitian metric $H$ on $(F,\nabla,\Pi)$ satisfying~\eqref{eq: AHYM intro} if and only if the parabolic bundle $(F,\nabla, \Pi)$ is slope polystable.  Moreover, if $e^{\psi} \in L^{p}(\overline{M}, \omega_{FS})$ for some $p>2$, then $H$ is strongly tamed by the parabolic structure.
\end{cor}

That $H$ is tame means that the asymptotics of $H$ near each puncture $p_{j}$ are well controlled, and determined by the parabolic structure.  We refer the reader to the body of the paper for the precise definitions.  This problem was solved by Loftin \cite{Loft2}, in general dimension but with the additional assumption that the base affine manifold is smooth and compact. Additional results in the compact case were proved by Biswas-Loftin \cite{BL}, and Biswas-Loftin-Stemmler \cite{BLS1,BLS2}. This setting is rather restrictive from the point of view of mirror symmetry. Cheng-Yau \cite{CY} showed that the only compact, Calabi-Yau manifolds with semi-flat, Ricci-flat metrics which arise from pulling back Monge-Amp\`ere metrics from a compact affine base are complex tori.  Indeed, the puncture points arising in our setting come from the singular fibers of the elliptic fibration $X \rightarrow B$.

The study of flat bundles with a parabolic structure on a punctured Riemann surface is by no means new.  There is a great deal of literature examining their geometric and algebraic properties, and the problem we consider here fits naturally into this larger framework.  Suppose that $(X,g)$ is a compact K\"ahler manifold, and let $(E,\nabla) \rightarrow X$ be a flat, complex vector bundle of rank $n$.  A smooth metric $H$ on $(E,\nabla)$ is said to be harmonic if $H$ defines an equivariant map from the universal cover of $X$ to $Gl(n)/U(n)$. Equivalently, the metric $H$ splits the connection as $\nabla = d + A - \Psi$ where $d+A$ is an $H$ unitary connection and $\Psi$ is a self-adjoint endormorphism valued $1$-form.  Then $H$ is harmonic if
\begin{equation*}
\star \nabla \star \Psi =0.
\end{equation*}
In the treatise \cite{Simp3}, Simpson shows that the category of flat bundles admitting a harmonic metric is equivalent to the category of stable Higgs bundles of degree zero on $X$.  This correspondence is induced by the existence of canonical metrics on the objects in either category.  Recall that a Higgs bundle on $X$ is a holomorphic vector bundle $E$  together with a section $\theta \in H^{0}(X, {\rm End}(E) \otimes \Omega^{1}_{X})$ satisfying $\theta \wedge \theta=0$. The notion of a Higgs bundle was introduced by   Hitchin \cite{Hit},  where the Donaldson-Uhlenbeck-Yau theorem was extended to the setting of Higgs bundles over compact curves.  This result was generalized by Simpson \cite{Simp}  to the setting of Higgs bundles over general K\"ahler manifolds making use of the ideas of Uhlenbeck-Yau \cite{UY}.  The correspondence developed by Simpson \cite{Simp3} says roughly that if $(E,\nabla)$ is a flat vector bundle on $X$ with a harmonic metric, then $E$ admits the structure of a stable Higgs bundle of degree $0$.   Simpson later showed that this correspondence was also valid over non-compact curves \cite{Simp2}.  
 
 The existence of harmonic metrics on flat bundles over non-compact manifolds is a fundamental problem in geometry with many applications.  For example, Jost-Zuo \cite{JZ} used the theory of harmonic metrics on quasi-projective varieties to deduce rigidity results about representations of $\pi_{1}(X)$ for quasi-projective varieties.  Recently, groundbreaking progress in the theory of constructible perverse sheaves has been made by Mochizuki making use of the existence of harmonic metrics with precise asymptotics on quasi-projective varieties, as well as far reaching generalizations of the Donaldson-Uhlenbeck-Yau theorem, and the non-abelian Hodge theory of Corlette-Simpson \cite{Simp, Simp2}; see \cite{TM, TM1, TM2, TM3} as well as the references therein.
 
 It turns out that the mirror symmetry problem described above falls naturally into this framework, as a twist of the harmonic metric equation.  Suppose that $(M, g)$ is a non-compact Riemann surface, and let $(E,\nabla) \rightarrow M$ be a flat vector bundle, and consider the equation
\begin{equation}\label{eq: twisted harmonic intro}
 -\frac{1}{2} \star \nabla \star \Psi = c I
 \end{equation}
 for some constant $c$.  We call such a metric a {\em Poisson} metric due to the obviously analogy with harmonic functions and Poisson's equation.  In this case, one can still follow Simpson's correspondence to see that a flat, parabolic bundle $(E,\nabla,\Pi)$ with a tame, Poisson metric gives rise to a Higgs-type bundle but the complex structure will not be integrable unless $c=0$.  Our main theorem is
 \begin{thm}\label{thm: main thm}
 Let $(\overline{M}\,,\bar{g})$ be a compact Riemann surface, and let $M =\overline{M}\, \backslash \{p_{1},\dots,p_{j}\}$.  Let $g = e^{\psi} \bar{g}$ be a smooth K\"ahler metric on $M$ such that $e^{\psi} \in L^{1}(\overline{M},\bar{g})$.  Suppose that $(E,\nabla, \Pi) \rightarrow M$ is a flat vector bundle with a parabolic structure and regular singularities.  We say that $H$ is a Poisson metric on $E$ if $H$ satisfies~\eqref{eq: twisted harmonic intro}, with constant
\begin{equation*}
c = \frac{\deg(E, \Pi)}{rk(E) Vol(M,g)}.
\end{equation*}
Here  ${\rm deg}(E, \Pi)$ denotes the parabolic degree of $E$.  Then $(E,\nabla, \Pi)$ admits a conformally strongly tamed Poisson metric if and only if $E$ is slope polystable.  If $e^{\psi} \in L^{p}(\overline{M}, \bar{g})$ for some $p>2$, then $H$ is strongly tamed by the parabolic structure.  Moreover, any such metric is unique up to multiplication by a positive constant.
\end{thm}

We refer the reader to Defintion~\ref{def: tame} for the definition of a tame metric, which includes precise asymptotics.  The proof of this theorem occupies the majority of this paper.  We begin in Section~\ref{geom motiv} by discussing affine manifolds and the dimension reduction of the Hermitian-Einstein equation, arising from mirror symmetry.   When the base manifold is equipped with a Monge-Amp\`ere metric, we show that this equation is the same as the equation for Poisson metrics on flat vector bundles over $M$.  In Section~\ref{parabolic back} we discuss some basics of flat vector bundles with regular singularities on a punctured Riemann surface.  We introduce the notion of a parabolic structure, and define the key concept of a parabolic framing.  It is here that we give the definition of a tame hermitian metric. 

 In Section~\ref{sec: subbundles} we discuss flat subbundles, and stability.  We prove the Chern-Weil formula for flat subbundles, and deduce the necessity of stability for the existence of Poisson metrics. In Section~\ref{sec: model solutions} we construct explicit, tame, local solutions to the Poisson metric equation near the punctures.  These metrics are conformal twists of local harmonic metrics.  These local solutions are crucial in our later work.  We conclude Section~\ref{sec: model solutions} by establishing some useful formulae which will be used repeatedly in the proof of Theorem~\ref{thm: main thm}.  

Suppose $H$ is a Poisson metric on the punctured ball, which is bounded above and below by a multiple of the model solutions constructed in Section~\ref{sec: model solutions}.  Is it necessarily true that $H$ is tamed by the parabolic structure?  This is the question we address in Section~\ref{sec: regularity}. It turns out that the answer is yes.  In essence, this requires the proof of several a priori estimates for solutions of the Poisson metric equation.  These matters are complicated by the fact that the differential operators in question have strong singularities at the origin; in fact, the coefficients are only $L^{1}$, and no better.  In Section~\ref{sec: don heat} we introduce a heat flow on manifolds with boundary whose limit points are Poisson metrics on compact subsets of $M$ with prescribed boundary values.  Much of the discussion here follows work of Donaldson \cite{Don5}, and so the treatment here is somewhat brief.  The main result we need is that the flow always converges to a smooth solution of the Poisson metric equation.  

Finally, in Section~\ref{sec: main proof}, we give the proof of the main theorem.  The idea is the following.  Let $\mathcal{U}_{\rho} := \cup_{j} B_{\rho}(p_{j})$.  We solve the Poisson metric equation on $M\backslash \mathcal{U}_{\rho}$ with boundary values given by the local solutions constructed in Section~\ref{sec: model solutions}.  Let $H_{\rho}$ denote this solution, and let $h_\rho=H_0^{-1}H_\rho$, where $H_0$ equals our local model solution on a small fixed radius around each puncture.  One may wonder why we need boundary values given by the local solution. This fact crucially implies that Tr$(h_\rho)$ is subharmonic on ${\mathcal U}_R$, forcing the supremum to be a fixed distance $R$ away from each puncture. We then take a limit as $\rho \rightarrow 0$.  In order to establish convergence of the $H_{\rho}$ to a limit $H_{\infty}$, it suffices to establish a uniform upper bound for $h_{\rho}$.  To prove this estimate, we follow the ideas of Uhlenbeck-Yau \cite{UY}.  In particular, if no such upper bound exists, we construct a destabilizing subbundle.  As a result, stability implies a uniform upper bound and we can pass to the limit as $\rho \rightarrow 0$.  In the limit we obtain a Poisson metric $H_{\infty}$, smooth away from the punctures.  Using the results in Section~\ref{sec: regularity} we are able to deduce that $H_{\infty}$ is in fact tamed by the parabolic structure.

Finally, let us remark that our main theorem implies the existence result of Corlette-Simpson \cite[Theorem 6]{Simp2}, if one takes as the background metric the model constructed in Section~\ref{sec: model solutions}.  Interestingly, our techniques are quite different from the methods used by Simpson in the treatment of harmonic metrics \cite{Simp2}.  The main difference, as remarked above, is that when $c=0$ one can make use of the induced stable Higgs bundle structure to prove estimates.  In our setting, no such structure exists and we must develop new techniques to account for these issues.  Moreover, when $c=0$ the conformal invariance of the equation allows one to reduce to the case when the background metric is restricted from the closed Riemann surface $\overline{M}$.  In our case we must also take into account the singularities of the background metric which lead to several issues involving the ellipticity of the operators under consideration.  An advantage of our techniques is that they yield rather precise asymptotics for tame Poisson metrics near the punctures.

\vspace{\baselineskip}
\noindent
{\bf Acknowledgements} We would like to thank D.H. Phong, J. Loftin, R. Wentworth, and C.-C. Liu for helpful discussions and comments.  We would like to thank V. Tosatti for several helpful suggestions.  The second author is grateful to S.-C. Lau and J. Zhou for helpful conversations.

\section{Geometric motivation and Dimension Reduction}\label{geom motiv}

The primary objects of study in this paper are flat vector bundles over affine surfaces.  Recall the following definition.

\begin{defn}
An affine manifold is a real manifold $M$ admitting a flat, torsion-free connection $D$ on the tangent bundle $TM$.  
\end{defn} 
 
Affine manifolds occur in great abundance.  For example, a theorem of Gunning \cite{Gun} says that all Riemann surfaces admit affine structures.  Let $M$ be an affine manifold of dimension $N$.  It is well known \cite{Shi}, that $M$ is affine if and only if $M$ admits a covering by coordinate charts whose transition functions are affine transformations; we call such coordinates affine coordinates.  Fix a local affine coordinate system ${\bf x}:=(x^1,...,x^N)$, so that the flat, torsion free connection on $TM$ is given by the exterior derivative $d$. 

\begin{defn}
An affine manifold $M$ is called special affine, if it admits a covariant constant volume form $d\nu$.  $M$ is K\"ahler if $M$ admits a metric which in local affine coordinates is the Hessian of a smooth function $\phi$:
 \begin{equation*}
 \phi_{ij}dx^idx^j=\frac{\del^2\phi}{\pl x^i\pl x^j}dx^idx^j.
 \end{equation*}
\end{defn}

Suppose that $M$ is a compact, special affine, K\"ahler manifold.  It follows from the work of  Cheng-Yau \cite{CY} that $M$ admits a semi-flat metric, which in local affine coordinates solves the real Monge-Amp\`ere equation, det($\phi_{ij})=1$.  In the non-compact case the existence of a semi-flat metric is not guaranteed. Nevertheless, when $M=\mathbb{P}^1-\{p_1, \dots, p_{m}\}$ is the punctured Riemann sphere a semi-flat metric was first constructed by Greene-Shapere-Vafa-Yau in \cite{BSVY}. Many examples were later constructed by Loftin \cite{Loft1}, and Gross-Tosatti-Zhang \cite{GTZ}.
 
Given an affine manifold $M$, the tangent bundle $TM$ automatically inherits the structure of a complex manifold.  Explicitly, let ${\bf x}:=(x^1,...,x^N)$ be local affine coordinates and $({\bf x,y}):=(x^1,...,x^N,y^1,...,y^N)$ be the induced coordinates on $TM$.   If ${\bf z}={\bf x}+\sqrt{-1}{\bf y}$, then it is easy to check that the transition functions of local affine coordinates are holomorphic.

In what follows, we let $E$ be a flat, complex vector bundle of rank $n$ over $M$, and let $H$ be any smooth, Hermitian metric on $E$.  Denote by $p:TM\longrightarrow M$ the projection and let $\tilde{E} = p^{*}E$ be the pull-back vector bundle over the complex manifold $TM$.  Below, we will show that $(\tilde{E}, p^{*}H)$ is naturally equipped with the structure of a holomorphic vector bundle with a unitary connection.  One can then ask whether the pulled-back metric $p^{*}H$ solves the Hermitian-Yang-Mills equation on $TM$.  We will show that the Hermitian-Yang-Mills equation for $(\tilde{E}, p^{*}H)$ reduces to a system of non-linear equations on $M$.  

Let $G$ denote the gauge group of $E$ and let $K$ be the maximal compact subgroup preserving the metric $H$. The Lie algebra of $G$ splits as $Lie(G)\cong Lie (K)\oplus Lie(G/K)$, so the flat connection $\nabla$ splits point-wise as
\begin{equation*}
\nabla=d+A-\Psi,
\end{equation*}
where the connection $D_{A}:=d+A$ preserves the metric, and $\Psi$ is a self-adjoint, endomorphism valued 1-form.   The flat connection $\nabla$ induces a holomorphic structure $\bar\pl_\nabla$ on $\ti E$, defined by
\begin{equation*}
\bar\pl_\nabla=\bar\pl+\frac12(A_j-\Psi_j)\,d\bar z^j,
\end{equation*}
where $\bf z = \bf x + \sqrt{-1}\bf y$ are holomorphic coordinates on $TM$ constructed above.  Clearly $\nabla^2=0$ if and only if $\bar\pl_\nabla^2=0$, in other words $\bar\pl_\nabla$ is holomorphic if and only if $\nabla$ is flat.  

Given a flat connection $\nabla$ and a metric $H$, we can also define a ``dual" flat connection as follows.
 \begin{defn}\label{defn: ass conn}
 Let $(E,\nabla)$ be a flat vector bundle, equipped with a smooth Hermitian metric $H$.  Recall that the metric $H$ allows us to write $\nabla =  d+ A-\Psi$.  Then we define the associated connection $\hat{\nabla}^{H}$ by
 \begin{equation*}
 \hat{\nabla}^{H} = d + A + \Psi .
 \end{equation*}
 \end{defn}
 
 \begin{lem}
 The connection $\nabla$ is flat if and only if $\hat\nabla^H$ is flat. 
 \end{lem}
 \begin{proof}
 Let $s$ and $t$ denote sections of $E$, and let $\langle\cdot,\cdot\rangle$ denote the inner product with respect to $H$. Because $D_A$ preserves the metric we have:
 \be
 d\langle s,t\rangle=\langle D_As,t\rangle+\langle s,D_At\rangle.
 \ee
 Now, using the fact that $\Psi$ is self adjoint:
 \be
 \begin{aligned}
  d\langle s,t\rangle=&\langle D_As,t\rangle+\langle \Psi s,t\rangle+\langle s,D_At\rangle-\langle s,\Psi t\rangle \\
=&\langle \hat\nabla^H s,t\rangle+\langle s,\nabla t\rangle.
  \end{aligned}
\ee
  Applying $d$ to the above equality yields:
  \be
  \begin{aligned}
  0=d^2\langle s,t\rangle=&\langle (\hat\nabla^H)^2s,t\rangle-\langle \hat\nabla^H s,\nabla t\rangle+\langle \hat\nabla^H s,\nabla t\rangle+\langle s,(\nabla)^2 t\rangle \\
=&\langle (\hat\nabla^H)^2 s,t\rangle+\langle s,(\nabla)^2 t\rangle,
  \end{aligned} 
 \ee
 where the minus sign above was introduced by sending the exterior derivative over a one form. This completes the proof of the lemma. 
 \end{proof}

 Because there is a holomorphic structure on the pulled back bundle $\tilde{E}$, one can define the unitary Chern connection with respect to the pulled back metric $p^{*}H$. It is a simple computation to check that the $(1,0)$ part of this connection can be expressed in holomorphic coordinates $\bf z$ as
\begin{equation*}
\pl_\nabla=\pl+\frac12(A_j+\Psi_j) dz^j.
\end{equation*}
Thus the above lemma simply corresponds with the well known fact that for a unitary Chern connection on a holomorphic bundle, both the $(0,1)$ and $(1,0)$ components of the connection are integrable.

This correspondence goes the other way as well. If $\tilde{E}$ admits a unitary Chern connection which is constant long the fibers of the projection $p$, then one can similarly define two flat connections on $E$; one which corresponds to the $(1,0)$ part of the connection, and the other which corresponds to the $(0,1)$ part.

Let $g_{j\bar k}$ be a Hermitian metric on $TM$. The bundle $(\ti E, \ti\nabla)$ satisfies the Hermitian-Yang-Mills equations when
\begin{numcases}{}
F^{2,0}_{\ti\nabla}=F^{0,2}_{\ti\nabla}=0 & \label{HYM1} \\
\Lambda F^{1,1}_{\ti\nabla}=cI. & \label{HYM2}
\end{numcases}
Here $\Lambda F^{1,1}_{\ti\nabla}$ denotes $g^{j\bar k} F_{\bar kj}$, and we have written the $(1,1)$-form $F^{1,1}_{\ti\nabla}$ as $F_{\bar kj}dz^j\wedge d\bar z^k$. If $\ti\nabla$ is constant along the fibers of the projection $p$, we would like to write these equations directly on the base manifold $M$. 
By the above correspondence, a flat connection $\nabla$ on $E$, together with a metric $H$, gives a unitary connection $\ti \nabla$ on the holomorphic bundle $\ti E$, which immediately satisfies \eqref{HYM1}. As previously mentioned, in terms of $E \rightarrow M$,  equation~\eqref{HYM1} is equivalent to the fact that both $\nabla$ and $\hat\nabla^H$ are flat. In affine coordinates we combine the two equations $(\nabla)^2=(\hat\nabla^H)^2=0$ to get the system:
\begin{equation}
\label{flat}
\begin{cases}
\dfrac{\pl}{\pl x^j}A_k-\dfrac{\pl}{\pl x^k}A_j+[A_j,A_k]+[\Psi_j,\Psi_k]=0\\\
{}\\\
\dfrac{\pl}{\pl x^j}\Psi_k+[A_j,\Psi_k]-\dfrac{\pl}{\pl x^k}\Psi_j-[A_k,\Psi_j]=0.  \end{cases}
\end{equation}
Recall the notation $D_{A}=d+A$ on $E$, and let $F_A$ be the curvature of this connection. Then the above can be expressed in a coordinate free manner as
\begin{equation}
\begin{cases}
 F_{A}+\Psi\wedge\Psi=0\\\
D_{A}\Psi=0 \end{cases}\nonumber
\end{equation}
giving the dimension reduction of equation~\eqref{HYM1}. The dimension reduction of \eqref{HYM2} requires the input of a metric on $TM$. Given the affine metric $\phi_{ij}$ on $M$, we define the following metric $g$ on $TM$:
\be
g=\phi_{ij}(dx^idx^j+dy^idy^j).
\ee
Now, writing the $(1,1)$ component of $F_{\ti\nabla}$ as $F_{\bar kj}dz^j\wedge d\bar z^k$, we have
\begin{equation*}
\begin{aligned}
F_{\bar kj} = &-\frac12\frac{\pl}{\pl\bar z^k}(A_j+\Psi_j) +\frac12\frac{\pl}{\pl z^j}(A_k-\Psi_k)\\&-\frac14(A_k-\Psi_k)(A_j+\Psi_j) +\frac14(A_j+\Psi_j)(A_k-\Psi_k).
\end{aligned}
\end{equation*}
Recall that $\frac{\pl}{\pl z^j}=\frac{1}{2}(\frac{\pl}{\pl x^j}-\sqrt{-1}\frac{\pl}{\pl y^j})$, and $A_j$, $\Psi_j$ are independent of the fibre coordinate $\bf y$. As a result the derivatives in $z$ reduce to derivatives in affine coordinates:
\begin{equation*}
\begin{aligned}
4F_{\bar kj}&=-\frac{\pl}{\pl x^k}(A_j+\Psi_j)+\frac{\pl}{\pl x^j}(A_k-\Psi_k)\\&-(A_k-\Psi_k)(A_j+\Psi_j)+(A_j+\Psi_j)(A_k-\Psi_k),
\end{aligned}
\end{equation*}
yielding an expression for $F_{\bar kj}$ defined only on $M$, which can be written as follows:
\begin{equation*}
4F_{\bar kj}=(F_A)_{kj}-[\Psi_j,\Psi_k]-D_{A,k}\Psi_j-D_{A,j}\Psi_k.
\end{equation*}
Equations \eqref{flat} imply that $(F_A)_{jk}=-[\Psi_j,\Psi_k]$ and $D_j\Psi_k=D_k\Psi_j$, and so
\begin{equation*}
F_{\bar kj}=-\frac{1}{2}\left(\frac{\pl}{\pl x^k}\Psi_j+[A_k,\Psi_j]-[\Psi_k,\Psi_j]\right)=-\frac{1}{2}\nabla_k\Psi_j.
\end{equation*}
It follows that the dimension reduction of \eqref{HYM2} is given in affine coordinates by
\begin{equation}
\label{AHYM affine}
K:=-\frac{1}{2}\phi^{jk}\nabla_k\Psi_j=c\mathbb{I}.
\end{equation}
The endomorphism $K$ of $E$, defined above, is the analogue of the trace of the curvature $\Lambda F_{\ti \nabla}$.

While the above formulae conveniently express the analogy between the Hermitian-Einstein equation on $\ti E\rightarrow TM$, and the dimension reduction to $M$, they are somewhat inconvenient.  For instance, when $M= \mathbb{P}^{1}-\{p_{1},\dots,p_{m}\}$, and $(E,\nabla)$ is a flat complex vector bundle on $M$, neither affine coordinates, nor flat frames exist in a full neighborhood of the punctures.  More precisely, both the flat torsion free connection $D$ on $TM$, and $\nabla$ on $E$ have monodromy around the punctures.  For this reason, it is useful to have formulae for the quantities above which are independent of the frame, and coordinate system.  

\begin{lem}\label{lem: general formulas}
Let $(E,\nabla)$ be a flat vector bundle, and fix a Hermitian metric $H$.  Suppose that in a given frame the flat connection can be expressed $\nabla = d + \Gamma$.  Then in this frame we have
\begin{equation}\label{eq: nabla H formula}
\hat{\nabla}^{H} =d+ H^{-1}dH - \Gamma^{\dagger_{H}}
\end{equation}
where $\dagger_{H}$ denotes the adjoint of $\Gamma$ with respect to $H$.  In particular, we have 
\begin{equation}\label{eq: psi formula}
\Psi(H) = \frac{1}{2}(\hat{\nabla}^{H} - \nabla) = \frac{1}{2}H^{-1}dH -\frac{1}{2}(\Gamma + \Gamma^{\dagger_{H}}).
\end{equation}
\end{lem}
This follows from a straightforward computation, and so we omit the proof.  In a flat frame $\Gamma=0$, and so
\begin{equation}
\label{Psiflat}
\Psi=\frac12H^{-1}dH,
\end{equation}
which can also be derived by pulling back to $\ti E$ and working in a holomorphic frame. Combining the above expression with the definition of $K$, we see that in a flat frame for $E$, and in affine coordinates on $M$,
\begin{equation}
\label{Kflat}
K=-\frac{1}{4}\phi^{ij}\frac{\pl}{\pl x^j}\left(H^{-1}\frac{\pl}{\pl x^i} H\right).
\end{equation}
The following Lemma gives an invariant expression for the curvature $K(H)$.

\begin{lem}\label{coordinate free}
Let $(E,\nabla)$ be a flat vector bundle, and fix a Hermitian metric $H$, then we have
\begin{equation*}
K(H) = \frac{-1}{2\sqrt{\det(\phi_{pq})}} \nabla_{i}\left(\sqrt{\det(\phi_{pq})} \phi^{ij}\Psi_{j}\right)
\end{equation*}
or, equivalently,
\begin{equation*}
K(H) = -\frac{1}{2} \star \nabla \star \Psi(H)
\end{equation*}
where $\star$ denotes the Hodge star operator of the metric $\phi_{ij}$ on $TM$.  
\end{lem}

\begin{proof}
Working in affine coordinates for $M$, recall that in affine metric $\phi_{ij}$ solves the real Monge-Amp\`ere equation det$(\phi_{ij})=1$. As considered in \cite{CY}, the volume form defined by det$(\phi_{ij})dx^1\wedge\cdots\wedge dx^N=dx^1\wedge \cdots\wedge dx^N$ is invariant under affine coordinate change, and is $d$ invariant as well. We denote this volume form by $d\nu$. As a first step, since $\phi_{ij}$ solves the real Monge-Amp\`ere equation, we conclude $\frac{\pl}{\pl x^j}(\phi^{ij})=0$. To see this, compute
\begin{equation*}
\begin{aligned}
\pl_{ j}(\phi^{ij})&=\frac{\pl}{\pl x^j}(\phi^{ij}{\rm det}(\phi_{\ell m}))\\&=  -\phi^{ip}\frac{\pl}{\pl x^j}\phi_{pq}\phi^{qj}{\rm det}(\phi_{\ell m}) + \phi^{ij}\phi^{pq}\frac{\pl}{\pl x^j}\phi_{qp}{\rm det}(\phi_{\ell m})\\
&=-\phi^{ip}\frac{\pl}{\pl x^j}\phi_{pq}\phi^{qj} + \phi^{ij}\phi^{pq}\frac{\pl}{\pl x^j}\phi_{qp}.
\end{aligned}
\end{equation*}
Changing indices and using the affine K\"ahler condition, it follows that the right hand side of the above equality vanishes. Thus our original equation for $K$ can be rewritten as
\begin{equation*}
K=-\frac{1}{2}\nabla_k\left(\phi^{jk}\Psi_j\right).
\end{equation*}

The next step is to compute the $\star$ operator with respect to the volume form $\nu$. We use the following system of equations
\be
dx^i\wedge\star dx^j= \langle dx^i,dx^j\rangle_\phi d\nu=\phi^{ij}dx^1\wedge\cdots \wedge dx^N
\ee
to conclude
\begin{equation*}
\star dx^i=\sum_j\phi^{ji} \,dx^1\wedge\cdots \wedge dx^{j-1}\wedge dx^{j+1}\wedge\cdots\wedge dx^N \epsilon_{j1\cdots (j-1)(j+1)\cdots N},
\end{equation*}
where $\epsilon$ stands for the Levi-Civita Symbol. The endomorphism valued one form $\Psi$ can be written as $\Psi_idx^i$.  Taking $\star$ gives
\be
\star\Psi=\sum_{ij}\phi^{ji}\Psi_i \,dx^1\wedge\cdots \wedge dx^{j-1}\wedge dx^{j+1}\wedge\cdots \wedge dx^N \epsilon_{j1\cdots (j-1)(j+1)\cdots N}.
\ee
It follows that
\be
\nabla \star\Psi=\nabla_j\left(\phi^{ji}\Psi_i\right) dx^1\wedge \cdots \wedge dx^N.
\ee
Taking an additional $\star$ to clear the volume form and multiplying by $-\frac12$ proves the lemma.

\end{proof}
In particular, when the background metric is of Monge-Amp\`ere type, the affine Hermitian-Yang-Mills equation is given by
\begin{equation}\label{AHYM}
K(H):= -\frac{1}{2} \star \nabla \star \Psi(H) = c \mathbb{I},
\end{equation}
for a constant $c$ which will be determined in the next section.

Let us briefly discuss the relationship with harmonic bundles.  Let $(E,\nabla)$ be a flat, complex vector bundle over a Riemann surface $(M,g)$, and let $H$ be a hermitian metric of $E$.  Again we get a splitting of the connection as $\nabla = d + A - \Psi$, where the connection $D_{A} := d +A$ is unitary with respect to $H$, and $\Psi$ is a self-adjoint endomorphism valued one form on $X$.  Following \cite{Simp2, Simp3, Goth}, we say that the metric $H$ is harmonic if
\begin{equation*}
\star D_{A} \star \Psi =0.
\end{equation*}
This is precisely the condition that $H$ gives rise to an equivariant harmonic map into the group $Gl(n)/U(n)$.  Moreover, it is an easy exercise to check that this equation is equivalent to~\eqref{AHYM} when $c=0$.  As a result, we will call equation~\eqref{AHYM} the Poisson metric equation.

 Throughout the paper we will need a formula comparing the curvatures of two metrics. 
\begin{lem}\label{lem: endo equation}
Suppose that $H_{0}$ and $H_{1}$ are hermitian metrics on $E$.  Then the hermitian endormorphism $h = H_{0}^{-1}H_{1}$ satisfies
\begin{equation*}
\frac{-1}{4\sqrt{\det(\phi_{pq})}} \nabla_{i}\left(\sqrt{\det(\phi_{pq})} \phi^{ij}\,h^{-1}\hat{\nabla}^{0}_{j}h\right) =K_{1} - K_{0}
\end{equation*}
or equivalently
\begin{equation*}
-\frac14\star\nabla \star (h^{-1} \hat{\nabla}^{0} h) = K_{1}-K_0.
\end{equation*}

\end{lem}
\begin{proof}
Working locally in a flat frame for $E$, we have $\hat\nabla^0=d+H_0^{-1}dH_0$, $\Psi_0=\frac12H_0^{-1}dH_0$, and $\Psi_1=\frac12H_1^{-1}dH_1$. Now, just as in the K\"ahler case, an easy computation shows
\begin{equation}\label{eq: psi relation}
\frac12h^{-1}\hat\nabla^0h=\Psi_1-\Psi_0.
\end{equation}
Applying Lemma \ref{coordinate free} completes the proof. 

\end{proof}

Much of the discussion so far is true for arbitrary affine manifolds. For the rest of the paper, we will consider the case of a punctured Riemann surface.  Let $\overline{M}$ be a compact Riemann surface, with marked points $\{p_{1} ,\dots, p_{j}\}$.  We take $M =\overline{M}\,\backslash \{p_{1} ,\dots, p_{j}\}$.  Let us discuss briefly our choice of metrics.   Near each marked point $p_{\ell}$ we fix a a small coordinate patch containing $p_{\ell}$ and disjoint from the other punctures.    Let $\bar{g}$ be any K\"ahler metric on $\overline{M}$ which agrees with the Euclidean metric in each of these coordinate patches.  Fix a hermitian metric $\phi$ on $M$.  Since $M$ has complex dimension $1$, we have $\phi = e^{\psi} \bar{g}$ for some function $\psi$ which is smooth on $M$.  We shall assume that $(M,g)$ has finite volume with respect to the natural Riemannian volume form, which we shall denote by $d\nu$;  that is, $e^{\psi} \in L^{1}(\overline{M}, \bar{g})$.  Furthermore the Laplace-Beltrami operator
\be
\Delta_{\bar g}(\cdot)=\frac1{\sqrt{{\rm det}\bar g}}\pl_i\left(\bar g^{ij}\sqrt{{\rm det}\bar g}\,\pl_j(\cdot)\right).
\ee
is conformal; that is $\Delta_\phi=e^{-\psi}\Delta_{\bar g}$. This leads to the useful fact that integrals of the Laplacian, as well as the integral of the norm of a one form $\beta$, are invariant under choice of background metric, i.e.:
\be
\Delta_\phi d\nu=\Delta_{\bar g} dV_{\bar g}\qquad{\rm and}\qquad |\beta|^2_\phi\, d\nu=|\beta|^2_{\bar g}\, dV_{\bar g}.
\ee
Thus, although we will adorn norms throughout the paper for clarity, in the situation described above we may switch to another conformal norm or drop the subscript for notational simplicity.  When we refer to metric balls, we will always mean metric balls with respect to the smooth metric $\bar{g}$ on $\overline{M}$.  Finally, when working in a coordinate patch near each puncture, we let $\Delta$ and $dV$ denote the Laplacian and volume form with respect to the Euclidean metric.  If the reader is not interested in the most general statements possible, then it is most convenient to assume that $e^{\psi} \in L^{p}(\overline{M}, \bar{g})$ for some $p>2$.  In fact, this case is quite interesting, since  Loftin's Monge-Amp\`ere type metrics on the punctured Riemann sphere are in fact  $L^{p}$ for any $p>2$ by \cite[Theorem 4]{Loft1}.
\section{Meromorphic Bundles, Parabolic framings, and Degree}\label{parabolic back}

Once again we consider vector bundles over $M =\overline{M}\,\backslash \{p_{1} ,\dots, p_{j}\}$, where $\overline{M}$ is a compact Riemann surface. Since the base has dimension $1$, any such vector bundle admits a holomorphic structure.  However, these bundles are singular near the punctures.  As a result, we need a formalism for discussing vector bundles and connections with singularities.  We will use the language of Deligne \cite{Del}, see also \cite{Sab}.

More generally, let $M$ be a complex manifold, and let $Z \subset M$ be a smooth complex hypersurface in $M$.  Let $\mathcal{O}_{M}(\ast Z)$ be the sheaf of meromorphic functions with poles on $Z$.

\begin{defn}
A \emph{meromorphic bundle} on $M$ with poles on $Z$ is a locally free sheaf of $\mathcal{O}_{M}(\ast Z)$-modules of finite rank.  A \emph{lattice} of this meromorphic bundle is a locally free $\mathcal{O}_{M}$ submodule of this meromorphic bundle, which has the same rank.
\end{defn}

In particular, if $\mathcal{E}$ is a lattice of the meromorphic bundle $\mathcal{M}$, then $\mathcal{E}$ defines a vector bundle on all of $M$ which coincides with $\mathcal{M}$ when restricted to $M\backslash Z$.  Moreover, we have
\begin{equation*}
\mathcal{M}  = \mathcal{O}_{M}(\ast Z) \otimes_{\mathcal{O}_{M}} \mathcal{E}
\end{equation*}
It is not clear that a meromorphic bundle over a complex manifold necessarily admits a lattice.  Nevertheless, in the case of a Riemann surface we have
\begin{prop}[\cite{Sab}, Proposition 0.8.4]
Let $M$ be a Riemann surface and let $Z\subset M$ be a discrete set of points.  Then any meromorphic bundle on $M$ with poles at the points of $Z$ contains at least one lattice.
\end{prop}

The sheaf of meromorphic differential $k$-forms on $M$ with poles on $Z$ is defined to be $\Omega^{k}_{M}(\ast Z) := \mathcal{O}_{M}(\ast Z) \otimes_{\mathcal{O}_{M}} \Omega^{k}_{M}$.  For our purposes, we will be interested in 1-forms with logarithmic poles along $Z$.

\begin{defn}
Let $p\in M$ be a point, and choose a local coordinate $z$ so that $z(p)=0$.  We say the a meromorphic 1-form $\omega = \phi dz$ is \emph{logarithmic} if 
\begin{equation*}
\Omega = \psi \frac{dz}{z}
\end{equation*}
for $\psi$ a holomorphic function.  We denote by $\Omega^{1}_{M}\langle \log Z \rangle$ the sheaf of logarithmic differential 1-forms.
\end{defn}

Let $\mathcal{M}$ be a meromorphic bundle on $M$.  As usual, a connection on $\mathcal{M}$ is a $\mathbb{C}$-linear homomorphism $\nabla: \mathcal{M} \rightarrow \Omega^{1}_{M}\otimes \mathcal{M}$ satisfying the Leibniz rule.    In a local basis of $\mathcal{M}$ over $\mathcal{O}_{M}(\star Z)$, the connection is written as $\nabla = d + \Omega$, and the matrix valued 1-form $\Omega$  has entries in $\Omega^{1}_{M}(\star Z)$.  Note that if $\mathcal{E}$ is a lattice of a meromorphic bundle with connection $(\mathcal{M},\nabla)$, it can happen that $\nabla(\mathcal{E})$ is {\em not} contained in $\Omega^{1}_{M} \otimes \mathcal{E}$.  Nevertheless, it is the case that $\nabla(\mathcal{E}) \subset \Omega^{1}_{M}(\star Z) \otimes \mathcal{E}$.  Therefore, $\nabla$ defines a meromorphic connection, which is not necessarily holomorphic, on the bundle $\mathcal{E}$.  We will say that $\mathcal{E}$ is a \emph{logarithmic lattice} of $(\mathcal{M}, \nabla)$ if
\begin{equation*}
\nabla(\mathcal{E}) \subset \Omega^{1}_{M}\langle \log Z\rangle \otimes \mathcal{E}
\end{equation*}

\begin{defn}
We say that a meromorphic bundle with a $\emph{flat}$ connection $(\mathcal{M},\nabla)$ has a regular singularity along $Z$ if, in a neighborhood $U$ of any point of $Z$, there exists a logarithmic lattice of $\mathcal{M}|_{U}$. 
\end{defn}

Let $\mathcal{D}$ denote the disk of radius one around $0 \in \mathbb{C}$. We have the following key theorem.

\begin{thm}[\cite{Sab}, Theorem 2.2.8]\label{thm: canonical form}
Let $(\mathcal{M},\nabla)$ be a meromorphic bundle with a flat connection $\nabla$ on $\mathcal{D}$, equipped with a coordinate $z$.  Let $\mathfrak{K} = \mathbb{C}\{z\}[z^{-1}]$ denote the field of convergent Laurent series with poles at $0$.  Then there exists a matrix $P \in GL_{d}(\mathfrak{K})$ such that, after changing gauge by the matrix $P$, the connection takes the form
\begin{equation*}
\nabla = d -\sqrt{-1} B_{0} \frac{dz}{z}
\end{equation*}
where $B_{0} \in M_{d}(\mathbb{C})$ is constant and in Jordan normal form.
\end{thm}
Note that the gauge given by the above Theorem gives rise naturally to a logarithmic lattice of the bundle $(\mathcal{M}, \nabla)$.  The matrix $-\sqrt{-1}B_{0}$ is called the residue of the connection $\nabla$;  see \cite{Sab} for a more intrinsic definition. 
\begin{defn}\label{rk: temporal}
It  follows easily from Theorem~\ref{thm: canonical form} that there is a frame in which the connection $\nabla$ can be written as $\nabla = d +B_{0}d\theta$, where $z=re^{i\theta}$ are polar coordinates on $\mathcal{D}$.  Moreover, by multiplying by the appropriate power of the coordinate $z$, we can ensure that if $\kappa$ is a generalized eigenvalue of $B_{0}$, then ${\rm Im}(\kappa) \in[0,1)$. Following Daskalopoulos-Wentworth \cite{DW}, we  call this the \emph{temporal framing}.
\end{defn}

As before, let $M=\overline{M}\backslash \{ p_{1},\dots, p_{m}\}$ be a punctured Riemann surface, and let $o \in M$ be a fixed point.  We assume that $m \geq 1$.  Let $\rho: \pi_{1}(M,o) \rightarrow GL(n,\mathbb{C})$ be a representation of the fundamental group with base point $o$.  It is well known that such a representation gives rise to a holomorphic vector bundle $E\rightarrow M$ with a flat connection such that the monodromy representation is precisely the $\rho$; see e.g. \cite[Theorem 0.15.8]{Sab}.  In fact, by the solution of the weak Riemann-Hilbert problem \cite[Corollary 2.3.2]{Sab}, there is a meromorphic bundle $\mathcal{M}\rightarrow\overline{M}$ with a flat connection $\nabla$ with regular singularities at the points $\{ p_{1},\dots,p_{m}\}$ such that $(\mathcal{M}, \nabla)|_M = (E, {\nabla})$. To ease notation from this point on we simply consider the flat bundle $(E,\nabla)$ over $M$.

The notion of a parabolic structure was introduced by Mehta-Seshadri \cite{MS} as a means to construct moduli of vector bundles on punctured Riemann surfaces.  Since their introduction, parabolic bundles have been the subject of much research, primarily due to their role in conformal field theory as elucidated by Witten \cite{Wit}.  A parabolic structure on a bundle $E$ with a unitary connection, as defined by Mehta-Seshadri, is nothing more than a complete flag of the fibre of $E$ at the punctures, together with a choice of weight for each subspace.  In our setting, the bundle $E$ is not assumed to have unitary monodromy, and so there is an additional requirement that the flags and weights be compatible with the monodromy representation.  Most of the definitions to follow appear also in \cite{Simp2}, and we refer the reader to this work for a slightly different presentation.

Fix a small ball $B_{r}(p_{i})$ around a puncture $p_i$.  By choosing local coordinates centered at $p_{i}$ we identify this with the unit disc $\mathcal{D}$.    Working in the temporal gauge, we have $\nabla = d+ B_{0} d\theta$ where $B_{0} \in M_{n}(\mathbb{C})$ is in Jordan normal form.  Each upper triangular block of the monodromy operator corresponds to a local indecomposable subbundle.  That is, the connection $\nabla$ decomposes $E$ into a direct sum of local, indecomposable subbundles
\begin{equation}\label{eq: local decomp}      
E = V_{1} \oplus V_{2} \oplus \cdots \oplus V_{k}.
\end{equation}
\begin{defn}
A parabolic structure on $(E ,\nabla)$ at $p_{i}$ is a choice of weights $w_{j} \in \mathbb{R}$ for $1\leq j \leq k$.  A parabolic structure $\Pi$ on $(E ,\nabla)$ is a choice of parabolic structure at each $p \in  \{ p_{1},\dots, p_{m}\}$.  We denote the flat vector bundle $(E,{\nabla})$ with a fixed parabolic structure by $(E,{\nabla},\Pi)$.
\end{defn}

Note that this data can easily be packaged in the form of a flag at each puncture $p_{i}$, by arranging the weights in increasing order $w_{1} \leq w_{2} \leq \cdots \leq w_{k}$, and defining a flag by taking $F^{\ell}E := \oplus_{i=1}^{\ell} V_{i}$.  We will see below that there is a further refinement of this flag which is relevant for the problem at hand.

Once a parabolic structure is fixed, we can define the degree of a flat bundle $(E,\nabla)$.  The parabolic structure also specifies a class of metrics $H$ on $E$ with optimal growth conditions near each puncture, which we shall see in the discussion to follow. 

\begin{defn}
Let $(E,\nabla, \Pi)$ be a flat vector bundle with a parabolic structure.  We define the degree of $E$ at $p_{i}$ to be
\begin{equation*}
\deg(E, \Pi, p_{i}) = \sum_{i=1}^{k} w_{i}\dim(V_{i}).
\end{equation*} 
We define the degree of $(E,\nabla, \Pi)$ to be
\begin{equation*}
\deg(E, \Pi) = \sum_{j=1}^{m} \deg(E, \Pi, p_{j}).
\end{equation*}
\end{defn}

Geometrically, the role of the parabolic structure is to fix the topology of the bundle $E$.  Let us consider a trivial example.  Consider $\mathcal{O}(a) \rightarrow \mathbb{P}^{1}$, and let $M = \mathbb{P}^{1} \backslash \{N, S\}$.  It is easy to see that $\mathcal{O}(a)|_{M}$ is isomorphic to the trivial bundle.  However, the isomorphism identifying $\mathcal{O}(2)$ and $\mathcal{O}(1)$ over $M$ has poles and zeros at the points $N,S$.  That is, the induced gauge transformation on the trivial bundle over $M$ is singular at the punctures.  The role of the parabolic structure is to reduce the gauge group by specifying the singularities of allowable gauge transformations, and in doing so, determine the topology of the bundle $E$.  This data is best contained by fixing a framing for $E$ in a neighborhood of each puncture which is compatible, in an appropriate sense, with the parabolic structure.

\begin{prop}[The Parabolic Framing]\label{prop: parabolic gauge} 
Fix a puncture $p_{j}$, and let $E = V_{1}\oplus \cdots\oplus V_{k}$ be the decomposition of $E$ into local indecomposable subbundles.  Let $w_{1}, \dots, w_{k}$ be the weights of the parabolic structure on $(E,\nabla)$ at $p_{j}$.  Then near $p_{j}$ there exists a frame for $(E,\nabla)$ so that
\begin{equation*}
\nabla = d + A \frac{dr}{r} + B d\theta
\end{equation*}
where $B$ is in Jordan normal form, with Jordan blocks $B = B_{1} \oplus \cdots \oplus B_{k}$, and $A = A_{1} \oplus \cdots \oplus A_{k}$ with
\begin{equation*}
B_{i} = \kappa_{i} \mathbb{I}_{j} + N_{i}, \qquad A_{i} = w_{i} \mathbb{I}_{i},
\end{equation*}
with ${\rm Im}(\kappa_{i}) \in [0,1)$.  In the above, $\mathbb{I}_{i}$ denotes the $\dim V_{i} \times \dim V_{i}$ identity matrix, and $N_{i}$ is the $\dim V_{i} \times \dim V_{i}$ nilpotent matrix with ones along the super-diagonal.
\end{prop}
\begin{proof}
The proof is almost trivial, given Theorem~\ref{thm: canonical form}.  We begin by working in the temporal gauge of Definition~\ref{rk: temporal}.  Let $w_{1}, \dots, w_{k}$ be the parabolic structure at $0 \in \Delta$.  Define a gauge transformation by setting
\begin{equation*}
g(r,\theta) = r^{w_{1}} \mathbb{I}_{V_{1}} \oplus \dots \oplus r^{w_{k}} \mathbb{I}_{V_{k}}.
\end{equation*}
Gauge transforming by $g$ yields the desired result.
\end{proof}
\begin{rk}\label{rk: parabolic framing}
The gauge defined above is {\em not} unique.  Indeed, gauge transforming by any matrix which commutes with both of $A, B$ yields another frame with the same properties.  As a result, it is important to fix one such frame in a neighbourhood of each puncture $p_{j}$.  We call this the {\em parabolic framing}.
\end{rk}

The flat connection induces a weight filtration on the local indecomposable subbundles $V_{j}$ of equation~\eqref{eq: local decomp} in the following way.  Fix a point $x \in \mathcal{D}$ and let $\mathcal{R}_{x}$ denote the ray from $0$ passing through $x$.  Fix a trivialization of $V_{j} \subset E$ over $\mathcal{R}_{x}$.  Let $V$ denote the fibre of $V_{j}$ over $x$.  Parallel transport around a closed loop induces the the restricted monodromy action $\mu:V \rightarrow V$, with a single generalized eigenvalue $\kappa_{j}$. Consider the nilpotent endomorphism $N: V \rightarrow V$ given by $N = \mu - \kappa_{j} \mathbb{I}$.  $N$ induces a complete flag of $V$ by
\begin{equation}\label{eq: nil filtration}
0 \subsetneq \ker N \subsetneq \ker N^{2} \subsetneq \cdots \subsetneq \ker N^{d_{j}} = V.
\end{equation}
In turn, this induces a grading on $V$ by imposing that $v \in V$ has weight $\tau_i=2i-(d_{j}+1)$ if $v \in \ker N^{i}$ but $v \notin \ker N^{i-1}$.  Equivalently, we require that multiplication by $N$ decreases weights by $2$, and that  the sum of the weights is zero.  It is important to point out that this grading is compatible with parallel transport in the following sense; if $y \in \mathcal{R}_{x}$ is any other point, and  $\sigma$ is a flat section of $V_{j}$ defined on $\mathcal{R}_{x}$, then $\sigma(x)$ has weight $\tau$ if an only if $\sigma(y)$ has weight $\tau$ for any $y \in \mathcal{R}_{x}$.

\begin{defn}
In the above setting, we call the weight $\tau$ of a section $\sigma(x)$ the \emph{ nilpotent weight} of $\sigma$ at $x$.
\end{defn}

The nilpotent weight filtration induces a complete flag of local subbundles of the local irreducible subbundle $V_{j}$
\begin{equation*}
\{0\} = V_{0,j}, \subsetneq V_{1,j} \subsetneq \cdots \subsetneq V_{d_{j},j} = V_{j}.
\end{equation*}
\begin{defn}\label{def: local flag}
We call each $V_{i,j} \subset V_{j}$ a \emph{local invariant subbundle} of $V_{j}$. 
\end{defn}
The nilpotent weight filtration refines the flag determined by the parabolic structure to define a complete flag of the fiber of $E$ at the puncture which, together with the weights of the parabolic structure, packages all of the data we will need.

Let us now give a very concrete picture of the the nilpotent weight filtration, as well as the local invariant subbundles, by working in the parabolic framing.  By considering each indecomposable subbundle 
$V_{j}$ individually, it suffices to consider the case when the connection is given in the parabolic framing by 
\begin{equation*}
\nabla = d + w_{j} \mathbb{I} \frac{dr}{r} + B_{j} d\theta
\end{equation*}
where $B_{j} = \kappa_{j} \mathbb{I} + N$ is a single Jordan block, with $N$ nilpotent.  The matrix $N$ induces a filtration of $\mathbb{C}^{d_{j}}$, in the same way as in~\eqref{eq: nil filtration}, again with weights assigned so that multiplication by $N$ decreases weights by $2$, and that  the sum of the weights is zero.  The nilpotent weight filtration on sections at each point $x\in \mathcal{D}$ is obtained by using the parabolic framing to identify the fiber of $V_{j}$ over $x$ with $\mathbb{C}^{d_{j}}$.

From this description it is clear that, up to taking direct sums, the local invariant subbundles are the only $\nabla$ invariant subbundles of $(E,\nabla)$ over the disk $\D$.

\begin{ex}\label{ex: 1}
Consider $M =\mathbb{C} \backslash \{0\} \simeq S^{2}\backslash \{N,S\}$ with polar coordinates $(r,\theta)$. Let $E =\mathbb{C}^{2}\rightarrow M$ be the trivial rank two bundle equipped with the standard basis
\begin{equation*}
e_{1} = \begin{bmatrix} 1\\0 \end{bmatrix} \quad e_{2} = \begin{bmatrix} 0 \\ 1 \end{bmatrix}.
\end{equation*}
We equip $E$ with a flat connection given by
\begin{equation*}
\nabla = d + \begin{pmatrix}1 & 1 \\ 0 & 1 \end{pmatrix} d\theta,
\end{equation*}
so that $E$ has no non-trivial indecomposable subbundles near $0$ or $\infty$.  Then a parabolic structure for $E$ is an assignment of a real number $w(0), w(\infty)$ to the points $N,S \in S^{2}$.  For simplicity, we assume $w(0)=w(\infty)=0$.  Then $e_{1}, e_{2}$ form a global parabolic gauge for $E$.  The weight filtration on $E$ is then easily determined to be
\begin{equation*}
0 \subset {\rm Span}\{ e_{1} \} \subset {\rm Span} \{e_{1},e_{2}\} = E,
\end{equation*}
and so $e_{1}$ has nilpotent weight $\tau_1=-1$ while $e_{2}$ has nilpotent weight $\tau_2=1$, and ${\rm Span}\{e_{1}\}$ is the only $\nabla$ invariant subbundle of $E$
\end{ex}

The parabolic structure, together with the nilpotent weights of the monodromy, determines the asymptotics of the Poisson metrics we consider in this paper.  Fix a small ball $B := B_{R}(p_{i})$ which is disjoint from the other $p_{j}, j\ne i$, and a local coordinate $z=re^{i\theta}$ on $\D \subset \mathbb{C}$.  

\begin{defn}\label{def: tame}
 With notation as above, fix the decomposition of $E$ into local indecomposable subbundles, as in equation~\eqref{eq: local decomp}.  Let $w_{\ell}$ be the weight assigned to $V_{\ell}$ by the parabolic structure and let $\kappa_{\ell}$ be the generalized eigenvalue of the residue of $\nabla|_{V_{\ell}}$.  Let $S$ be a $\nabla$ invariant subbundle of $E$ over $\D$, which we write as $S = S_{1} \oplus \cdots \oplus S_{k}$, where each $S_{\ell}$ is a local invariant subbundle of the local indecomposable $V_{\ell}$. We say that the metric $H$ is \emph{tamed by the parabolic structure} if the following two conditions hold at each puncture.
\begin{itemize}
\item[(A)] $|K(H)| \in L^{1}(\D, d\nu)$
\smallskip

\item[(B)] Let $H_{S}$ denote the induced metric on $S$, and $\Psi^{S} := \Psi(H_{S})$ be the induced $End(S)$ valued $1$-form.  Then there is an $\epsilon >0$ such that 
\begin{equation*}
\Tr\left( \Psi^{S}(\ddr) \right) =- \sum_{\ell=1}^{k}\frac{w_{\ell}}{r} {\rm rk}(S_{\ell}) + o\left(\frac{1}{r|\log(r)|^{\epsilon}}\right).
\end{equation*}
\end{itemize}
 We say that the metric $H$ is  \emph{strongly tamed by the parabolic structure} if, in addition, the following conditions hold:
\begin{itemize}
\item[(C)]
Let $\sigma$ be a flat section of $V_{\ell}$ defined over the ray $\theta = \theta_{0}$, and let $\tau$ be the nilpotent weight of $\sigma$.  Then, there is a constant $C>0$ such that, for $r \ll 1$
\begin{equation*}
\label{eq: asymptotics}
 C^{-1} r^{-w_{\ell}}|\log (r)|^{\frac{\tau}{2}}\leq \frac{|\sigma|_{H}(r,\theta_{0})}{|\sigma|_{H}(1,\theta_{0})} \leq C r^{-w_{\ell}}|\log (r)|^{\frac{\tau}{2}}.
 \end{equation*}
\item[(D)]
There exists  $\epsilon >0$ so that, in any frame compatible with the decomposition of $E$ into local indecomposables, we have
\begin{equation*}
\Psi =  -\bigoplus_{\ell=1}^{k}\left(w_{\ell}\frac{dr}{r} +{\rm Re}(\kappa_{\ell}) d\theta\right) \mathbb{I}_{V_{\ell}} + o\left(\frac{dr}{r|\log r|^{\epsilon}} + \frac{d\theta}{|\log r|^{\epsilon}}\right)
\end{equation*}

\end{itemize}
Finally, we say that $H$ is \emph{conformally tamed} (resp. \emph{ conformally strongly tamed}) if there exists a function $u \in C^{\infty}(M, \mathbb{R}) \cap L^{2}(\overline{M}, dV)$ solving $\Delta_{\bar{g}} u =f$ in the distributional sense on $(\overline{M},\bar{g})$, for a function $f \in C^{\infty}(M, \mathbb{R})\cap L^{1}(\overline{M}, dV)$, such that $e^{-u} H$ is tamed (resp. strongly tamed) by the parabolic structure.
\end{defn}

 Let us make some remarks about this definition.  Condition (B) is essentially the same as the condition for tameness imposed by Simpson \cite{Simp2}, and is by far the most important.  Indeed, if we are interested only in Poisson metrics, then condition (A) can be discarded immediately.  Moreover, condition (C) essentially follows from stability, and we impose this condition only for convenience, and to shorten the statements of the theorems throughout the paper.  Similarly for   condition (D), which implies condition (B), but is less convenient to work with.  The condition of conformally tame is important only in the case that the conformal factor of the metric $g = e^{\psi} \bar{g}$ has $e^{\psi} \notin L^{p}(M, dV)$ for any $ p > 2$;  if the reader is interested only in this case, then the modifier ``conformally" can be removed from the paper.   Note that if $H$ is tamed by the parabolic structure, and $S \subset E$ is a flat subbundle, it is not clear that the induced metric $H_{S}$ is tame.  Nevertheless, this is true, as we will prove in Proposition~\ref{prop: sub bun deg}, below.

Let us give an example of a strongly tame metric, which will be of use to us later.

\begin{ex}\label{ex: 2}
Continuing with Example~\ref{ex: 1}, we define a hermitian metric on $E \rightarrow \D$ near $0 \in \mathbb{C}$.  Set 
\begin{equation*}
 H_{0} = \begin{pmatrix} \frac{1}{-\log r} & 0 \\ 0 & -\log r \end{pmatrix},
 \end{equation*}
 and similarly a local metric $H_{\infty}$ near $\infty$.  Then any smooth hermitian metric on $E \rightarrow M$ which agrees with $H_{0}$ near zero, and $H_{\infty}$ near $\infty$ is strongly tamed by the parabolic structure.
 \end{ex}

In the following proposition we show that if $H$ is a metric which is tamed by the parabolic structure, then the degree is computed by the integral of the trace of the curvature $K$, defined in equation~\eqref{AHYM}.

\begin{prop}\label{prop: alg deg = an deg}
Let $(E,\nabla,\Pi)$ be a flat vector bundle over $M$ with a parabolic structure.  Let $H$ be a conformally tame metric on $E$.  Then we have
\begin{equation*}
\deg(E,\Pi)=\frac{1}{\pi}\int_M{\rm Tr}(K(H))\,d\nu.
\end{equation*}
\end{prop}

\begin{proof}
Write $H = e^{u}\hat{H}$, and assume that $\hat{H}$ satisfies only condition (B) of Definition~\ref{def: tame}. We apply Lemma \ref{coordinate free} to conclude  
\begin{equation*}
{\rm Tr}(K)\,d\nu=-\frac{1}{2}d\star{\rm Tr}(\Psi(\hat{H})) -\frac{1}{4} \Delta_{\bar{g}}u dV.
\end{equation*}
For $\rho \ll1$, we define $M_\rho:=M - \bigcup_kB_{\rho}(p_k)$.  By Stokes' Theorem, we compute
\begin{equation*}
\begin{aligned}
\lim_{\rho\rightarrow 0}\int_{M_\rho}{\rm Tr}(K)\,d\nu&=-\frac{1}{2}\lim_{\rho\rightarrow 0}\int_{M_\rho}d\star{\rm Tr}(\Psi(\hat{H})) \\
&= -\frac{1}{2}\lim_{\rho\rightarrow 0}\int_{\pl M_\rho}\star{\rm Tr}(\Psi(\hat{H})),
\end{aligned}
\end{equation*}
where we have used that $\Delta_{\bar{g}}u \in L^{1}(M, dV)$, so that $\int_{M} \Delta_{\bar{g}}u dV =0$.

In order to apply condition (B) of Definition \ref{def: tame}, we express the above integral in polar coordinates. As a first step we write the metric $\phi_{ij}$ in polar coordinates. By assumption, the metric is written in complex coordinates $z$ is given by $e^{\psi}\,dzd\bar z$. Switching to polar coordinates the metric $\phi_{ij} = e^{\psi}(dr^{2} + r^{2}d\theta^{2})$, and so one easily computes

\begin{equation}\label{eq: hodge star}
\star dr=rd\theta ,\qquad \star d\theta=-(1/r)\,dr.
\end{equation}

Let us suppress the dependence on $\hat{H}$.  Write the one form $\Tr(\Psi)$ as $\Tr(\Psi_r )dr+\Tr(\Psi_\t) d\theta$. For small enough $\rho$ the set $\pl M_\rho$ is just the union of boundaries of balls $\bigcup_j\pl B_{\rho}(p_j)$. We restrict the computation to one such ball for simplicity. The restriction of the one form $dr$ to $\pl B_\rho(p_j)$ is zero, and so
\begin{equation*}
-\frac12\lim_{\rho\rightarrow 0}\int_{\pl B_{\rho}(p_j)}\star{\rm Tr}(\Psi)=-\lim_{\rho\rightarrow 0}\frac{\rho}2\int_0^{2\pi}{\rm Tr}(\Psi_r)\,d\t.\nonumber
\end{equation*}

By Definition~\ref{def: tame} we have
\begin{equation*}
-\frac{1}{2}\lim_{\rho\rightarrow 0}\rho\int_0^{2\pi}{\rm Tr}(\Psi_r)\,d\t=
\pi\sum_id_i w_i-\frac{1}{2}\lim_{\rho\rightarrow 0}\rho\int_0^{2\pi}o\left(\frac{1}{\rho|\log(\rho)|^\epsilon}\right)\,d\t.
\end{equation*}
and the second term on the right vanishes, and so $\lim_{\rho \rightarrow 0} \int_{M_{\rho}} \Tr(K) d\nu = \pi \deg(E,\Pi).$
We now apply condition (A) of Definition~\ref{def: tame} to conclude that
\begin{equation*}
\int_{M} \Tr(K) d\nu = \lim_{\rho \rightarrow 0} \int_{M_{\rho}} \Tr(K) d\nu.
\end{equation*}
\end{proof}

\begin{cor}\label{cor: c def}
If $H$ is a smooth metric on $(E,\Pi, \nabla)$ conformally tamed by the parabolic structure, and solving equation \eqref{AHYM}, then the constant $c$ appearing on the right hand side of equation \eqref{AHYM} is
\begin{equation*}
c=\frac{2\pi \deg(E, \Pi)}{{\rm rk}(E){\rm Vol}(M, d\nu)}.\nonumber
\end{equation*}
\end{cor}

\section{Stability, Subbundles, and the Chern-Weil Formula}\label{sec: subbundles}

The existence of a Poisson metric on the flat bundle $(E,\nabla)$ is intimately tied to an algebraic notion of stability, as in the case of holomorphic bundles over compact K\"ahler manifolds.  In this section we introduce a notion of stability, which is the analog of Mumford-Takemoto stability, and prove that any flat bundle admitting an Poisson metric is necessarily stable. 

\begin{defn}
Let $S \subset E$ be a subbundle of $E$.  We say that $S$ is a flat subbundle of $(E,\nabla)$ if $\nabla(S) \subseteq S$; that is, the flat connection $\nabla$ preserves $S$.
\end{defn}

If $S$ is a flat subbundle of $(E,\nabla)$ then it is clear that near any puncture $p_{j}$, $S$ is a direct summand of local invariant subbundles of $E$; see Definition~\ref{def: local flag}.  In particular, a parabolic structure $\Pi$ on $(E,\nabla)$ induces a parabolic structure $\Pi|_{S}$ on the flat bundle $(S,\nabla)$ by restriction.

\begin{defn}
If $S$ is a flat subbundle of $(E,\nabla, \Pi)$ we define
\begin{equation*}
\deg(S) := \deg (S, \Pi|_{S}).
\end{equation*}
\end{defn}

Finally, we are brought to the relevant notion of stability;
\begin{defn}
We say that $(E,\nabla, \Pi)$ is slope stable if, for any flat subbundle $S\subset E$ we have
\begin{equation*}
\mu(S):=\frac{\deg(S)}{{\rm rk}(S)} < \frac{\deg(E)}{{\rm rk(E)}} =: \mu(E).
\end{equation*}
We say that $E$ is semi-stable if $\mu(S) \leq \mu(E)$, and polystable if $E$ is a direct sum of stable bundles of the same slope.
\end{defn}

This is clearly the analog of the familiar notion of Mumford-Takemoto stability \cite{NS, Don1, UY}.  Note that in the present setting it suffices to consider only subbundles, rather that coherent, torsion-free subsheaves. This is one major simplification that arises in working with flat bundles over affine manifolds as opposed to holomorphic bundles over K\"ahler manifolds.   


If $(E,\nabla, \Pi)$ is equipped with a hermitian metric $H$ which is tamed by the parabolic structure, it is not necessarily true that the associated connection $\hat{\nabla}^{H}$ of Definition~\ref{defn: ass conn} preserves the flat subbundle $S$.  We introduce a second fundamental form quantity which measures this defect.  Let $\pi$ be the orthogonal projection from $E$ to $S$ with respect to the tame metric $H$.  We regard $\pi$ as a $H$-self-adjoint section of ${\rm Hom}(E,E)$.  Define a section of ${\rm Hom}(E,E)\otimes TM^{*}$ by
\begin{equation*}
\beta = (\mathbb{I}-\pi)\circ \hat{\nabla}^{H}( \pi) \circ \pi.
\end{equation*}
It is clear that $\beta$ measures the failure of $\hat{\nabla}^{H}$ to preserve $S$.  In fact, $\beta$ is a component of the endomorphism valued $1$-form $\Psi^{H}$.  Recall that $\Psi^{H}= \frac{1}{2}(\hat{\nabla}^{H} - \nabla)$.  Then, since $S$ is flat
\begin{equation*}
(\mathbb{I} - \pi) \Psi^{H} \pi = \frac{1}{2}\beta.
\end{equation*}
The metric $H$ restricts to $S$ to give a hermitian metric $H|_{S}$.  Then we have
\begin{equation*}
\Psi^{S} := \Psi^{H_{S}} = \pi \circ \Psi^{H} \circ \pi = \Psi^{H} \pi- \frac{1}{2}\beta
\end{equation*}
Since $S$ is flat, we have $\nabla \pi = \pi \circ \nabla(\pi)\circ (\mathbb{I} - \pi)$.  Computing locally,
\begin{equation*}
\begin{aligned}
K(H_{S}) &= -\frac{1}{2}\phi^{jk}\nabla_{j} \Psi^{S}_{k} =  -\frac{1}{2}\phi^{jk}\nabla_{j}\left( \pi \Psi^{E}_{k} \pi\right)\\
&=  -\frac{1}{2}\phi^{jk}\left(\nabla_{j}(\pi) \Psi^{E}_{k} \pi + \pi \nabla_{j} \Psi^{E}_{k} \pi + \pi \Psi^{E}_{k} \pi \nabla_{j}( \pi) (\mathbb{I}-\pi) \right)
\end{aligned}
\end{equation*}
The last term vanishes on $S$, and so
\begin{equation}\label{eq: curv decr}
\begin{aligned}
K(H_{S}) &= -\frac{1}{2} \phi^{jk} \pi \nabla_{j}(\pi) (\mathbb{I}-\pi) \Psi^{E}_{k} \pi + \pi K(H)\pi\\
&=  -\frac{1}{4} \phi^{jk} \pi \nabla_{j}(\pi) (\mathbb{I}-\pi) (\mathbb{I}-\pi) \hat{\nabla}^{H}_{k}(\pi) \pi + \pi K(H) \pi\\
&= -\frac{1}{4}\phi^{jk} \beta_{j}^{\dagger_{H}} \beta_{k} + \pi K(H) \pi
\end{aligned}
\end{equation}

This equation has several important consequences.  For example, we can now prove that the degree of a flat subbundle is computed by integrating $\Tr(K(H_{S}))$ over $M$.

\begin{prop}\label{prop: sub bun deg}
Suppose that $(E,\nabla, \Pi)$ admits a conformally tame Hermitian metric $H$, and let $S\subset E$ be a $\nabla$-invariant subbundle.  Then the induced metric $H_{S}$ is conformally tamed by the parabolic structure $\Pi|_{S}$.  Moreover, let $\beta$ be the second fundamental form of $H$, then we have
\begin{equation*}
\int_{M} |K(H_{S})|_{H}d\nu +\int_{M} |\beta|^{2}_{H\otimes \phi} d\nu  < +\infty.
\end{equation*}
In particular,
\begin{equation*}
\deg(S) = \frac{1}{\pi} \int_{M} \Tr(K(H_{S})) d\nu.
\end{equation*}
\end{prop}
\begin{proof}
By rescaling $H$ by the conformal factor $e^{-u}$, it suffices to prove the result when $H$ is tame.  Note that $H_{S}$ satisfies condition (B) of Definition~\ref{def: tame} by tautology.  It suffices to prove the first statement, for once $|K(H_{S})|_{H} \in L^{1}(M,d\nu)$ we can apply Proposition~\ref{prop: alg deg = an deg} to deduce the last equation.  Let $\beta$ be the second fundamental form of $S$ with respect to $H$.  Since $\phi^{jk}\beta_{j}^{\dagger_{H}} \beta_{k}$ is a positive, $H$-self-adjoint endomorphism, ~\eqref{eq: curv decr} implies
\begin{equation*}
|K(H_{S})|_{H} \leq \frac{1}{4} |\beta|_{H}^{2} + |K(H)|_{H}.
\end{equation*}
In particular, it suffices to prove that $\int_{M}|\beta|_{H}^{2} d\nu <+\infty$.  To do this, we take the trace of equation~\eqref{eq: curv decr} to see
\begin{equation*}
\frac{1}{4} |\beta|_{H}^{2} = -\Tr(K(H_{S})) + \Tr(\pi K(H) \pi).
\end{equation*}
Since $H_{S}$ satisfies property (B) of Definition~\ref{def: tame}, arguing as in Proposition~\ref{prop: alg deg = an deg} we have
\begin{equation*}
\frac{1}{4}\int_{M_{\rho}} |\beta|_{H}^{2} d\nu = -\pi \deg(S,\Pi|_{S}) + \int_{M_{\rho}} \Tr(\pi K(H) \pi) d\nu + o(\frac{1}{|\log \rho|^{\epsilon}}).
\end{equation*}
The right hand side is uniformly bounded in $\rho$, and so we use the Lebesgue monotone convergence theorem to conclude. 
\end{proof}

As a consequence we can prove the following important result.

\begin{prop}[The Chern-Weil Formula]\label{prop: Chern-Weil}
Let $S$ be a flat subbundle of $(E,\nabla, \Pi)$, and suppose that $H$ is a hermitian metric on $E$ conformally tamed by the parabolic structure.  Then,
\begin{equation*}
\pi \deg(S) = \int_{M} \Tr(\pi K(H) \pi)d\nu - \int_{M}\frac{1}{4} |\beta|^{2}_{H\otimes \phi}d\nu
\end{equation*}
\end{prop}

\begin{proof}
The proof is trivial.  Take the trace of equation~\eqref{eq: curv decr} and integrate over $M$.  The integration is justified by Proposition~\ref{prop: sub bun deg}.
\end{proof}

An immediate consequence is
\begin{prop}
Suppose $(E,\nabla, \Pi)$ admits a conformally tame, Hermitian $H$ metric satisfying $K(H) = \frac{2\pi \mu(E)}{{\rm Vol}(M, d\nu)} \mathbb{I}$.  Then, for any subbundle $S$ we have
\begin{equation*}
\mu(S) \leq \mu(E)
\end{equation*}
with equality if and only if $E$ splits as a flat, orthogonal, direct sum of stable bundles $S_{1}, \dots, S_{k}$ with $\mu(S_{i}) = \mu(E)$ for $1 \leq i \leq k$.  If this happens, then the metrics $H_{S_{i}}$ are conformally tame Poisson metrics.
\end{prop}
\begin{proof}
By the Chern-Weil formula we have
\begin{equation*}
\pi \deg(S) = \pi\mu(E){\rm rk}(S) - \frac{1}{4} \int_{M} |\beta|^{2}_{H\otimes \phi} d\nu
\end{equation*}
In particular, $\mu(S) \leq \mu(E)$, with equality if and only if $\beta \equiv 0$.  But if $\beta \equiv 0$, then the projection $\pi$ is flat, and hence $E$ splits into a direct sum of flat bundles $S \oplus S^{\perp}$, and the restrictions $H_{S}$ and $H_{S^{\perp}}$ are Poisson, and conformally tame by Proposition~\ref{prop: sub bun deg}.  One then replaces $E$ with each of $S$ and $S^{\perp}$, and repeats the argument. 
\end{proof}

We conclude this section with a proposition that plays an important role in the proof of uniqueness. 

\begin{prop}\label{prop: simple}
Suppose $(E,\nabla, \Pi)$ is slope stable and admits a metric $H$ which is conformally tamed by $\Pi$.  The flat connection $\nabla$ induces a flat connection (still denoted $\nabla$) on $End(E)$. Then any $H$ self-adjoint endomorphism $f$ in the kernel of $\nabla$ is a multiple of the identity.

\end{prop}
\begin{proof}
 The key to this proposition is that both the image and kernel of such an endomorphism define $\nabla$ invariant subbundles of $E$. Let $f$ be a non-zero, $H$-self adjoint endomorphism of $E$ that satisfies $\nabla f=0$. At a point $x\in M$, let $a$ be a nonzero eigenvalue of $f$. Then the endomorphism $\Lambda:=f-a\mathbb(I)$ is also flat. 
 
 Consider a section $s$ of $E$ that lies in the image of $\Lambda$. Then $s=\Lambda(t)$ for some section $t$. Because $\Lambda$ is flat, we have $\nabla(s)=\nabla(\Lambda(t))=\Lambda(\nabla(t))$. Thus the image of $\Lambda$, denoted ${\mathcal Im}(\Lambda)$, is a flat subbundle of $E$. Furthermore, if we let $k$ be a section of $E$ such that $\Lambda(k)=0$, then we have $\Lambda(\nabla(k))=\nabla(\Lambda(k))=\nabla(0)=0$. It follows that the kernel of $\Lambda$, denoted ${\mathcal Ker}(\Lambda)$, is a flat subbundle of $E$ as well. 

Suppose $f$ is not a multiple of the identity. Then both ${\mathcal Im}(\Lambda)$ and ${\mathcal Ker}(\Lambda)$ are nonzero, proper flat subbundles of $E$. In any local frame, we know $E$ is the topological direct sum $E={\mathcal Im}(\Lambda)\oplus{\mathcal Ker}(\Lambda)$. Because both  ${\mathcal Im}(\Lambda)$ and ${\mathcal Ker}(\Lambda)$ are preserved by $\nabla$, they each must be made up of a direct sum of invariant subbundles from \eqref{eq: local decomp}, which correspond to the Jordan blocks of the residue $B_0$ coming from $\nabla$ in a temporal framing. This marks the key difference between ${\mathcal Im}(\Lambda)$ and an arbitrary flat subbundle $S$ of $E$, since for $S$ we only know it is a direct summand of local subbundles of $E$ given by definition \ref{def: local flag}, as opposed to a bundle from \eqref{eq: local decomp}.

From the definition of $\Pi$ it is now clear that
\be
\deg(E)=\deg({\mathcal Im}(\Lambda))+\deg({\mathcal Ker}(\Lambda)).
\ee
Using the fact that
$rk(E)=rk({\mathcal Im}(\Lambda))+rk({\mathcal Ker}(\Lambda))$,
a simple computation shows
\be
\mu({\mathcal Im}(\Lambda))<\mu(E)\qquad\Leftrightarrow\qquad \mu({\mathcal Ker}(\Lambda))>\mu(E).\nonumber
\ee
Since both ${\mathcal Im}(\Lambda)$ and ${\mathcal Ker}(\Lambda)$ are nonzero, proper flat subbundles of $E$, this violates stability. Thus $f$ is a multiple of the identity. 
\end{proof}

\section{A local solution near the punctures}\label{sec: model solutions}
In this section we construct an explicit solution to equation \eqref{AHYM} in a neighborhood of each puncture $p_{i}$.

\begin{thm}\label{thm: H0}
Let $(E,\nabla)$ be a flat vector bundle over $M=\overline{M}\,\backslash\{p_{1},\dots,p_{m}\}$, and suppose that $\nabla$ has regular singularities.  Let $\Pi$ be a parabolic structure for $(E,\nabla)$.  Then there exists a smooth metric $H$ on $E$, strongly conformally tamed by $\Pi$, solving equation~\eqref{AHYM} in a neighborhood of $p_{j}, j=1,\dots,m$, for the constant $c$ determined by Corollary~\ref{cor: c def}.  Moreover, if $e^{\psi} \in L^{p}(\overline{M}, \mathbb{R})$ for some $p>2$, then $H$ is strongly tamed by the parabolic structure.
\end{thm}

Clearly the theorem is local, in the sense that the metric can be made arbitrary away from the punctures.  Fix a puncture $p_{j}$, and a small ball $B_{R}(p_{j}) \subset \overline{M}$ and  work in the parabolic framing on $B_{R}(p_{j})\backslash \{p_{j}\}$. We also identify $B_{R}(p_{j})$ with ${\mathcal D}$, the unit disk in $\mathbb{C}$, and work in polar coordinates $(r,\theta)$.
\begin{prop}\label{prop: local model 1}
Suppose that on $\mathcal{D}$, the connection $\nabla$ is given in the parabolic framing by $\nabla = d + A \frac{dr}{r} + iB_{0} d\theta$ where $A = {\rm diag}\{a,a,\dots,a\}$ and $B_{0}$ is a single Jordan block with $\kappa$ on the diagonal.  Define  
\begin{equation}\label{eq: diagonal entries}
\lambda_{i}(r) = \frac{(n-i)!}{(i-1)!}|\log(r)|^{2i-(n+1)} \quad \text{ for } i \leq i \leq n.
\end{equation}
Then the metric $H = {\rm diag} \{\lambda_{1}(r) ,\lambda_{2}(r), \dots, \lambda_{n}(r)\}$ is strongly tamed by the parabolic structure, and solves $K(H)=0$.
\end{prop}
\begin{proof}
Suppose that $H = {\rm diag} \{\lambda_{1}(r) ,\lambda_{2}(r), \dots, \lambda_{n}(r)\}$.  We will compute explicitly the system of differential equations that $\lambda_{i}$ must solve in order to satisfy $K(H)=0$. Define the matrix $N$ by $N_{ij} = \delta_{i,j+1}$, so that $(B_{0})_{ij} = \kappa \delta_{i,j} + N_{ij}$. We define a gauge transformation $\sigma$ by
\begin{equation}\label{eq: flat gauge transform}
\sigma = r^{-a}e^{-\kappa \theta} \exp(-\theta N),
\end{equation}
where $\sigma^{-1}$ transforms between the parabolic gauge and a flat frame. 
In the flat frame, the metric is given by $\hat{H} := \sigma^{\dagger} H \sigma$, and $\Psi = \frac{1}{2} \hat{H}^{-1}d\hat{H}$.  Using equation~\eqref{eq: flat gauge transform}, we compute
\begin{equation}\label{eq: local eqn for flat}
\begin{aligned}
 \hat{H}^{-1}d\hat{H} =& \left(\frac{-2a}{r} \mathbb{I} + \exp(N\theta)H^{-1}\pl_{r}H \exp(-N\theta)\right) dr + \left(-2{\rm Re}(\kappa) \mathbb{I} - N\right) d\t\\
& -\bigg( \exp(N\theta)H^{-1}N^{T} H \exp(-N\theta)\bigg)d\t
\end{aligned}
\end{equation}
Recalling equation~\eqref{eq: hodge star}, it is an easy computation to verify that $\star \nabla \star \Psi =0$ if and only if
\begin{equation*}
r\del_r \left(r  \del_r \log(\lambda_{i})\right) =
\left\{ \begin{array}{ll}
 \frac{\lambda_{i}}{\lambda_{i+1}}-\frac{\lambda_{i-1}}{\lambda_{i}} & \text{ if } 1\leq i\leq n-1\\[8pt]
 -\frac{\lambda_{n-1}}{\lambda_{n}} & \text{ if } i=n
\end{array}
\right.
\end{equation*}
where we set $\lambda_{0} =0$.  It is straightforward to verify that the $\lambda_{i}$'s given by~\eqref{eq: diagonal entries} satisfy this system.  It only remains to show that $H$ is strongly tamed by the parabolic structure.  Note that conditions (A) and (C) of Definition~\ref{def: tame} are automatically satisfied.  Conditions (B) and (D) follow immediately from equation~\eqref{eq: local eqn for flat} by direct computation.  Alternatively, one can combine Lemma~\ref{lem: unitary formula}, and Remark~\ref{rk: local form for sub} below with equation~\eqref{eq: psi formula}.

\end{proof}

We now give the proof of Theorem~\ref{thm: H0}, modulo some details which appear at the end of this section.

\begin{proof}[Proof of Theorem~\ref{thm: H0}]

We construct the metric $H_{0}$ in the statement of Theorem~\ref{thm: H0}, by taking direct sums and conformal rescalings of the local model metrics in Proposition~\ref{prop: local model 1}.  For each puncture $p_{j}$, fix a ball $B_{2R_{j}}(p_{j})$ where $E = V_{1}\oplus \dots\oplus V_{k}$, as in equation~\eqref{eq: local decomp}, and each $V_{\ell}$ admits a metric $H_{V_{\ell}}$ given by Proposition~\ref{prop: local model 1}.  We take $\hat{H}$ to be any smooth metric on $E$ which agrees with $H_{V_{1}}\oplus H_{V_{2}} \oplus \dots \oplus H_{V_{k}}$  on $B_{R_{j}}(p_{j})$.  A metric of this sort can easily be constructed using a partition of unity.  Observe that $\hat{H}$ is strongly tamed by the parabolic structure.  To see this, observe that conditions (A) and (C) of Definition~\ref{def: tame} are automatically satisfied.  Condition (B) follows immediately from equation~\eqref{eq: local eqn for flat} by direct computation.  Alternatively, one can combine Lemma~\ref{lem: unitary formula}, and Remark~\ref{rk: local form for sub} below with equation~\eqref{eq: psi formula}.

Consider the metric $H_{0} = e^{u}\hat{H}$.  By direct computation we have
\begin{equation*}
K(H_{0}) = -\frac{1}{4}\Delta_\phi u \mathbb{I} + K(\hat{H}).
\end{equation*}
By Lemma~\ref{functionu} below, there exists a function $u \in C^{\infty}(\bar{M},\mathbb{R}) \cap L^{2}(\overline{M}, dV)$ solving
\begin{equation}\label{eq: laplace}
 -\frac{1}{4}\Delta_\phi u = -\frac{1}{n}\Tr \left(K(\hat{H})\right) + c.
 \end{equation}
The metric $H_{0} = e^{u}\hat{H}$ solves equation~\eqref{AHYM} in a neighborhood of each puncture $p_{j}$.  By Lemma~\ref{functionu}, $H$ is conformally strongly tamed by the parabolic structure.  Moreover, if $e^{\psi} \in L^{p}(\overline{M}, dV)$ for some $p>2$ then $u\in C^{1,\alpha}(\overline{M},\mathbb{R})$, and so $H_{0}$ is strongly tamed. 
 \end{proof}

It remains to prove that we can solve the Laplace equation~\eqref{eq: laplace}.  The only reason this is not trivial is that the metric $\phi_{ij}$ is singular at the punctures.  Nevertheless, $\phi_{ij}$ is conformal to a smooth K\"ahler metric on $\overline{M}$, and the singularities of the conformal factor are sufficiently mild to make this possible.

\begin{lem}\label{functionu}
Let $\hat{H}$ be any smooth metric on $E$ tamed by the parabolic structure, and satisfying $\Tr(K(\hat{H}))=0$ in an open neighborhood of each puncture $p_{j}$.  Then there exists a function  $u \in C^{\infty}(M,\mathbb{R}) \cap L^{2}(\overline{M}, dV)$ solving
\begin{equation*}
 -\frac{1}{4}\Delta_\phi u = -\frac{1}{n}\Tr \left(K(\hat{H})\right) + c,
 \end{equation*}
on $M$, in the sense of distributions.  Moreover, if $e^{\psi} \in L^{p}(\overline{M}, \mathbb{R})$ for some $p>2$, then $u\in C^{1,\alpha}(\overline{M},\mathbb{R})$ for some $\alpha >0$.
 \end{lem}
 \begin{proof}
Let $f=-\frac{1}{n}\Tr \left(K(\hat{H})\right) + c$.  Since $\hat{H}$ is tamed by the parabolic structure, Proposition~\ref{prop: alg deg = an deg} gives
\begin{equation*}
\int_{M}(-\frac{1}{n}\Tr \left(K(\hat{H})\right) + c)d\nu = \int_{M} e^{\psi} (-\frac{1}{n}\Tr \left(K(\hat{H})\right) + c)dV=0,
\end{equation*}
where $e^{\psi}$ is the conformal factor of the metric $\phi_{ij}$. In particular, there exists a function  $u :\overline{M} \rightarrow \mathbb{R}$ solving $\Delta u =e^{\psi} f$.  Such a function can easily be constructed by integrating against the Green's function of $(\overline{M}, \bar{g})$.  Clearly $u$ is smooth on the open manifold $M$.  Also, by assumption $e^{\psi} \in L^{p}(\overline{M}, dV)$ for some $p\geq 1$ and $f \equiv c$ in a neighborhood of each puncture.  If $p>2$, then elliptic regularity implies that $u \in C^{1,\alpha}(\overline{M},\mathbb{R})$ for some $\alpha >0$.   Otherwise, $u \in W^{1,q}(\overline{M}, dV)$ for every $q<2$ by \cite[Lemma 14]{Mar}.
\end{proof}

\begin{rk}
Let us point out that the above lemma is the source of the technical difficulties which necessitate the introduction of the terminology {\em conformally tamed}.  In particular, the conformal factor $u$ appearing in Definition~\ref{def: tame} is determined by the parabolic structure and the metric $g$ on $M$, at least asymptotically near the punctures.
\end{rk}

To end this section, we compute some local formulae which will be useful in the study of the regularity theory of weak solutions of the Poisson metric equation~\eqref{AHYM}.  By construction, the metric $H_{0}$ is given in a neighborhood of each puncture $p_{j}$ by a direct sum of metrics on each of the local indecomposable subbundles $V_{\ell}$.  That is, we have $H_{0} = H_{V_{1}} \oplus \cdots \oplus H_{V_{k}}$, and in the parabolic framing, we have
\begin{equation}\label{eqn: explicit metric}
(H_{V_{\ell}})_{\alpha \beta} = \delta_{\alpha \beta} \left[e^{u}\frac{(d_{\ell} -\alpha)!}{(\alpha-1)!}|\log(r)|^{2\alpha-(d_{\ell}+1)}\right]
\end{equation}
where, $d_{\ell} = \dim V_{\ell}$ and $u \in C^{1,\alpha}(\overline{M},\mathbb{R})$.

\begin{lem}\label{lem: unitary formula}
Near each puncture $p_{j}$, there exists a unitary framing for $(E, H_{0})$ so that the connection is given in polar coordinates by $\nabla = d + \Omega_{r} dr + \Omega_{\theta} d\theta$, where $\Omega_{r}$ and $\Omega_{\theta}$ are block diagonal with respect to the decomposition of $E$ into local indecomposable subbundles $V_{\ell}$, and when restricted to each indecomposable local subbundle $V_{\ell}$ we have
\begin{equation}
\begin{aligned}
(\Omega_{r}|_{V_{\ell}})_{\alpha \beta} &= \delta_{\alpha \beta}\left[-\frac{1}{2}\pl_r u  + \frac{2\alpha -(d_{\ell}+1)}{2r\log r}  + \frac{w_{\ell}}{r}\right]   \\
\Omega_{\theta}|_{V_{\ell}} &=  \delta_{\alpha \beta}\left[-\frac{1}{2}\pl_\t u + \kappa_{\ell}\right] + \delta_{(\alpha+1)\beta} \left[\frac{\sqrt{\alpha(d_{\ell}-\alpha)}}{|\log(r)|} \right].
\end{aligned}
\end{equation}
Moreover, the flat connection associated to $H_{0}$ is given by $\hat{\nabla}^{0} = d -  \Omega_{r}^{\dagger} dr - \Omega_{\theta}^{\dagger} d\theta$, where $\dagger$ is shorthand for $\dagger_{H_0}$.
\end{lem}
\begin{proof}
The proof is just a computation, using the explicit form of the metric $H_{0}$ in a parabolic framing.  It suffices to consider each indecomposable subbundle separately.  Working in the parabolic framing, we define a diagonal matrix by $\sigma = (H_{0})^{-1/2}$.  On an indecomposable subbundle $V_{\ell}$ we have
\begin{equation*}
(\sigma|_{V_{\ell}})_{\alpha \beta} = \delta_{\alpha \beta} \left[e^{-u/2}\sqrt{\frac{(\alpha-1)!}{(d_{\ell} -\alpha)!}}|\log(r)|^{\frac{(d_{\ell}+1)}{2}-\alpha}\right]
\end{equation*}
The connection on $V_{\ell}$ is given by
\begin{equation*}
\nabla|_{V_{\ell}} = d + \sigma^{-1} d\sigma + w_{\ell} \mathbb{I}_{V_{\ell}} \frac{dr}{r} + \kappa_{\ell} \mathbb{I} d\theta + \sigma^{-1}N \sigma d\theta.
\end{equation*}
Here $N$ is the $d_{\ell} \times d_{\ell}$ nilpotent matrix with $N_{\alpha \beta} = \delta _{(\alpha+1)\beta}$ (ie. ones on the super-diagonal and zeroes elsewhere).  Then by direct computation we have 
\begin{equation*}
( \Omega_{r}|_{V_{\ell}})_{\alpha \beta} = \delta_{\alpha \beta}\left[-\frac{1}{2}\pl_ru  + \frac{2\alpha -(d_{\ell}+1)}{2r\log r}  + \frac{w_{\ell}}{r}\right],
\end{equation*}
and similarly
\begin{equation*}
(\Omega_{\theta}|_{V_{\ell}})_{\alpha \beta} = \delta_{\alpha \beta}\left[-\frac{1}{2}\pl_\t u + \kappa_{\ell}\right] + \delta_{(\alpha+1)\beta} \left[\frac{\sqrt{\alpha(d_{\ell}-\alpha)}}{|\log(r)|} \right].
\end{equation*}
\end{proof}

\begin{rk}\label{rk: local form for sub}
Note that if $V \subset E$ is a local indecomposable subbundle which is given in the unitary framing constructed above as ${\rm Span}\{e_{1}, \dots, e_{k}\}$,  then any local invariant subbundle $S \subset V$ is given by ${\rm Span}\{e_{1},\dots, e_{\ell}\}$ for $\ell \leq k$.  Moreover, if $\pi^{S}$ is the orthogonal projection to $S$, regarded as an element of ${\rm Hom}(E,E)$, then in the unitary framing, $\pi^{S}$ is diagonal, and consists only of ones and zeroes. 
\end{rk} 

\section{A priori estimates for bounded solutions}
\label{sec: regularity}

The aim of this section is to prove some a priori estimates for Poisson metrics; namely, solutions of~\eqref{AHYM}.  We fix a flat connection with regular singularities and a compatible parabolic structure as before.  Let $H_{0}$ denote the model solution given by Theorem~\ref{thm: H0}.  The main theorem of this section is
\begin{thm}\label{thm: regularity}
 Suppose that $H$ is a smooth, hermitian metric on $M$ solving ~\eqref{AHYM} with the property that there is a constant $0<C<\infty$ such that the positive definite, hermitian endomorphism $h := H_{0}^{-1} H$ has eigenvalues bounded below by $C^{-1}$ and above by $C$.  Suppose also that
 \begin{equation*}
 \int_{M} |\nabla h|^{2}_{H_{0}\otimes \phi} d\nu <C.
 \end{equation*}
 Then $H$ is strongly conformally tamed by the parabolic structure.  If $e^{\psi} \in L^{p}(\overline{M}, dV)$ for some $p>2$, then $H$ is strongly tamed by the parabolic structure.
 \end{thm}
 
First, note that $H$ trivially has the asymptotics given in part (C) of Definition~\ref{def: tame} due to the boundedness assumption, and the fact that $H_{0}$ is (conformally) strongly tamed.  As a result, the main content of Theorem~\ref{thm: regularity} is a gradient estimate near the punctures, of the type in parts (B) and (D) of Definition~\ref{def: tame}.  This matter is somewhat complicated by the presence of singular gauge transformations.  In order to remove this difficulty, throughout this section we work exclusively in the unitary framing for $(E, H_{0})$ near each puncture $p_{j}$ given by Lemma~\ref{lem: unitary formula}. As the techniques in this section are completely local, we return to the disk $\mathcal{D} \subset \mathbb{R}^{2}$, and identify the bundle $E$ with the restriction of the trivial bundle $\mathbb{C}^{n} \rightarrow \mathcal{D}$ via the unitary framing for $(E, H_{0})$.  As before, the monodromy of the flat connection $\nabla$ decomposes $E$ into a direct sum of indecomposable subbundles, which we write as 
\begin{equation*}
E = V_{1} \oplus \cdots \oplus V_{k}.
\end{equation*}
Let let $h$ be a function on $\D$ valued in the $n \times n$ hermitian matrices.  Then the decomposition of $E$ induces a decomposition of $h$ which we write as
\begin{equation*}
h = \oplus h_{ij}, \qquad h_{ij} \in {\rm Hom} (V_{j}, V_{i}).
\end{equation*} 
Since $h$ is hermitian, we have that ${\overline h_{ij}}^{T} = h_{ji}$.  Any such function defines a local section of ${\rm Hom}(E,E)$ via the identifications above, and we clearly have
\begin{equation*}
(\nabla h)^{\dagger} = \hat{\nabla}^{0} h.
\end{equation*}
For the rest of this section we let $\dagger$ be shorthand for $\dagger_{H_0}$. We begin by proving a lemma which relates the Sobolev spaces defined by the singular, flat connection $\nabla$ with the standard Euclidean Sobolev spaces.

\begin{prop}\label{prop: sob compare}
Let $h$ be a hermitian matrix valued function defined on $ \mathcal{D} \backslash \{0\}$, with bounded $L^{\infty}$ norm.  Suppose that $h(t,\theta)$ is $C^{3}$ as a hermitian matrix valued function on the circle $\{r=t\}$ for each $0<t<1$, and that 
\begin{equation*}
\int_{B_{\rho}(0) \backslash \{0\}} \phi^{ij}\Tr\left(\hat{\nabla}_{i}^{0}h (\hat{\nabla}_{j}^{0}h)^\dagger \right) d\nu  < C,
\end{equation*}
for $0<\rho <1$. Then $h$ extends to an element of $W^{1,2}\left(B_{\rho}(0)\right)$.  Moreover, there is a constant $A$ depending only on $\|h\|_{L^{\infty}(\mathcal{D})}$, the parabolic structure, and the monodromy $\nabla$ around $0$ so that
\begin{equation*}
\int_{B_{\rho}(0)} |dh|_{\bar{g}}^{2} dV < A\left(C + \frac{1}{-\log(\rho)}\right).
\end{equation*}
\end{prop}
\begin{proof}
It suffices to prove that 
\begin{equation*}
\int_{B_{\rho}(0)\backslash \{0\}} |dh|_{\bar{g}}^{2} dV < A\left( C+ \frac{1}{-\log(\rho)}\right)
\end{equation*}
as it is a classical fact that $W^{1,2}\left(B_{\rho}(0) \backslash \{0\}\right) = W^{1,2}\left(B_{\rho}(0)\right)$.  Thanks to Lemma~\ref{lem: unitary formula} we can write the connections $\nabla, \hat{\nabla}^{0}$ as
\begin{equation*}
\begin{aligned}
\nabla &= d + \left(-\frac{\del_{r}u}{2} \mathbb{I} + \frac{M_{1}}{r\log r} + \frac{W}{r} \right) dr + \left( -\frac{\del_{\theta}u}{2}\mathbb{I} + \frac{M_{2}}{| \log r |} + K\right) d\theta\\
\hat{\nabla}^{0} &= d +\left(\frac{\del_{r}u}{2}\mathbb{I} - \frac{M_{1}}{r\log r} - \frac{W}{r} \right) dr + \left( \frac{\del_{\theta}u}{2}\mathbb{I} - \frac{M_{2}^{T}}{| \log r |} -\overline{K}\right) d\theta.
\end{aligned}
\end{equation*}
Here, $M_{1}, M_{2}$ denote matrices with constant coefficients whose precise form will not be needed, but can be easily determined from Lemma~\ref{lem: unitary formula}. The matrices $W, K$ denote the matrices of parabolic weights, and the generalized eigenvalues of the residue of $\nabla$ respectively.  We can ignore the terms containing derivatives of $u$, since they act trivially on the endomorphism bundle.  We compute
\begin{equation*}
\begin{aligned}
\frac{1}{2}(\nabla h - \hat{\nabla}^{0} h) =&\left( \oplus_{i<j} \frac{(w_{i}-w_{j})}{r}h_{ij} + \frac{[M_{1},h]}{r\log(r)} \right)dr\\
&+ \left( \oplus_{i<j} {\rm Re}(\kappa_{i}-\kappa_{j}) h_{ij} + \frac{[M_{2}+M_{2}^{T},h]}{2|\log r|} \right) d\theta
\end{aligned}
\end{equation*}
Taking norms and integrating we have
\begin{equation}\label{eq:  diff of nabla}
\sum_{1 \leq i , j \leq n} \int_{B_{\rho}(0)} \left(\frac{(w_{i}-w_{j})^{2}}{r^{2}}+ \frac{{\rm Re}(\kappa_{i}-\kappa_{j})^{2}}{r^{2}}\right)|h_{ij}|^{2} dV < C + \frac{A(\|h\|_{L^{\infty}})}{-\log \rho},
\end{equation}
where $|h_{ij}|^{2} = \Tr(h_{ij}(h_{ij})^{\dagger})$.  Similarly we compute
\begin{equation*}
\frac{1}{2}(\nabla h + \hat{\nabla}^{0}h) = dh+ \left(\oplus_{i<j}\sqrt{-1}{\rm Im}(\kappa_{i}-\kappa_{j}) h_{ij} +\frac{[M_{2}-M_{2}^{T},h]}{2|\log r|}\right) d\theta
\end{equation*}
Taking norms and integrating we obtain
\begin{equation}\label{eq: sum of nabla}
\sum_{1\leq i, j \leq n} \int_{B_{\rho}(0)} \frac{|\del_{\theta}h_{ij} + \sqrt{-1}{\rm Im}(\kappa_{i}-\kappa_{j}) h_{ij}|^{2}}{r^{2}} dV \leq C + \frac{A(\|h\|_{L^{\infty}})}{-\log \rho},
\end{equation}
as well as a much stronger estimate for the radial derivative,
\begin{equation*}
\int_{B_{\rho}(0)} |\del_{r} h|^{2} dV < C.
\end{equation*}

As a result, it suffices to estimate the integral of $r^{-2}|\del_{\theta} h_{ij}|^{2}$ over $B_{\rho}(0)$.  If $\kappa_{i}=\kappa_{j}$, then we are done by equation~\eqref{eq: sum of nabla}, and so we may assume this is not the case.  If ${\rm Re}(\kappa_{i} -\kappa_{j}) \ne 0$, then the estimate in equation~\eqref{eq: diff of nabla}, combined with~\eqref{eq: sum of nabla} implies the result.  Thus, we are reduced to the case when ${\rm Re}(\kappa_{i}) = {\rm Re}(\kappa_{j})$, and ${\rm Im}(\kappa_{i}) \ne {\rm Im}(\kappa_{j})$.  By the choice of the unitary framing, we know that ${\rm Im}(\kappa_{i}) \in[0,1)$ for each $1\leq i \leq n$.  We claim that there is a number $\delta >0$ such that
\begin{equation}\label{eq: no kernel est}
\int_{B_{\rho}(0)} \frac{|\del_{\theta}h_{ij} + \sqrt{-1}{\rm Im}(\kappa_{i}-\kappa_{j}) h_{ij}|^{2}}{r^{2}} dV > \delta \int_{B_{\rho}(0)} \frac{|h_{ij}|^{2}}{r^{2}} dV.
\end{equation}
The proposition clearly follows from this claim, so we are reduced to proving~\eqref{eq: no kernel est}.  This estimate essentially follows from the elementary fact that, on the circle, the operator $\del_{\theta} + i\epsilon$ has no kernel for $\epsilon \notin \mathbb{Z}\backslash \{0\}$.  For ease of notation, let us set $\lambda_{ij} =  {\rm Im}(\kappa_{i}-\kappa_{j}) \in (-1,0)\cup(0,1)$.  Set
\begin{equation*}
\delta_{ij} = \min \{ |1+\lambda_{ij}|, |\lambda_{ij}|, |\lambda_{ij}-1| \} > 0.
\end{equation*}  
We write the integral on the left hand side of~\eqref{eq: no kernel est} as
\begin{equation*}
\int_{0}^{\rho} \frac{dr}{r^{2}} \int_{0}^{2\pi} |\del_{\theta} h_{ij}(r,\theta) + \sqrt{-1}\lambda_{ij} h_{ij}(r,\theta)|^{2} d\theta
\end{equation*}
On each circle, we write $h_{ij}$ as its Fourier series
\begin{equation*}
h_{ij}(r,\theta) = \sum_{N \in \mathbb{Z}} b^{N}_{ij}(r) e^{\sqrt{-1}N\theta}, \qquad \del_{\theta}h_{ij} = \sqrt{-1}\sum_{N \in \mathbb{Z}} Nb_{ij}^{N}(r)e^{\sqrt{-1}N\theta}.
\end{equation*}
where each equality is valid since $h(t,\theta)$ is $C^{3}$ on the circle $\{r=t\}$ for $0<t\leq R$.  We compute 
\begin{equation*}
\begin{aligned}
\int_{0}^{2\pi} |\del_{\theta} h_{ij}(r,\theta) + \sqrt{-1}\lambda_{ij} h_{ij}(r,\theta)|^{2} d\theta =&  2\pi \sum_{N\in \mathbb{Z}}  (N+ \lambda_{ij})^{2}|b^{N}_{ij}(r)|^{2}\\
&\geq2\pi  \delta_{ij}   \sum_{N\in \mathbb{Z}} |b^{N}_{ij}(r)|^{2}\\ &= \delta_{ij} \int_{0}^{2\pi} |h_{ij}(r,\theta)|^{2} d\theta.
\end{aligned} 
\end{equation*}
The inequality in~\eqref{eq: no kernel est} clearly follows from this estimate, and the proposition is proved.
\end{proof}

The assumptions of Theorem~\ref{thm: regularity}, together with Lemma \ref{lem: endo equation} imply that the hermitian endomorphism $h=H_0^{-1}H$ satisfies

\be
\frac{1}{\sqrt{\det(\phi_{pq})}}\nabla_{i}\left( \sqrt{\det(\phi_{pq})} \phi^{ij}\, h^{-1} \hat{\nabla}^{0}_{j} h\right) = 0.
\ee
Since the right hand side is zero, we can multiply the above equation by the conformal factor relating $\phi_{pq}$ to the background metric $\bar{g}_{pq}$, which by assumption, is Euclidean on $\D$.  In particular, we have
\begin{equation}\label{eq: endo equation}
\frac{1}{\sqrt{\det(\bar{g}_{pq})}}\nabla_{i}\left( \sqrt{\det(\bar{g}_{pq})} \bar{g}^{ij}\, h^{-1} \hat{\nabla}^{0}_{j} h\right)  =0.
\end{equation}
The previous proposition permits us to integrate by parts, and so we can prove

\begin{lem}\label{lem: weak solution}
Suppose that both $H$ and $H_0$ are $C^{2}$ solutions to equation~\eqref{AHYM} on $\mathcal{D}-\{0\}$ with the property that the endomorphism $h = H_{0}^{-1} H$ is bounded from above and below, and has $|\hat{\nabla}^{0}h|_{H_{0}\otimes \phi} \in L^{2}(\D, d\nu)$.  Then $h$ is a weak solution of equation~\eqref{eq: endo equation} on $\D$ in the sense that, for any compactly supported hermitian matrix valued function $k \in L^{\infty}(\mathcal{D})\cap W^{1,2}(\mathcal{D})$ defined on $\mathcal{D}$ for which $|\hat{\nabla}^{0} k|_{H_{0}\otimes \phi} \in L^{2}(\mathcal{D}, d\nu)$ we have
\begin{equation*}
\int_{\mathcal{D}} \bar{g}^{ij}\Tr\left(h^{-1} \hat{\nabla}_{i}^{0}h (\hat{\nabla}_{j}^{0} k)^{\dagger}\right) dV =0.
\end{equation*}
\end{lem}

The proof is straightforward, and so we omit the details.  We need one final estimate.
\begin{lem}\label{lem: gradient decay}
Suppose that $h \in C^{\infty}(\D \backslash \{0\})\cap L^{\infty}(\D)$ is a weak solution of equation~\eqref{eq: endo equation}, in the sense of Lemma~\ref{lem: weak solution}.  Then for every $0< \rho \ll1$ we have the estimate
\begin{equation}\label{eq: key 1}
\int_{B_{\rho}(0)} |\nabla h|_{\bar g}^{2}dV \leq -100\pi\frac{ \|h\|_{L^{\infty}}^{2}}{\log \rho}
\end{equation}
\end{lem}
\begin{proof}
In order to establish this estimate, for $\sigma <\rho$ we introduce the test function
\begin{equation*}
G^{\sigma} = \left\{\begin{array}{rl}
\log \rho - \log r  & r \geq \sigma\\[8pt]
-\frac{r^{2}}{2\sigma^{2}} + \log(\frac{\rho}{\sigma}) +\frac{1}{2} & r < \sigma\\[8pt].
\end{array}
\right.
\end{equation*}
It is easily verified that $G^{\sigma} \in C^{1}(B_{\rho}(0))$ and $G^{\sigma}$ is positive and vanishes on $\del B_{\rho}(0)$.  As a result, we can take $k = h G^{\sigma}$ as a test function in Lemma~\ref{lem: weak solution} to obtain
\begin{equation*}
\int_{B_{\rho}(0)}\bar{g}^{ij}\Tr(h^{-1}\nabla_{i}h(\nabla_{j}h)^{\dagger})G^{\sigma}dV = - \int_{B_{\rho}(0)}\bar{g}^{ij}\nabla_{i}\Tr(h)\nabla_{j}G^{\sigma} dV.
\end{equation*}
Since $h \in W^{1,2}(B_{\rho}(0))$ by Proposition~\ref{prop: sob compare}, it follows that $\Tr(h) \in W^{1,2}(B_{\rho}(0)$.  Moreover, $G^{\sigma}$ is smooth away from the set $r=\sigma$.  Thus, we can integrate by parts on the right hand side of the above equation to obtain
\begin{equation*}
\begin{aligned}
-\int_{B_{\rho}(0)}\bar{g}^{ij}\nabla_{i}\Tr(h)\nabla_{j}G^{\sigma} = &- \int_{\del B_{\rho}(0)} \Tr(h) (\nabla G^{\sigma}\cdot n)dS \\
&+ \int_{B_{\rho}(0)}\Tr(h)\Delta G^{\sigma} dV,
\end{aligned}
\end{equation*}
where the second integral is understood to be over $B_{\rho}(0)\backslash\{r=\sigma\}$, where $\Delta G^{\sigma}$ is defined.  The first integral is easily bounded.  Using the formula for $G^{\sigma}$ we have $(\nabla G^{\sigma}\cdot n) = \frac{-1}{\rho}$ on $\del B_{\rho}(0)$, and so
\begin{equation*}
- \int_{\del B_{\rho}(0)} \Tr(h) (\nabla G^{\sigma}\cdot n)dS \leq 2\pi \|h\|_{L^{\infty}}.
\end{equation*}
For the second integral, we observe that $\Delta G^{\sigma} =0$ on $B_{\rho}(0) \backslash B_{\sigma}(0)$, and $\Delta G^{\sigma} \leq 0$ on $B_{\sigma}(0)$.  Since $\Tr(h) >0$, the second integral is clearly negative.  As a result, we have
\begin{equation*}
\int_{B_{\rho}(0)}\bar{g}^{ij}\Tr(h^{-1}\nabla_{i}h(\nabla_{j}h)^{\dagger})G^{\sigma}dV \leq 2\pi \sup \Tr(h).
\end{equation*}
The integrand on the left hand side of this estimate is clearly positive.  Choose $\sigma \ll \rho^{3/2}$.  Then, on $B_{\rho^{3/2}}(0) \subset B_{\rho}(0)$, we have $G^{\sigma} \geq -\frac{1}{2}\log \rho$, and hence
\begin{equation*}
\int_{B_{\rho^{3/2}}(0)} |\nabla h|^{2}dV \leq -100\pi\frac{ \|h\|_{L^{\infty}}^{2}}{\log \rho^{3/2}}
\end{equation*}
which is nothing other than equation~\eqref{eq: key 1}.
\end{proof}

In order to prove Theorem~\ref{thm: regularity}, we will study the regularity properties of  bounded solutions to equation~\eqref{eq: endo equation} when written in logarithmic coordinates on the punctured ball.  We set
\begin{equation*}
x = - \log r \quad y =\theta,
\end{equation*}
so that we can take $(x,y) \in (100, \infty) \times (-\infty, \infty) := \log\D$.  In these coordinates the connection $\nabla$ is given by
\begin{equation*}
\nabla = d + (-\frac{\del_{x} u}{2} \mathbb{I} - \frac{M_{1}}{x} + W) dx + (-\frac{\del_{y} u}{2} \mathbb{I} +\frac{M_{2}}{x} +K)dy
\end{equation*}
 for constant matrices $M_{1}, M_{2}, K$.  Again, we ignore the terms containing derivatives of $u$ since they act trivially on ${\rm Hom}(E,E)$. The key point is that the connection coefficients are smooth, and uniformly bounded in any $C^{k}$ norm on $\log \D$.  The metric $\bar{g}$ is easily computed to be $e^{-2x}(dx^{2}+dy^{2})$, and hence the pulled-back hermitian matrix valued function $h(x,y)$ solves
 \begin{equation}\label{eq: endo log coords}
\star \nabla \star (h^{-1}\hat{\nabla}^{0}_{j}h) = \nabla_{x}(h^{-1}\hat{\nabla}^{0}_{x}h) +  \nabla_{y}(h^{-1} \hat{\nabla}^{0}_{y}h) =0.
 \end{equation} 
Moreover, the estimate~\eqref{eq: key 1}, combined with Proposition~\ref{prop: sob compare} implies there is a universal  constant $C$ such that
\begin{equation}\label{eq: key 1 log}
\int_{0}^{2\pi} \int_{-\log \rho}^{\infty} |\del_{x} h|^{2} + |\del_{y} h|^{2} dx dy \leq \frac{-C \| h\|_{L^{\infty}}^{2}}{\log \rho}.
\end{equation}

The final ingredient in the proof of Theorem~\ref{thm: regularity} is the following estimate, which is a modification of an estimate due to Hildebrandt \cite{Hil} in the study of harmonic maps.  This estimate was exploited by Bando-Siu \cite{BaS} in the study of Hermitian-Einstein metrics on coherent sheaves.  As the proof is quite long, we have deferred it to the Appendix, where we provide a detailed proof for the convenience of the reader.  

\begin{prop}\label{prop: C1a estimate}
Suppose $h(x,y) \in C^{\infty}(\log \D)\cap L^{\infty}(\log \D)$ is a hermitian matrix valued function solving equation~\eqref{eq: endo log coords}.  Then there exists constants $C, \alpha>0$ depending only on $\|h\|_{L^{\infty}(\log \D)}$, and $\|h^{-1}\|_{L^{\infty}(\log \D)}$  so that
\begin{equation*}
\|h\|_{C^{1,\alpha}(\log \D)} \leq C.
\end{equation*}
\end{prop}

Note that, in order to prove this proposition, it suffices to prove interior estimates.  This is taken up in generality in the Appendix.

We now give the proof of Theorem~\ref{thm: regularity}, assuming Proposition~\ref{prop: C1a estimate}.  
\begin{proof}[Proof of Theorem~\ref{thm: regularity}]
It suffices to prove that conditions (B) and (D) of Definition~\ref{def: tame} hold.  Set
\begin{equation*}
\phi(r_{0}) = \sup_{[0,2\pi] \times [-\log r_{0},\infty)} |\del_{x} h|.
\end{equation*}
Choose a point $(x_{0}, y_{0}) \in [0,2\pi] \times [-\log r_{0},\infty)$ such that $|\del_{x} h(x_{0},y_{0})| \geq \frac{\phi(r_{0}^{2})}{2}$.  By the Proposition~\ref{prop: C1a estimate}, there is a uniform constant $C>0$, so that $|\del_{x} h(x,y)| \geq \frac{\phi(r_{0}^{2})}{4}$ on the ball of radius $\left(\frac{\phi(r_{0}^{2})}{4C}\right)^{1/\alpha}$ centered at  $(x_{0}, y_{0})$.  We use this estimate to bound below the integral on the left hand side of~\eqref{eq: key 1 log} with $\rho = r_{0}$, to obtain
\begin{equation*}
\left(\frac{\phi(r_{0}^{2})}{4}\right)^{2(1+\frac{1}{\alpha})}\frac{\pi}{C^{2/\alpha}} \leq \frac{C'}{-\log(r_{0})}
\end{equation*}
for uniform constants $C, C'$.  Reorganizing gives
\begin{equation*}
\phi(r_{0}^{2}) \leq C\left(-2\log( r_{0})\right)^{-(\alpha/(2\alpha+2))}
\end{equation*}
for a different, uniform constant $C$.  In particular, we have
\begin{equation*}
\phi(\rho) \leq C(-\log( \rho))^{(-\alpha/(2\alpha+2))}.
\end{equation*}
Rewriting this in polar coordinates on $\D$ gives
\begin{equation*}
|\del_{r} h (\rho,\theta)| \leq \frac{C}{\rho}(-\log( \rho))^{(-\alpha/(2\alpha+2))}.
\end{equation*}
An identicaly argument proves $|\del_{\theta} h| \leq C(-\log( \rho))^{(-\alpha/(2\alpha+2))}$.  It remains only to estimate the size of the off-diagonal components of $h$, namely $h_{ij}$.  Combining estimate ~\eqref{eq: no kernel est}, and Lemma~\ref{lem: gradient decay}, we have,
\begin{equation*}
\int_{0}^{2\pi} \int_{-\log r_{0}}^{\infty} |h_{ij}|^{2} dx dy \leq \frac{C}{-\log(r_{0})}.
\end{equation*}
Since $h$ is uniformly bounded in $C^{1,\alpha}(\log \D)$ an argument similar to the one just given implies that
\begin{equation*}
|h_{ij}(r,\theta)| \leq \frac{C}{(-\log r)^{1/4}}.
\end{equation*}
 Fix a local invariant subbundle $S$.  Let $H_{S}$ denote the restriction of $H$ to $S$, and $\Psi^{S} = \Psi(H_{S})$.  We also let $\Psi^{S}_{0} = \Psi(H_{0}|_{S})$.  Denote by $h_{S}:S \rightarrow S$ the map induced by $h$. By equation~\eqref{eq: psi relation}, we have
 \begin{equation*}
 \Psi^{S} = \Psi^{S}_{0} + h_{S}^{-1}\hat{\nabla}^{0}h_{S}.
 \end{equation*}
The above estimates combined with Lemma~\ref{lem: unitary formula} imply that there is a $\epsilon>0$ so that, in a unitary framing
\begin{equation*}
\hat{\nabla}^{0}h_{S}  = o\left(\frac{dr}{r|\log r|^{\epsilon}} + \frac{d\theta}{|\log r|^{\epsilon}}\right).
 \end{equation*}
It is a simple exercise in linear algebra that the upper bound for $h^{-1}$ implies an upper bound for $h_{S}^{-1}$.  This is not immediate, since $h$ may not preserve $S$.  Finally, since $H_{0}$ is conformally strongly tamed by the parabolic structure, the result follows.  Moreover, if $e^{\psi} \in L^{p}(\overline{M}, dV)$ for some $p>2$ then $H_{0}$ is strongly tamed, and we are done.
\end{proof}

One might hope for stronger regularity results than what we have obtained in Theorem~\ref{thm: regularity}.  The next simple example illustrates the borderline regularity of solutions of equation~\eqref{eq: endo equation}.

\begin{ex}
Again, we return to the setting considered in Examples~\ref{ex: 1} and~\ref{ex: 2}.  Define a section $ \sigma \in {\rm End}(E)$ by 
\begin{equation*}
\sigma = \begin{pmatrix}1 & 1 \\ 0 & 1 \end{pmatrix}
\end{equation*}
where everything is expressed in the frame $\{ e_{1}, e_{2}\}$ as before.  Then one can easily check that $\nabla \sigma =0$.  It follows immediately that if $H_{0}$ is the local model solution of Theorem~\ref{thm: H0}, given explicitly in Example~\ref{ex: 2}, then
$H = \sigma^{\dagger} H_{0} \sigma$ is also a local solution.  That is, the metric given in the frame $\{e_{1}, e_{2}\}$ by
\begin{equation*}
H = \begin{pmatrix} \frac{-1}{\log r} & \frac{-1}{\log r} \\ \frac{-1}{\log r} & - [\log r + \frac{1}{\log r}] \end{pmatrix}
\end{equation*}
is also Poisson on $\D\backslash \{0\}$.  One easily computes that in an $H_{0}$-unitary frame we have
\begin{equation*}
h := H_{0}^{-1}H = \begin{pmatrix} 1 & \frac{1}{-\log r} \\  \frac{1}{-\log r} &1+ \frac{1}{(\log r)^{2}}  \end{pmatrix}.
\end{equation*}
While this is continuous, and satisfies $\del_{r} h = o(1/r)$, it is not $C^{\alpha}$ for any $\alpha>0$ at the origin.
\end{ex}

It may be the case that solutions of~\eqref{eq: endo equation} are in fact continuous on $\D$ when expressed in an $H_{0}$ unitary frame, however, we have not been able to prove this optimal regularity result, except for the off diagonal terms $h_{ij}$.

\section{The Donaldson heat flow with boundary}\label{sec: don heat}

The remainder of this paper is devoted to constructing approximate solutions of the Poisson metric equation~\eqref{AHYM}.  We fix an initial metric $H_{0}$, as given by Theorem~\ref{thm: H0}, which is conformally tamed by the parabolic structure, and is Poisson on $B_{R}(p_{j})$ for each puncture $p_{j}$.  For every $r \leq R$ we set $\mathcal{U}_{r} = \cup_{j=1}^{m} B_{r}(p_{j})$ and define $M_{r} = M \backslash \overline{\mathcal{U}_{r}}$.  As a first step, we want to find a Hermitian metric $H_{r} \in C^{\infty}(M_{r}) \cap C^{0}(\overline{M_{r}})$ solving the boundary value problem
\begin{equation}\label{eq: BVP}
\left\{\begin{array}{rl}
-\frac12\star \nabla \star \Psi(H_{r}) = c\mathbb{I}, & \text{ on } M_{r}\\[8pt]
\det(H_{0}^{-1}H_{r}) =1 & \text{ on }M_{r}\\[8pt]
H_{r}|_{\del M_{r}} = H_{0}|_{\del M_{r}} & \\[8pt]
\end{array}
\right.
\end{equation}
This system is closely related to the boundary value problem for Hermitian-Einstein metrics on K\"ahler manifolds, which was studied by Donaldson \cite{Don5} using parabolic techniques, and as a result, much of Donaldon's work carries over with only minor adjustments.  In fact, much of what follows is valid on a general affine manifold with boundary and with more general boundary values.  Consider the parabolic equation
\begin{equation}\label{eq: PBVP}
\left\{\begin{array}{lr}
H^{-1}\pl_t H= -(K(H) - c\mathbb{I}), & \text{ on } M_{r}\\[8pt]
H(0) = H_{0} & \text{ on } M_{r}\\[8pt]
H(t)|_{\del M_{r}} = H_{0}. \\[8pt]
\end{array}
\right.
\end{equation}
Since the metric $\phi_{ij}$ is non-degenerate on $M_{r}$ the above system is parabolic, and hence a solution exists for short time by the general theory of parabolic equations.  As in \cite{Don5}, the long time existence of \eqref{eq: PBVP} follows from \cite{Simp}, with minor modifications for our current setting, and so we will omit the details.  In fact, even the convergence of the flow follows from the arguments of \cite{Don5}, and \cite{Simp}, but we will explain the main ingredients below.   As a first step we give a lemma which is analogous to a standard, but important result for the Donaldson heat flow.

\begin{lem}\label{lem: DHF curv}
Let $\square_{t} = \frac{1}{4} \phi^{ij} \hat{\nabla}^{H(t)}_{j}\nabla_{i}$. Then, along the flow~\eqref{eq: PBVP} the curvature $K$ satisfies
$(\pl_t- \square_{t}) K(t) = 0$.
In particular, $(\pl_t-\frac{1}{4}\Delta) |K(t)|^{2} \leq 0.$
\end{lem}
\begin{proof}
For simplicity we denote $H(t)$ by $H$. Beginning with $\Psi$, we work in a flat frame and compute:
\be
\pl_t\Psi_j=\frac{1}{2}\pl_t(H^{-1}\pl_jH)=\frac{1}{2}\left(-H^{-1}\pl_tHH^{-1}\pl_jH+H^{-1}\pl_j\pl_t H\right).\nonumber
\ee
Using the description of $\hat\nabla^H$ in a flat frame, a similar computation shows that the above expression equals $\frac12\hat\nabla^H_j(H^{-1}\pl_t H)$. As before, set $h=H_0^{-1}H$.
From the proof of Lemma \ref{lem: endo equation} it follows that
\be
\pl_t(\hat \nabla^H_jh\,h^{-1})=2\pl_t(\Psi_j-\Psi_j^0)=\hat\nabla^H_j(H^{-1}\pl_t H).\nonumber
\ee
Applying Lemma \ref{coordinate free},
 \be
 \label{evolvK}
 \pl_t K=\pl_t(K-K_0)=-\frac{1}{4}\star\nabla\star\left(\pl_t(\hat\nabla^Hh\,h^{-1})\right)=-\frac14\star\nabla\star\hat\nabla^H(H^{-1}\pl_t H).\nonumber
\ee
The definition of the flow~\eqref{eq: PBVP} now gives
\be
 \pl_t K=\frac14\star\nabla\star\hat\nabla^H(K)
\ee
By direct computation we have 
\be
\frac{1}{4}\phi^{jk}[\nabla_{k},\hat\nabla^H_j]K=\frac{1}{2}\phi^{jk}\nabla_k\Psi_j K-K\phi^{jk}\frac{1}{2}\nabla_k\Psi_j=-KK+KK=0,\nonumber
\ee
which implies we can switch the order of derivatives to obtain
\be
\left(\pl_t-\Box_t\right)K=0.
\ee
Applying the heat operator to $|K|^2$ and using the above equation proves the the lemma. 
\end{proof}

This result is important in the long-time existence and convergence of the flow~\eqref{eq: PBVP}, and we will use it in what follows.  A crucial ingredient in the convergence of~\eqref{eq: PBVP} is the following standard lemma; see, for example, \cite{Don5}.

\begin{lem}\label{lem: heat sub sol}
Suppose that $f \geq 0$ is a sub-solution of the heat equation on $M_{r} \times[0, \infty)$.  If $ f=0$ on $\del M_{r}$ for all time, then $f$ decays exponentially to zero, ie.
\begin{equation*}
\sup_{x\in M_{r}} f(x,t) \leq Ce^{-\epsilon t}
\end{equation*} 
where $\epsilon$ depends only on $M_{r}$ and $C$ depends only on $f(0)$.
\end{lem}

Suppose that $H(t)$ is a solution of~\eqref{eq: PBVP}. We apply the above lemma to the quantity $\mathcal{E} = |K - c\mathbb{I}|^{2}$.  By Lemma~\ref{lem: DHF curv} we see that $\mathcal{E}$ is a subsolution of the heat equation.  Since $H_{0}$ satisfies $K(H_{0}) = c\mathbb{I}$ on $\del M_{r}$, we see that $\mathcal{E}$ satisfies the hypothesis of Lemma~\ref{lem: heat sub sol}, and hence $\mathcal{E} \leq Ce^{-\epsilon t}$.  In particular, we have
\begin{equation*}
 \int_{0}^{t} \sqrt{\mathcal{E}(x,t)} dt \leq C
\end{equation*}
for a constant $C$ independent of $x$.  From this, the estimates of Simpson \cite{Simp} and Donaldson \cite{Don5} can be adapted to prove that $H(t)$  converges along a subsequence to a limiting metric $H_{\infty}$.  Since $\mathcal{E}$ decays exponentially,  $H_{\infty}$ solves $K(H_{\infty}) = c \mathbb{I}$, and $H_{\infty}|_{\del M_{r}} = H_{0}|_{\del M_{r}}$.  We claim that $\det( H_{0}^{-1}H_{\infty})=1$.  Assuming this claim, we have proved

\begin{thm}\label{thm: Mr solution}
For any $r \ll 1$, there exists a hermitian metric $\ti H_{r}$ on $E$, which is smooth on $M_{r}$ and continuous on $\overline{M_{r}}$ solving the system~\eqref{eq: BVP}.
\end{thm}

It suffices to prove the following lemma.

\begin{lem}
Let $H(t)$ be the solution of the flow~\eqref{eq: PBVP}, and let $h(t) := H_{0}^{-1}H(t)$ be the intertwining endomorphism.  Then we have $\det(h(t)) =1$.
\end{lem}
\begin{proof}
We compute
\begin{equation*}
\pl_t \log(\det h) = \Tr(h^{-1}\pl_t{h}) = - \Tr(K(t)-cI).
\end{equation*}
On the other hand, by Lemma~\ref{lem: DHF curv},
\begin{equation*}
\left(\pl_t -\Box_t\right) \Tr(K(t)-cI) =0.
\end{equation*}
Moreover, $\Tr(K(0) -cI) =0$ and $\Tr(K(t)-cI) |_{\del M_{r}} =0$ by the definition of $H_{0}$.  As a result, $\Tr(K(t)-cI) \equiv 0$ for all time, so
$\log (\det h(t)) = \log(\det h(0)) = 0$.
\end{proof}

\section{Constructing a limit and the Proof of Theorem~\ref{thm: main thm}}
\label{sec: main proof}

In Section \ref{sec: model solutions} we constructed a local solution inside of $B_R(p_j)$ for a small fixed $R$. As before, for $\rho \leq R$ we set $\mathcal{U}_{\rho} = \cup_{j=1}^{m} B_{\rho}(p_{j})$ and define $M_{\rho} = M \backslash \overline{\mathcal{U}_{\rho}}$. For every $\rho \leq R$ define an approximate solution to ~\eqref{AHYM} using the local model solution $H_0$ obtained in Theorem~\ref{thm: H0}, and the solution $\ti H_{\rho}$, defined on $M_{\rho}$ given by Theorem~\ref{thm: Mr solution}.  We set
\begin{equation*}
H_{\rho} := \left\{\begin{array}{lr}
\ti H_{\rho}, & \text{ on } M_{\rho}\\[8pt]
H_{0} & \text{ on } M\backslash M_{\rho}\\[8pt]
\end{array}
\right.
\end{equation*}

$H_{\rho}$ is continuous on $M$, and smooth on $M \backslash \del M_{\rho}$, and by definition, it is Poisson on $M \backslash \del M_{\rho}$.  Our goal is to take the limit as $\rho \rightarrow 0$, and show that $H_{\rho}$ converges to a smooth Poisson metric $H_{\infty}$ on all of $M$.  Moreover, we must establish that the limit $H_{\infty}$ is conformally tamed by the parabolic structure.  The estimate that makes all of this possible is a uniform upper bound for $H_{\rho}$ in terms of $H_{0}$.  We follow the general strategy of Uhlenbeck-Yau \cite{UY}.  Namely, we show that if no uniform upper bound exists, then $(E,\nabla, \Pi)$ contains a destabilizing subbundle.  In particular, if $(E,\nabla, \Pi)$ is stable, then we can take a limit to obtain a Poisson metric $H_{\infty}$.  Moreover, the upper bound allows us to apply the results of Section~\ref{sec: regularity} to conclude that $H_{\infty}$ is conformally strongly tamed by the parabolic structure, which establishes the main theorem.  All of this will be taken up in greater detail below.

Rather than working with $H_{\rho}$, it is more convenient to consider the positive, hermitian endormorphism $h_{\rho} := H_{0}^{-1}H_{\rho}$.  Note that, for any puncture $p$, if $\D$ denotes the disk around $p$, equipped with polar coordinates $(r,\theta)$, then $h_{\rho}$ is smooth as a function of $\theta$.  Moreover, $\det h_{\rho} \equiv 1$. 
One may wonder why we need to choose the local model solution $H_0$ as the boundary value for $H_{\rho}$ on $\del M_{\rho}$, as opposed to any initial metric. In fact, this choice of metric is fundamental, since it implies a weak comparison estimate for $H_{\rho}$ compared to $H_{0}$, a fact which is central in the estimates to follow. 

The main estimate in this section is

\begin{prop}\label{prop: C0 bound}
Let $\rho_{i}$ be any sequence in $(0,R)$ which is strictly decreasing with $\lim_{i\rightarrow \infty} \rho_{i} =0$.  Let $m_{i} := \sup_{M} \Tr(h_{\rho_{i}})$, and suppose that
\begin{equation*}
\lim_{i \rightarrow \infty} m_{i} = \infty
\end{equation*}
then $(E,\nabla,\Pi)$ is not stable.
\end{prop}

The proof of this proposition, which follows the outline of Uhlenbeck-Yau \cite{UY}, will occupy the bulk of this section.  The rough idea is the following: if the estimate does not hold, set $\ti h_{\rho_{i}} = m_{i}^{-1} h_{\rho_{i}}$.  Let us suppress the symbol $\rho$ in order to simplify notation.  Since $\ti h_{i}$ is a positive definite, hermitian endomorphism we can form the $H_{0}$-self-adjoint endormorphism $\ti h_{i}^{\sigma}$ for any $\sigma \in (0,1]$.  We then pass to the limit as $i \rightarrow \infty$ and $\sigma \rightarrow 0$.  The fundamental observation of Uhlenbeck-Yau is that this limit is a projection to a subbundle, and that this subbundle destabilizes $E$.  In order to make this argument rigorous, we need to prove several estimates for the endomorphisms $ \ti h_{i}^{\sigma}$.

Fix a point $x \in M\backslash \del M_{\rho_i}$ and choose local coordinates in a neighborhood of $x$.  Following Uhlenbeck-Yau \cite{UY} we have the following inequality 
\begin{equation}\label{eq: sigma grad est}
\phi^{\alpha\beta}\langle h_{i}^{-1} \hat{\nabla}^{0}_{\alpha} h_{i}, \hat{\nabla}^{0}_{\beta} h_{i}^{\sigma} \rangle_{H_{0}} \geq |h_{i}^{-\sigma/2}\hat{\nabla}^{0}h_{i}^{\sigma}|_{H_{0}\otimes \phi}^{2}
\end{equation} 
as well as the formula 
\begin{equation}\label{eq: delta h sigma}
\phi^{\alpha\beta} \del_{\beta}\langle h_{i}^{-1}\hat{\nabla}^{0}_{\alpha}h_{i}, h_{i}^{\sigma}\rangle_{H_{0}} = \phi^{\alpha\beta}\del_{\alpha} \Tr(h_{i}^{\sigma-1}\hat{\nabla}^{0}_{\beta}h_{i}) = \frac{1}{\sigma} \Delta_{\phi} \Tr(h_{i}^{\sigma}).
\end{equation}
Both of the above equations can be seen by computing locally in a frame where $h_i$ is diagonal; see \cite[Lemma 3.4.4]{TL} for details.  By Lemma~\ref{lem: endo equation}, for every point $x \in M \backslash \del M_{\rho_{i}}$ we have
\begin{equation*}
-\frac{1}{4} \star \nabla \star (h_{i}^{-1}\hat{\nabla}^{0}h_{i}) = c\mathbb{I} - K(H_{0}).
\end{equation*}
We take the inner product of the above equation with $h_{i}^{\sigma}$ and apply the product rule to see
\begin{equation*}
\begin{aligned}
\langle c\mathbb{I} -K(H_{0}), h_{i}^{\sigma}\rangle_{H_{0}} &= -\frac{1}{4}\phi^{\alpha\beta} \Tr\left(\nabla_{\beta}\left(h_{i}^{-1}\hat{\nabla}^{0}_{\alpha} h_{i}\right) h_{i}^{\sigma}\right)\\
&=-\frac{1}{4}\phi^{\alpha\beta}\del_{\beta}\Tr\left(h_{i}^{\sigma-1}\hat{\nabla}^{0}_{\alpha}h_{i}\right) + \frac{1}{4}\phi^{\alpha\beta}\langle h_{i}^{-1}\hat{\nabla}^{0}_{\alpha} h_{i}, \hat{\nabla}^{0}_{\beta}h_{i}^{\sigma}\rangle_{H_{0}}
\end{aligned}
\end{equation*}
We apply~\eqref{eq: sigma grad est} and~\eqref{eq: delta h sigma} to obtain

\begin{equation}\label{eq: h sigma key est}
\frac{1}{\sigma}\Delta_{\phi} \Tr(h_{i}^{\sigma}) \geq -4\langle c\mathbb{I}- K(H_{0}), h_{i}^{\sigma}\rangle_{H_{0}} +|h_{i}^{-\sigma/2}\hat{\nabla}^{0}h_{i}^{\sigma}|_{H_{0}\otimes \phi}^{2}.
\end{equation}

Let $m_{i}(\sigma) := \sup_{M}\Tr(h_{i}^{\sigma})$, so that $n^{-1}m_{i}^{\sigma} \leq m_{i}(\sigma) \leq nm_{i}^{\sigma}$.  Then we have the following key lemma.

\begin{lem}\label{lem: h sigma upper bound}
The function $\Tr(h_{i}^{\sigma})$ must achieve its maximum on $M_{R} = M \backslash \mathcal{U}_{R}$.
\end{lem}
\begin{proof}
The proof follows from the comparison principle.  Since $K(H_{0}) = c\mathbb{I}$ on $\mathcal{U}_R$, equation~\eqref{eq: h sigma key est} becomes
\begin{equation*}
\Delta_{\phi}\Tr(h_{i}^{\sigma}) \geq 0.
\end{equation*}
Rescaling by the conformal factor implies that $\Delta \Tr(h_{i}^{\sigma}) \geq 0$, where now the Laplacian is with respect to the local Euclidean metric.  Moreover, $\Tr(h_{i}^{\sigma}) =n$ on $\del M_{\rho_{i}}$ , thanks to the fact that $h_{i} = \mathbb{I}$ on $\del M_{\rho_{i}}$ by construction.  By the AM-GM inequality, we have
\begin{equation*}
\Tr(h_{i}^{\sigma}) \geq n \det(h_{i})^{\sigma/n} = n.
\end{equation*}
In local polar coordinates $(r,\theta)$ we set
\begin{equation*}
w(r) := \left(\frac{m_{i}(\sigma)-n}{\log(R)-\log(\rho_{i})}\right) \log(r) + \frac{n\log(R) - m_{i}(\sigma)\log(\rho_{i})}{\log(R)-\log(\rho_{i})}.
\end{equation*}
The function $w(r)$ is clearly harmonic on $\mathcal{U}_R \backslash \mathcal{U}_{\rho_{i}}$, and satisfies
\begin{equation*}
w|_{\del \mathcal{U}_{\rho_{i}}} = n, \qquad w|_{\del \mathcal{U}_R} = m_{i}(\sigma).
\end{equation*}
By the comparison principle $\Tr(h_{i}^{\sigma}) \leq w(r)$ on $\mathcal{U}_R \backslash \mathcal{U}_{\rho_{i}}$.  If $m_{i}(\sigma)=n$, then $\Tr(h_{i}(\sigma) )\equiv n$ on $\mathcal{U}_R \backslash \mathcal{U}_{\rho_{i}}$ and the lemma follows.  Otherwise, $m_{i}(\sigma)>n$, in which case the result follows from the fact that $w(r) < m_{i}(\sigma)$ for $r < R$.
\end{proof}

\begin{lem}\label{lem: int of Deltah}
Fix the real number $\sigma$ so that $0\leq\sigma\leq 1$. The integral of $\Delta{\rm Tr}(h^\sigma_i)$ over all of $M$ exists and is nonpositive, i.e.
\begin{equation*}
\int_{M} \Delta_{\phi}{\rm Tr}(h^\sigma_i) d\nu =\int_M\Delta_{\bar{g}} {\rm Tr}(h^\sigma_i)dV\leq 0.
\end{equation*}
\end{lem}
\begin{proof}
Throughout this lemma we use the Laplacian $\Delta_{\bar{g}}$, and suppress the subscript for convenience.  First let us comment that this estimate is obvious in the case that $h_{i}$ is $C^{1}$ on an open neighborhood of $M_{\rho_{i}}$.  To see this, integrate by parts and use that $\Delta {\rm Tr}(h^\sigma_i) \geq 0$ on $B_R \backslash B_{\rho_{i}}$, together with $\Tr(h_{i}^{\sigma}) =n = \inf_{M}\Tr(h_{i}^{\sigma})$ on $\del M_{\rho_{i}}$ to determine the sign of the boundary contribution.  Thus, the difficulty in this lemma is to determine the sign of the integral without assuming that a normal derivative exists.  

Break the integral into two pieces, one on $\mathcal{U}_{\rho_{i}}$ and the other on $M_{\rho_{i}}$:
\begin{equation*}
\int_M\Delta {\rm Tr}( h_{i}^\sigma)dV=\int_{\mathcal{U}_{\rho_{i}}}\Delta {\rm Tr}( h_{i}^\sigma)dV+\int_{M_{\rho_{i}}}\Delta {\rm Tr}(h_{i}^\sigma)dV.
\end{equation*}
The first integral on the right vanishes since $h_{i}$ is constant in $\mathcal{U}_{\rho_{i}}$. Consider the second integral on the right. To check this integral is well defined, note that $\Delta \Tr(h_{i}^\sigma)\geq 0$ in a neighborhood of $\pl M_{\rho_{i}}$, thanks to~\eqref{eq: h sigma key est}. Thus, for $\epsilon \ll1 $, the integral
\begin{equation*}
\int_{M_{\rho_{i}+\epsilon}}\Delta{\rm Tr}( h_{i}^\sigma)dV
\end{equation*}
is monotone increasing as $\epsilon \rightarrow 0$, so a limit exists in $(-\infty, \infty]$. By showing that the sequence is non-positive (the content of the lemma), we can conclude the limit is finite since the sequence is increasing and bounded above. To ease notation, set $f = {\rm Tr}(h_{i}^\sigma)$. First, we consider the special case when $f > n$ on $\mathcal{U}_R \backslash \mathcal{U}_{\rho_{i}}$.  Choose a sequence $\epsilon_{k}$ decreasing to $0$, such that $n+\epsilon_{k} < \inf_{\del M_{R}} f$.  Let $S_{\epsilon} = \{  f > n+\epsilon \}$.  By Sard's Theorem, $\del S_{\epsilon_{k}}$ is smooth submanifold of $\mathbb{R}^{2}$ for some sequence $\epsilon_{k} \rightarrow 0$. Let $N_{k,i}$ denote the connected component of $p_{i}$ in $S_{\epsilon_{k}}^{c}$, and set
\begin{equation*}
N_{k} = \cup_{i} N_{k,i}
\end{equation*}
By our choice of $\epsilon_{k}$, we know that $N_{k,j}^{c} \cap B_{\rho_{i}}(p_{j})$ has non-empty interior, and $N_{k}$ decreases to  $\mathcal{U}_{\rho_{i}}$, and $\del N_{k} \subset (\mathcal{U}_R \backslash \mathcal{U}_{\rho_{i}})^{o}$.  Since $\Delta f \geq 0$ on $\mathcal{U}_R \backslash \mathcal{U}_{\rho_{i}}$ we have that
\begin{equation*}
\int_{N_{k}^{c}}\Delta fdV \leq \int_{N^{c}_{k+1}}\Delta fdV.
\end{equation*}
Fix a point $p \in \del N_{k}$.  Since $S_{\epsilon_{k}}$ is open, and $\del S_{\epsilon_{k}}$ is a smooth curve in $\mathbb{R}^{2}$, we can find a small constant $\delta>0$ and a point $\hat{p} \in S_{\epsilon_{k}}$ such that $B := B_{\delta}(\hat{p}) \subset S_{\epsilon_{k}}$ and $\del B \cap \del S_{\epsilon_{k}} = \{p\}$.  To see this, choose coordinates $(x,y)$ on a small open set $U \subset \mathbb{R}^{2}$ such that $p = (0,0)$ and $\del S_{\epsilon_{k}} = \{y=0\}$ and such that $S_{\epsilon_{k}} \cap U \subset \{y>0 \}$, then it is straightforward to construct the ball $B$.  By shrinking $\delta$ if necessary we may assume that $B\Subset M_{\rho_{i}}$, so that $f$ is smooth in a neighborhood of $B$.  Then, on $B$ we have $f \geq n+\epsilon_{k}$, and $f(p) = n+\epsilon_{k}$ and hence, $ \nabla f \cdot \eta (p) <0$ where $\eta$ is the outward pointing normal vector of $\del N_{k}^{c}$ at $p$.  It follows that     
\begin{equation*}
\int_{N_{k}^{c}}\Delta f dV = \int_{\del N^{c}_{k}} \nabla f \cdot \eta \,dS \leq 0.
\end{equation*}
Hence, by the monotone convergence theorem we have
\begin{equation*}
\int_{M_{\rho_{i}}} \Delta f dV= \lim_{k\rightarrow \infty}\int_{N_{k}^{c}}\Delta f  dV\leq 0.
\end{equation*}

Now, in general, it is not true that $f>n$ on $ \mathcal{U}_R \backslash \mathcal{U}_{\rho_{i}}$, and so one cannot immediately apply the monotone convergence theorem .  In order to remedy this choose bump functions $\phi_{j}$ which are identically $1$ in a neighborhood of $B_{R}(p_{j})$, $0 \leq \phi_{j} \leq 1$ and have disjoint supports. Also, we consider the function 
\begin{equation*}
\psi(r,\theta) := \left\{
     \begin{array}{lr}
       (r-\rho_{i})^{2}, & r \geq \rho_{i} \\
       0  & r < \rho_{i}  \\
     \end{array}
   \right.
\end{equation*} 
defined in coordinate neighborhood of each puncture $p_{j}$. Note $\psi$ is subharmonic in $\mathcal{U}_R \backslash \mathcal{U}_{\rho_{i}}$. Set
\begin{equation*}
g_{m} = f + \frac{1}{m} \sum_{i} \phi_{i}\psi_{i}
\end{equation*}
On $\mathcal{U}_R \backslash \mathcal{U}_{\rho_{i}}$ we have that $g_m$ satisfies $\Delta g_{m}  = \Delta f+ m^{-1} \sum_{i}\Delta \psi_{i} \geq 0$, and $g_{m} > n$.  Then we can apply the previous argument to obtain
\begin{equation*}
\int_{M_{\rho_{i}}} \Delta g_{m}dV \leq 0,
\end{equation*}
from which it follows that
\begin{equation*}
\int_{M_{\rho_{i}}} \Delta f dV\leq -\frac{1}{m} dV\int_{M_{\rho_i}} \Delta(   \sum_{j} \phi_{j}\psi_{j} )dV.
\end{equation*}
But this holds for all $m \geq 0$.  Taking the limit as $m \rightarrow \infty$ proves the lemma.

\end{proof}

Suppose now that $m_{i}\rightarrow \infty$.  Then for each $\sigma \in (0,1]$, $m_{i}(\sigma) \rightarrow \infty$.  Set $\ti h_{i}^{\sigma} = m_{i}(\sigma)^{-1} h_{i}^{\sigma}$.  By inequality~\eqref{eq: h sigma key est} we have
\begin{equation*}
\begin{aligned}
\int_{M \backslash \del \mathcal{U}_{\rho_{i}}}\frac{1}{\sigma}\Delta_{\phi} \Tr(\ti h_{i}^{\sigma})d\nu \geq -4&\int_{M\backslash \del \mathcal{U}_{\rho_{i}}}\langle c\mathbb{I}- K(H_{0}), \ti h_{i}^{\sigma}\rangle_{H_{0}\otimes\phi}d\nu\\
& +\int_{M \backslash \del \mathcal{U}_{\rho_{i}}}|\ti h_{i}^{-\sigma/2}\hat{\nabla}^{0}\ti h_{i}^{\sigma}|_{H_{0}\otimes\phi}^{2}d\nu.
\end{aligned}
\end{equation*}
By Lemma~\ref{lem: int of Deltah} the right hand side of this equation is negative, and so
\begin{equation*}
\int_{M \backslash \del \mathcal{U}_{\rho_{i}}}|\ti h_{i}^{-\sigma/2}\hat{\nabla}^{0}\ti h_{i}^{\sigma}|_{H_{0}\otimes\phi}^{2}d\nu \leq C(H_{0})
\end{equation*}
where we have used the fact that $\sup_{M}\Tr(\ti h_{i}^{\sigma}) =n$.  Now, since $\ti h_{i} \leq \mathbb{I}$, we have $\ti h_{i}^{-\sigma/2} \geq \mathbb{I}$, and thus
\begin{equation}\label{eq: hi L21 bnd}
\int_{M \backslash \del \mathcal{U}_{\rho_{i}}}|\hat{\nabla}^{0}\ti h_{i}^{\sigma}|_{H_{0}\otimes\phi}^{2}d\nu \leq C(H_{0}).
\end{equation}

We would like to use this estimate to find a weak $W^{1,2}$ limit of $\ti h_{i}^{\sigma}$, but there are several details to address before this is possible.   First, we claim that $h_{i}^{\sigma}$ has a weak derivative.  Clearly it suffices to prove the claim for $h_{i}$.  Since $h_{i}$ is smooth on $M_{\rho_{i}}$ it suffices to prove that a weak derivative exists near $\del \mathcal{U}_{\rho_{i}}$.  Fix a puncture $p_{j}$ and local polar coordinates $(r,\theta)$.  Since $h_{i}$ is smooth as a function of $\theta$, we need only show that $\del_{r} h_{i}$ is well defined.  This follows easily from the fact that $h_{i}$ is continuous, by integration by parts. We leave the details to the reader.  We obtain
\begin{equation}\label{eq: L21 bound ti h sig}
\int_{M}|\hat{\nabla}^{0}\ti h_{i}^{\sigma}|_{H_{0}\otimes\phi}^{2}d\nu \leq C(H_{0}).
\end{equation}
We still cannot lean on the general theory of Hilbert spaces to take a weak limit since it is not clear that the space of hermitian endomorphisms equipped with the covariant derivative $\hat{\nabla}^{0}$ is complete with respect to the natural inner product induced by the metrics $\phi_{ij}$ and $H_{0}$.  Instead, we consider the ad hoc Hilbert space
\begin{equation*}
\mathcal{H} := W^{1,2}\big(\mathcal{U}_R,( \mathbb{C}^{n^{2}}, d, \bar{g}), dV\big) \oplus W^{1,2}\big(M_{R-\delta}, (E,\hat{\nabla}^{0}, H_{0}\otimes\phi), d\nu\big)
\end{equation*}
where the first space is the space of $\mathbb{C}^{n^{2}}$ valued functions on $\mathcal{U}_{R}$ equipped with the Euclidean metric and connection.  If $s= s_{1} \oplus s_{2} \in \mathcal{H}$ satisfies $s_{2} = s_{1}$ a.e. on $\mathcal{U}_{R}\backslash \mathcal{U}_{R-\delta} $ when $s_{2}$ is expressed in the $H_{0}$-unitary framing, then $s$ canonically defines an element of $L^{2}\left(M, (E \otimes E^{*}, H_{0}\otimes\phi)\right)$.  Moreover, by Proposition~\ref{prop: sob compare}, each $\ti h_{i}^{\sigma}$ defines an element of $\mathcal{H}$ after appropriate identifications using the $H_{0}$-unitary framing.  As a result, after passing to a subsequence (which we shall not relabel), we obtain a weak limit $h_{\infty}^{\sigma}$ in $\mathcal{H}$.  Since $H_{0}$ and $\hat{\nabla}^{0}$ are smooth on $M_{R-\delta}$ for $\delta \in (0, R)$, we can apply Rellich's Lemma to obtain the strong convergence of $h^{\sigma}_{i}$ to $h_{\infty}^{\sigma}$ in $L^{2}\left(M,(E\otimes E^{*},H_{0})\right)$, and hence $h^{\sigma}_{i}$ converges to $h_{\infty}^{\sigma}$ pointwise, almost everywhere.  As a result $\|h_{\infty}^{\sigma}\|_{L^{\infty}(M,(E\otimes E^{*},H_{0}))} \leq C$.  By the usual reflexivity of Hilbert spaces, we obtain the estimates

\begin{equation}\label{eq: weak lim int bnd}
\int_{M_R} |\hat{\nabla}^{0}h_{\infty}^{\sigma}|_{H_{0}\otimes\phi}^{2}d\nu \leq C(\sigma),\quad \int_{\mathcal{U}_{R}} |dh_{\infty}^{\sigma}|^{2} dV \leq C(\sigma),
\end{equation}
for a constant $C(\sigma)$ which may depend on $\sigma$.  We claim that in fact, $C$ can be taken to depend only on the initial metric $H_{0}$.  For example, by the weak convergence we have
\begin{equation*}
\begin{aligned}
\int_{\mathcal{U}_R} |dh_{\infty}^{\sigma}|^{2} dV &\leq \lim_{i\rightarrow \infty} \int_{\mathcal{U}_R} \langle dh_{\infty}^{\sigma}, d\ti h_{i}^{\sigma}\rangle_{H_{0}} dV + \int_{\mathcal{U}_R}\langle h_{\infty}^{\sigma}, \ti h_{i}^{\sigma}\rangle dV\\
&\leq  \lim_{i\rightarrow \infty}\left( \int_{\mathcal{U}_R} |dh_{\infty}^{\sigma}|^{2} dV\right)^{1/2}\left( \int_{\mathcal{U}_R} |d\ti h_{i}^{\sigma}|^{2} dV\right)^{1/2} + C.
\end{aligned}
\end{equation*}
Thanks to~\eqref{eq: L21 bound ti h sig}, this implies
\begin{equation*}
\int_{\mathcal{U}_R}|dh_{\infty}^{\sigma}|^{2} dV \leq C(H_{0}) \left( \int_{\mathcal{U}_R} |dh_{\infty}^{\sigma}|^{2} dV\right)^{1/2} +C(H_{0})
\end{equation*}
which clearly implies the claim.  A similar argument holds for the first integral in~\eqref{eq: weak lim int bnd}.  In particular, $\{h_{\infty}^{\sigma}\}_{\sigma \in (0,1]}$ defines a bounded sequence in $\mathcal{H}$, and so we may take a second weak limit, sending $\sigma \rightarrow 0$. Notice that again convergence is strong in $L^2$, and therefore the sequence converges pointwise almost everywhere. Thus, this limit can be viewed as taking the eigenvalues of $h^1_\infty$ to the power $\sigma$ as  $\sigma \rightarrow 0$. This implies $h^0_\infty$ is independent of a choice of subsequence as $\sigma \rightarrow 0$. 

Now, before we take the limit in $\sigma$, let us first show that the limit $h_{\infty}^{\sigma}$ is not identically zero.

\begin{lem}\label{lem: hinfty non degen}
There exists a positive constant $\delta = \delta(R, H_{0}, \sigma)$ so that, for each $i \in \mathbb{N}$ with $i \gg 0$ and $\sigma \in (0, 1]$ there holds
\begin{equation*}
\delta \leq  \int_{M} \Tr(\ti h_{i}^{\sigma})dV.
\end{equation*}
\end{lem}
\begin{proof}
Let $p_{i} \in M$ be a point where $\Tr(h_{i}^{\sigma})$ achieves its supremum.  Note that such a point exists, despite the fact the $M$ is non-compact, by Lemma~\ref{lem: h sigma upper bound}.  In fact, thanks to Lemma~\ref{lem: h sigma upper bound}, $p_{i} \in M_R$.  Thus, for $ i\gg0$ we have $\text{dist}(p_{i}, B_{\rho_{i}}) \geq \frac{R}{2}$.  Let $B = B_{\frac{R}{4}}(p_{i})$, which we identify with $B_{\frac{R}{4}}(0) \subset \mathbb{R}^{2}$.  By inequality~\eqref{eq: h sigma key est}
\begin{equation*}
\Delta_{\phi} \Tr(h_{i}^{\sigma}) \geq -C\sigma m_{i}(\sigma).
\end{equation*}
Moreover, since $B \subset M_{R/2}$, the conformal factor relating the metrics $\phi_{ij}$ and $\bar{g}_{ij}$ is uniformly bounded above and below by constants depending only on $R$, and so
\begin{equation*}
\Delta \Tr(\ti h_{i}^{\sigma}) \geq -C\sigma.
\end{equation*}
for a constant $C$ depending only on $H_{0}, R$.  The proof follows the elementary Lemma~\ref{lem: subharm a priori}, below.
\end{proof}

\begin{lem}\label{lem: subharm a priori}
Suppose $f$ is a $C^{2}$ function on $B_{r}(0)\subset \mathbb{R}^{2}$ satisfying 
 \begin{equation}
  \left\{
     \begin{array}{lr}
       a^{ij}\del_{i}\del_{j}f \geq -C, & \\
       0 \leq f \leq 1 &  \\
       f(0)=1
     \end{array}
   \right.
\end{equation} 
where the $a^{ij}$ are smooth and satisfy $\lambda \mathbb{I} \leq a^{ij} \leq \Lambda \mathbb{I}$.  Then there exists a constant $\delta = \delta(\lambda, C,r)>0$ such that
\begin{equation*}
\delta \leq \int_{B_{r}(0)} fdV.
\end{equation*}
\end{lem}
This estimate follows easily from the comparison principle. Now, Lemma~\ref{lem: hinfty non degen} implies that there are constants $C_{1}, C_{2}$ independent of $i, \sigma$ so that
\begin{equation*}
1 \leq C_{1} \int_{M} \Tr(\ti h_{i}^{\sigma})dV \leq C_{1} \|\ti h_{i}^{\sigma}\|_{L^{1}(M)} \leq C_{2} \|\ti h_{i}^{\sigma}\|_{L^{2}(M)}.
\end{equation*}
Since $\ti h_{i}^{\sigma}$ converges to $h_{\infty}^{\sigma}$ in $L^{2}(M)$, $h_{\infty}^{\sigma}$ is not identically $0$.
We now take the weak limit $h_{\infty}^{\sigma}$ as $\sigma \rightarrow 0$ to obtain $h_{\infty}^{0} \in \mathcal{H}$. Again by Rellich's lemma $h_{\infty}^{\sigma}$ converges to $h_{\infty}^{0}$ pointwise almost eveywhere, and clearly $h_{\infty}^{0}$ defines an element of $L^{2}\left(M,(E\otimes E^{*},H_{0}\otimes\phi)\right)$.  We set
\begin{equation*}
\pi := \mathbb{I} - h_{\infty}^{0}.
\end{equation*}
In order to finish the proof of Proposition~\ref{prop: C0 bound} it suffices to prove that $\pi$ defines a proper, smooth, $\nabla$ invariant subbundle of $E$, with $\mu(S) < \mu(E)$.  We take this up in the next two propositions.
\begin{prop}
Set $S := \pi(E)$.  Then $S$ is a smooth, $\nabla$ invariant subbundle of $E$.
\end{prop}
\begin{proof}
The proof is essentially due to Loftin \cite{Loft2}.  Note that the proof there (and here) is much simpler than the original argument of Uhlenbeck-Yau \cite{UY}, since on affine manifolds any $L^2_1$ subbundle is in fact smooth, as opposed to just being a torsion-free coherent subsheaf in the K\"ahler case.  To begin the proof, we first observe that $\pi^{\dagger} = \pi$, since each $\ti h_{i}^{\sigma}$ is $H_{0}$ self-adjoint for every $i, \sigma$, and the convergence of $\ti h_{i}^{\sigma}$ to $\pi$ is pointwise almost everywhere.  Similarly, we have
\begin{equation*}
\pi^{2} = \lim_{\sigma \rightarrow 0} \lim_{i\rightarrow \infty} (\mathbb{I} - \ti h_{i}^{\sigma})^{2} =\mathbb{I}+ \lim_{\sigma \rightarrow 0} \lim_{i\rightarrow \infty} (\ti h_{i}^{2\sigma}- 2\ti h_{i}^{\sigma})  = \mathbb{I}-h_\infty^0=\pi,
\end{equation*}
since the pointwise limit as $\sigma\rightarrow 0$ is independent of subsequence.

The key step is to show that $\pi$ is flat in the $L^{1}$ sense; that is 
\begin{equation*}
\|(\mathbb{I}-\pi) \nabla \pi \|_{L^{1}(M,(E \otimes E^{*}, H_{0}\otimes\phi))} = 0
\end{equation*}
Note that is is not even clear, a priori, that $\nabla \pi$ is integrable.  Since $\pi$ is $H_{0}$-self-adjoint, we have the pointwise identity
\begin{equation*}
\big|(\mathbb{I}-\pi) \nabla \pi^{\,}\big|_{H_{0}\otimes \phi} = \left| \left((\mathbb{I}-\pi) \nabla \pi\right)^{\dagger} \right|_{H_{0}\otimes \phi} = \left|\pi \hat{\nabla}^{0}(I- \pi)\right|_{H_{0}\otimes \phi},
\end{equation*}
and so it suffices to prove that this last quantity is zero almost everywhere.  In order to prove this statement we observe that the eigenvalues of $\ti h_{i}^{\sigma}$ lie in the open interval $(0,1)$.  For any real numbers $0 \leq \lambda \leq 1$ and $0 <s \leq \kappa \leq 1$, it holds that (\cite{TL}, page 87)
\begin{equation*}
0 \leq \frac{s+\kappa}{s}(1-\lambda^{2}) \leq \lambda^{-\kappa}.
\end{equation*}
Working in an orthonormal frame for $\ti h_{i}$ it follows that for any $0 \leq s \leq \sigma/2 \leq 1$
\begin{equation*}
0 \leq \frac{s+ \frac{\sigma}{2}}{s}(\mathbb{I}- \ti h_{i}^{s}) \leq \ti h_{i}^{-\sigma/2},
\end{equation*}
and so
\begin{equation*}
\begin{aligned}
\int_{M}|(\mathbb{I}- \ti h_{i}^{s})\hat{\nabla}^{0} \ti h_{i}^{\sigma}|^{2}_{H_0\otimes\phi} d\nu &\leq \left(\frac{2s}{2s+\sigma}\right)^{2} \int_{M}|\ti h_{i}^{-\sigma/2}\hat{\nabla}^{0} \ti h_{i}^{\sigma}|^{2}_{H_0\otimes\phi} d\nu\\
&\leq \left(\frac{2s}{2s+\sigma}\right)^{2} C(H_{0})
\end{aligned}
\end{equation*}
where the last line follows from~\eqref{eq: hi L21 bnd}.  It follows that for each $0 \leq s \leq \sigma/2 \leq 1$, the sequence $\{ (\mathbb{I}- \ti h_{i}^{s})\hat{\nabla}^{0} \ti h_{i}^{\sigma}\}_{i \in \mathbb{N}}$ is bounded in the Hilbert space $L^{2}(M, E\otimes E^{*} \otimes TM^{*}, H_{0} \otimes\phi, d\nu)$, and so weak compactness implies
\begin{equation*}
\int_{M}|(\mathbb{I}- \ti h_{\infty}^{s})\hat{\nabla}^{0} \ti h_{\infty}^{\sigma}|^{2}_{H_0\otimes \phi}d\nu \leq \left(\frac{2s}{2s+\sigma}\right)^{2} C(H_{0}).
\end{equation*}
Now we take a limit as $s \rightarrow 0$.  Since $(\mathbb{I}- \ti h_{\infty}^{s})$ converges strongly to $\pi$ in $L^{2}(M, E\otimes E^{*}, H_{0} \otimes \phi^{ij}, d\nu)$, H\"older's inequality implies
\begin{equation*}
\int_{M}|\pi \hat{\nabla}^{0} \ti h_{\infty}^{\sigma}|^2_{H_0\otimes \phi}d\nu  = 0.
\end{equation*}
In particular, $\pi \hat{\nabla}^{0} \ti h_{\infty}^{\sigma} =0$ almost everywhere.  Since $\pi \hat{\nabla}^{0} \ti h_{\infty}^{\sigma}$ converges weakly to $\pi \hat{\nabla}^{0}(\mathbb{I}-\pi)$ we have
\begin{equation*}
\int_{M}|\pi \hat{\nabla}^{0}(\mathbb{I}-\pi)|^2_{H_0\otimes \phi}d\nu = 0.
\end{equation*}
It remains only to prove that $\pi$ is smooth. Clearly this is a local matter.  As a result, the argument in \cite{Loft2} carries over verbatim to prove that $\pi$ is smooth. We omit the details. 
\end{proof}

\begin{prop}
$S \subset E$ is a proper subbundle of $E$, with $\mu(S) < \mu(E)$.
\end{prop}
\begin{proof}
We must show that $S$ is a proper, non-trivial subbundle of $E$, and that $\mu(S) \geq \mu(E)$.  We begin by proving that $S$ is non-trivial.  First, we know that $h_{\infty}^{\sigma} \ne 0$ for any $\sigma >0$.  Since $h_{\infty}^{\sigma}$ converges pointwise to $\pi$, we clearly have $\pi \ne 0$, and so $ rk(S)  =  rk(\mathbb{I} - \pi) < n = rk(E)$, so $S \ne E$.  Moreover, since $\det h_{i} :=1$, it follows that $\det \ti h_{i}^{\sigma} \rightarrow 0$ uniformly as $ i \rightarrow \infty$, and so $h_{\infty}^{\sigma}$ has a zero eigenvalue at almost every point of $M$, which implies that ${\rm rk}(S) >0$.   It remains only to show that $\mu(S) \geq \mu(E)$.  Since $\pi^{2}=\pi$ and $\pi^{\dagger_{H_{0}}} = \pi$, $\pi$ is necessarily the orthogonal projection to $S$ with respect to the metric $H_{0}$.  Moreover, since $S$ is $\nabla$ invariant, near a puncture $p_{j}$, $S$ is a direct sum of local invariant subbundles (see Definition \ref{def: local flag}). Using the local model of Proposition \ref{prop: local model 1}, the second fundamental form satisfies 
\begin{equation}\label{eq: Ho beta}
\int_{B_{\rho}} |\hat\nabla^0 \pi|^{2}_{H_{0}\otimes \phi}  d\nu \leq \frac{C}{-{\rm log} \rho}.
\end{equation}
This can easily be seen by combining Remark~\ref{rk: local form for sub} and Lemma~\ref{lem: unitary formula} and computing explicitly. 

 We apply the Chern-Weil formula of Proposition~\ref{prop: Chern-Weil} with the metric $H_{0}$. Let $c = \frac{2\pi}{{\rm Vol}(M, d\nu)} \mu(E)$  and suppress the subscript $0$ for convenience. We have
 \begin{equation*}
\begin{aligned}
\mu(S) =& \frac{1}{2\pi{\rm rk}(S)}\int_{M} \Tr\left(\pi K(H)\pi -c \mathbb{I}_{S}\right) d\nu \\
&- \frac{1}{8\pi{\rm rk}(S)} \int_{M} |\hat{\nabla}^{0} \pi|^{2}_{H_{0}\otimes \phi} d\nu + \mu(E),
\end{aligned}
\end{equation*}
 and so it suffices to show that 
\begin{equation*}
\int_{M} \Tr\left(\pi K(H)\pi -c \mathbb{I}_{S}\right) d\nu \geq  \frac{1}{4} \int_{M} |\hat{\nabla}^{0} \pi|^{2}_{H_{0}\otimes \phi} d\nu
\end{equation*}
in order to verify that $S$ is destabilizing. 
Recall that $\Tr(K_{0} - c \mathbb{I})=0$.  Since $\lim_{\sigma \rightarrow 0} \lim_{i\rightarrow \infty}( \mathbb{I} - \ti h_{i}^{\sigma}) = \pi$ strongly in $L^{2}$, we have
\begin{equation*}
\int_{M} \Tr\left(\pi K(H)\pi -c \mathbb{I}_{S}\right) d\nu = - \lim_{\sigma \rightarrow 0} \lim_{i\rightarrow \infty} \int_{M} \Tr\left(K(H) -c \mathbb{I})\ti h_{i}^{\sigma}\right) d\nu.
\end{equation*}
By equation~\eqref{eq: h sigma key est}, we have
\begin{equation*}
-4\int_{M} \Tr\left(K(H) -c \mathbb{I})\ti h_{i}^{\sigma}\right) d\nu \geq -\frac{1}{\sigma}\int_{M}\Delta_{\phi}\Tr(\ti h_{i}^{\sigma})d\nu + \int_{M}|\ti h_{i}^{-\sigma/2}\hat{\nabla}^{0} \ti h_{i}^{\sigma}|_{H_{0}\otimes \phi}^{2} d\nu
\end{equation*}
By Lemma~\ref{lem: int of Deltah}, the first term on the right is non-negative, and so
\begin{equation*}
\begin{aligned}
-\int_{M} \Tr\left((K(H) -c \mathbb{I})\ti h_{i}^{\sigma}\right) d\nu &\geq \frac{1}{4}\int_{M}|\ti h_{i}^{-\sigma/2}\hat{\nabla}^{0} \ti h_{i}^{\sigma}|_{H_{0}\otimes \phi}^{2} d\nu\\
&\geq \frac{1}{4}\int_{M}|\hat{\nabla}^{0}(\mathbb{I}- \ti h_{i}^{\sigma})|_{H_{0}\otimes \phi}^{2} d\nu.
\end{aligned}
\end{equation*}
In order to conclude the proof we take a weak limit. First, by weak convergence in $\mathcal{H}$, for any fixed radius $\rho$ the function $(\mathbb{I}- \ti h_{i}^{\sigma})$ converges weakly to $\pi$ in $W^{1,2}(M_\rho, (E, \hat{\nabla}^{0}, H_{0}\otimes \phi), d\nu)$. This follows from the trivial observation that once $\rho$ is fixed, all the connection coefficients and metrics are smooth and uniformly bounded on $M_\rho$. Thus, by lower semi-continuity of weak limits we have
\begin{equation*}
\lim_{\sigma \rightarrow 0} \lim_{i\rightarrow 0} \int_{M_{\rho}}|\hat{\nabla}^{0}(\mathbb{I}- \ti h_{i}^{\sigma})|_{H_{0}\otimes \phi}^{2} d\nu \geq \int_{M_{\rho}} |\hat{\nabla}^{0} \pi|^{2}_{H_{0}\otimes \phi} d\nu.
\end{equation*}
for every $\rho >0$.  Yet, by ~\eqref{eq: Ho beta}, the integral of $ |\hat{\nabla}^{0} \pi|^{2}$ on $M_\rho$ differs from the integral over all of $M$ by $\frac{C}{-{\rm log} \rho}$, which is arbitrarily small. Thus we can conclude
\be
\int_{M} \Tr\left(\pi K(H)\pi -c \mathbb{I}_{S}\right) d\nu\geq   \frac{1}{4} \int_{M} |\hat{\nabla}^{0} \pi|^{2}_{H_{0}\otimes \phi} d\nu-\epsilon,\nonumber
\ee
for any $\epsilon>0$.  Alternatively, one can apply Proposition~\ref{prop: sub bun deg} and the Dominated Convergence Theorem. This implies the result. 
\end{proof}

Having established Proposition~\ref{prop: C0 bound}, we turn to the proof of the main theorem
\begin{proof}[Proof of Theorem~\ref{thm: main thm}]
We can assume that $E$ is stable, for if $E$ is polystable then it suffices to consider each stable factor separately.  Let $\rho_{i}$ be any sequence in $(0,R)$ which is strictly decreasing with $\lim_{i\rightarrow \infty} \rho_{i} =0$. As in the beginning of this section we define endomorphisms $h_i:=h_{\rho_i}$, which satisfy equation \eqref{eq: delta h sigma} on $M\backslash \del M_{\rho_i}$. Because $E$ is stable Proposition~\ref{prop: C0 bound} implies the sequence $h_{i}$ has a uniform $C^{0}$ bound.  Since $\det h_{i} =1$, it follows that $h_{i}^{-1}$ has a uniform $C^{0}$ bound as well.  For every point $x \in M$ there exists a radius $\rho_N$ from the sequence above so that $x\in M_{\rho_{N}}$. Consider a coordinate ball $x \in B \Subset M_{\rho_{N+1}}$, and a flat frame on $B$.  Then in $B$ we have
\begin{equation*}
-\frac{1}{4\sqrt{\det(\phi)}}\del_{\alpha}\left(\sqrt{\det(\phi)} \phi^{\alpha \beta}h_{i}^{-1}\del_{\beta}h_{i}\right) = c\mathbb{I} - K(H_{0})
\end{equation*}
for each $i \geq N+1$.  Since $B \Subset M_{\rho_{N+1}}$, the metrics $\phi_{ij}, H_{0}$ are smooth and have uniformly bounded geometry.  Since $x$ is fixed distance from the puncture, following \cite[Proposition 1]{BaS}, $h_{i}$ is uniformly bounded in $C^{1,\alpha}$ for some $\alpha >0$ independent of $i$.  The higher order regularity follows easily by bootstrapping.  For example,  expanding above equation we have
\begin{equation*}
\phi^{\alpha\beta} \del_{\alpha}\del_{\beta} h_{i} = F_{i}
\end{equation*}
for a matrix valued function $F_{i}$ which is uniformly bounded in $C^{\alpha}$, independent of $i$. The standard Schauder theory implies that $h_{i}$ is uniformly bounded in $C^{2,\alpha}$ for some $\alpha \in (0,1)$. Bootstrapping then implies that $h_{i}$ is uniformly bounded in $C^{\infty}$, independent of $i$.  By passing to a subsequence, we obtain a smooth limit $h_{\infty,N}$ on $B \Subset M_{\rho_{N+1}}$.  The metric $H_{\infty, N} := H_{0}h_{\infty,N}$ clearly satisfies $K(H_{\infty, N}) = c \mathbb{I}$ on $M_{\rho_{N}}$.   One then repeats this argument for a sequence $N \rightarrow \infty$ to obtain a smooth, positive definite hermitian limit $h_{\infty}$ on all of $M$, which is bounded above and below.  Combining the $C^{0}$ bound with equation~\eqref{eq: L21 bound ti h sig}, it follows that $h_i$ is in $\mathcal{H}$. By Fatou's Lemma the smooth limit $h_\infty$ satisfies
\begin{equation*}
\int_{M} |\hat{\nabla}^{0}h_{\infty}|^{2}_{H_0\otimes\phi}d\nu \leq C.
\end{equation*}
We can thus apply Theorem~\ref{thm: regularity} to conclude that the metric $H_{\infty} := H_{0}h_{\infty}$ is conformally tamed by the parabolic structure.  The asymptotics for $H$ follow immediately from the upper bound for $h_{\infty}, h_{\infty}^{-1}$ and the explicit formula for $H_{0}$ near the punctures.   The proof is complete.
\end{proof}

Now that we have established existence of a smooth Poisson metric satisfying on $M$, we conclude with a short proof of uniqueness.

\begin{thm}[Uniqueness]
Let $M$ be a punctured Riemann surface, equipped with a K\"ahler metric metric $\phi_{ij}$ with finite volume. Suppose that $(E,\nabla, \Pi) \rightarrow M$ is a flat vector bundle with a parabolic structure.  If $E$ admits two Poisson metrics $H_1$ and $H_2$ and a function $u \in C^{\infty}(M, \mathbb{R}) \cap L^{2}(\overline{M}, dV)$ solving $\Delta_{\bar{g}} u =f$ in the distributional sense on $(\overline{M},\bar{g})$, for a function $f \in C^{\infty}(M, \mathbb{R})\cap L^{1}(\overline{M}, dV)$ so that $e^{-u}H_{1}$ and $e^{-u}H_{2}$ are strongly tamed by $\Pi$ then $H_1=\lambda H_2$ for some positive real number $\lambda$. 

\end{thm}
\begin{proof}
Set $h=H_1^{-1}H_2$. Applying Lemma \ref{lem: endo equation} we see
\be
-\frac14\star\nabla\star(h^{-1}\hat\nabla^{1} h)=c\mathbb{I}-c\mathbb{I}=0.\nonumber
\ee
Expanding this equation and taking the trace, we have
\begin{equation*}
\Delta_{\bar{g}} \Tr(h) = |h^{-1/2} \hat{\nabla}^{H_{1}} h|^{2}_{H_{1}\otimes \bar{g}}
\end{equation*}
We now integrate over $M_{\rho} = M \backslash \mathcal{U}_{\rho}$. Applying Stokes' Theorem we have
\begin{equation*}
\int_{M_{\rho}} \Delta_{\bar{g}} \Tr(h) dV = \sum_j\rho\int_0^{2\pi}{\rm Tr}\,(\nabla_{r}h)(\rho ,\theta)d\t,
\end{equation*}
where the right hand integral is over $\del B_{\rho}(p_{j})$. Since both $e^{-u}H_1$ and $e^{-u}H_2$ are strongly tamed by $\Pi$, condition (D) of Defintion~\ref{def: tame} implies that
\begin{equation}\label{eq: zero}
h^{-1}\hat\nabla^{H_{1}}_rh(\rho, \theta)=\Psi_1(\frac{\pl}{\pl r})-\Psi_2(\frac{\pl}{\pl r})=o\left(\frac 1{\rho|{\rm log}(\rho)|^\epsilon}\right)
\end{equation}
for some $\epsilon>0$, while condition (C) implies that $h, h^{-1}$ are uniformly bounded above.  In particular, we have
$\Tr(\nabla_{r} h)(\rho, \theta) = o\left(\frac 1{\rho|{\rm log}(\rho)|^\epsilon}\right)$
As a result, we obtain
\begin{equation*}
\lim_{\rho\rightarrow 0} \int_{M_{\rho}}  |h^{-1/2} \hat{\nabla}^{H_{1}} h|^{2}_{H_{1}\otimes \bar{g}} d\nu =0.
\end{equation*}
Since $h$ is positive definite, we must have $|\nabla h| = | \hat{\nabla}^{H_{1}} h| =0$. It follows that $h$ is self-adjoint and flat.  Applying Proposition \ref{prop: simple} completes the proof.

\end{proof}

\section{Appendix}

Our goal is to prove a priori interior estimates for bounded solutions of the Poisson metric equation.  These estimates are similar to the regularity results of Hildebrandt for harmonic maps \cite{Hil}.  Our considerations are local, so we restrict our attention to the case of $B_{1} \subset \mathbb{R}^{2}$, and prove interior estimates.  These estimates imply Proposition~\ref{prop: C1a estimate}.

\begin{prop}\label{prop: C1a estimate app}
Suppose $h(x,y) \in C^{\infty}(B_{1})\cap L^{\infty}(B_{1})$ is a hermitian matrix valued function solving
\begin{equation*}
\delta^{ij}\nabla_{i}(h^{-1} \hat{\nabla}^{0}_{j}h) =0
\end{equation*}
for differential operators
\begin{equation*}
\nabla_{i} = d + \Gamma, \qquad \hat{\nabla}^{0} = d+ \Gamma^{0} = d -\Gamma^{\dagger}
\end{equation*}
where $\Gamma$ is smooth. Then there exists constants $C, \alpha>0$ depending only on $\|h\|_{L^{\infty}(B_{1})}$, $\|h^{-1}\|_{L^{\infty}(B_{1})}$, $\| dh\|_{L^{2}(B_{1})}$ and the $C^{2}(B_{1})$ norm of $\Gamma$  so that
\begin{equation*}
\|h\|_{C^{1,\alpha}(B_{\frac{1}{2}})} \leq C.
\end{equation*}
\end{prop}
The first step is to establish an a priori $C^{\alpha}$ estimate, in terms of the $L^{\infty}$ norm. 

\begin{lem}\label{lem: Ca estimate}
In the setting of Proposition~\ref{prop: C1a estimate app}, there exists constants $C >0$, and $\alpha \in (0,1)$ depending only $\|h \|_{L^{\infty}(B_{1})}$ and $\|h^{-1}\|_{L^{\infty}(B_{1})}$ and the $C^{2}(B_{1})$ norm of $\Gamma$ such that
\begin{equation*}
\|h\|_{C^{\alpha}(B_{1/2})} \leq C
\end{equation*}
\end{lem}
\begin{proof}
Fix a point $x_{0} \in B_{1/2}$, and let $B_{\rho}$ denoted the ball $B_{\rho}(x_{0})$.  Define
\begin{equation*}
 \hbar = \frac{1}{3\pi\rho^{2}} \int_{B_{2\rho}\backslash B_{\rho}} h dV
 \end{equation*}
 and let $\hat{K} := K(\hbar)$ be the curvature of the metric induced by $\hbar$.  Note that $\hbar, \hbar^{-1}$ are bounded by $\|h\|_{L^{\infty}}$ and $\|h^{-1}\|_{L^{\infty}}$ respectively. In the argument to follow, all constants will only depend on the stated data. By Lemma~\ref{lem: endo equation}, we have
\begin{equation*}
-\frac{1}{4}e^{-\psi}\delta^{jk}\nabla_{j}\left((\hbar^{-1}h)^{-1} \hat{\nabla}^{\hbar}_{k}\hbar^{-1}h\right) = K(h) - \hat{K}
\end{equation*}
Expanding this and taking the trace yields
\begin{equation*}
\begin{aligned}
-\frac{1}{4}e^{-\psi}\Delta\Tr(\hbar^{-1}h) &+ e^{-\psi}\delta^{jk}\Tr\left(\nabla_{j}(\hbar^{-1} h)(\hbar^{-1}h)^{-1}\hat{\nabla}^{\hbar}_{k}(\hbar^{-1}h)\right)\\
 &= \Tr\left(\hbar^{-1}h(K(h)-\hat{K})\right)
 \end{aligned}
\end{equation*}
The right hand side can be expanded, since $K(h) = cI= K(H_{0})$. A second application of  Lemma~\ref{lem: endo equation} with respect to the metrics $H_0$ and $\hbar$ yields  
\begin{equation*}
cI -\hat{K} = -\frac{1}{4}e^{-\psi}\delta^{jk}\nabla_{j}\left((\hat{\nabla}^{0}_{k}\hbar^{-1})\hbar\right).
\end{equation*}
As a result, 
\begin{equation*}
\begin{aligned}
-\frac{1}{4}e^{-\psi}\Delta\Tr(\hbar^{-1}h) &+ e^{-\psi}\delta^{jk}\Tr\left(\nabla_{j}(\hbar^{-1} h)(\hbar^{-1}h)^{-1}\hat{\nabla}^{\hbar}_{k}(\hbar^{-1}h)\right)\\
 &=  -\frac{1}{4}e^{-\psi}\delta^{jk}{\rm Tr}\left(\hbar^{-1}h\,\,\nabla_{j}\left((\hat{\nabla}^{0}_{k}\hbar^{-1})\hbar\right)\right).
 \end{aligned}
\end{equation*}
The coefficients of $\nabla, \hat{\nabla}^{0}$ are uniformly bounded in $C^{2}(B_{1})$, and thus after multiplying both sides of the above equation by $e^{\psi}$ we can conclude
\begin{equation*}
\Delta\Tr(\hbar^{-1}h) \geq 4\delta^{jk}\Tr\left(\nabla_{j}(\hbar^{-1} h)(\hbar^{-1}h)^{-1}\hat{\nabla}^{\hbar}_{k}(\hbar^{-1}h)\right) - C.
\end{equation*}
 Using the equation for $\hat{\nabla}^{\hbar}$, we easily obtain
\begin{equation*}
\delta^{jk}\Tr\left(\nabla_{j}(\hbar^{-1} h)(\hbar^{-1}h)^{-1}\hat{\nabla}^{\hbar}_{k}(\hbar^{-1}h)\right) \geq \delta^{jk}\Tr\left(\del_{j}(\hbar^{-1} h)(\hbar^{-1}h)^{-1}\del_{k}(\hbar^{-1}h)\right) - C.
\end{equation*}
Finally, we have
\begin{equation*}
\begin{aligned}
\delta^{jk}\Tr\left(\del_{j}(\hbar^{-1} h)(\hbar^{-1}h)^{-1}\del_{k}(\hbar^{-1}h)\right) &= |(\hbar^{-1}h)^{-1/2}d(\hbar^{-1}h)|^{2} \geq c|d(\hbar^{-1}h)|^{2}\\
&\geq c'|dh|^{2}.
\end{aligned}
\end{equation*}
  One can argue identically to prove a similar estimate for $h^{-1}\hbar $.  Set $\sigma(h,\hbar) = \Tr(\hbar^{-1}h + h^{-1}\hbar)-2n$.  We then conclude
\begin{equation}\label{eq: Ca key}
\Delta \sigma(h,\hbar) \geq c|dh|^{2} -C.
\end{equation}
 Let $\xi$ be a smooth, radially symmetric cut-off function which is identically $1$ in $B_{\rho}$ identically zero outside $B_{2\rho}$ and satisfies $|\del_{r} \xi| \leq 10\rho^{-1}$, $|\del_{r}^{2}\xi| \leq 10\rho^{-2}$. Thanks to the estimate~\eqref{eq: Ca key}, we have
\begin{equation*}
\int_{B_{\rho}} |dh|^{2} \leq \int_{B_{2\rho}}\xi |dh|^{2} \leq C\int_{B_{2\rho}}\xi \Delta \sigma(h,\hbar) + C\rho^{2}
\end{equation*}
Since $\xi$ has compact support in $B_{2\rho}$ we can integrate the Laplacian by parts to obtain
\begin{equation*}
\int_{B_{2\rho}}\xi \Delta \sigma(h,\hbar) = \int_{B_{2\rho} \backslash B_{\rho}} \sigma(h,\hbar)\Delta \xi \leq \frac{10}{\rho^{2}} \int_{B_{2\rho} \backslash B_{\rho}} \sigma(h,\hbar).
\end{equation*}
Now, there is a constant $C$, again only depending on the $L^{\infty}$ bounds for $h, h^{-1}$, so that $\sigma(h,\hbar) \leq C |h-\hbar|^{2}$, (see, for example, \cite{Don1}) and so
\begin{equation*}
\int_{B_{\rho}}|dh|^{2} \leq \frac{C}{\rho^{2}}\int_{B_{2\rho}\backslash B_{\rho}}|h-\hbar|^{2} + C\rho^{2}.
\end{equation*}
By the Poincar\'e inequality we obtain
\begin{equation*}
\int_{B_{\rho}}|dh|^{2} \leq \frac{C}{C+1}\int_{B_{2\rho}}|dh|^{2} + \frac{C}{C+1} \rho^{2}.
\end{equation*}
That this estimate implies the lemma is a standard result in the elliptic theory; see, for instance, \cite{GT,HL}.
\end{proof}

We can now prove the Proposition.

\begin{proof}[Proof of Proposition~\ref{prop: C1a estimate app}]
As in the proof of Lemma~\ref{lem: Ca estimate}, fix $x_{0} \in B_{1/2}$, and let $B_{\rho}$ denote the ball of radius $\rho$ about $x_{0}$.  Recall that $h$ is a solution of the equation
\begin{equation*}
\delta^{ij}\nabla_{i}(h^{-1}\hat{\nabla}^{0}_{j} h) =0.
\end{equation*}
Define $\hbar := h(x_{0})$.  Multiplying the above equation by $\hbar$ we have
\begin{equation*}
\delta^{ij}\nabla_{i}(\hbar h^{-1}\hat{\nabla}^{0}_{j} h) =-\delta^{ij}[\Gamma_{i}, \hbar]h^{-1} \hat{\nabla}^{0}_{j}h.
\end{equation*}
Let $k \in L^{\infty}(B_{\rho}) \cap W^{1,2}(B_{\rho})$ be a hermitian matrix valued function with compact support in $B_{\rho}$.  Multiplying the above equation by $k$ and integrating we have
\begin{equation}\label{eq: C1a ineq 1}
\begin{aligned}
\int_{B_{\rho}}\delta^{ij} \Tr\left(\hbar h^{-1} \hat{\nabla}^{0}_{j}h (\hat{\nabla}^{0}_{i} k )^{\dagger}\right) dV &= \int_{B_{\rho}} \delta^{ij} \Tr\left([\Gamma_{i}, \hbar]h^{-1} (\del_{j}h)k\right)dV\\ &+ \int_{B_{\rho}} \delta^{ij} \Tr\left([\Gamma_{i}, \hbar]h^{-1} [\hat{\Gamma}^{0}_{j},h]k\right) dV.
\end{aligned}
\end{equation}
Write $\hbar h^{-1} = \mathbb{I} + (\hbar h^{-1} - \mathbb{I})$, and express the left hand side above as
\begin{equation}\label{eq: C1a ineq 3}
\begin{aligned}
\int_{B_{\rho}}\delta^{ij} \Tr\left(\hbar h^{-1} \hat{\nabla}^{0}_{j}h (\hat{\nabla}^{0}_{i} k )^{\dagger}\right) dV &= \int_{B_{\rho}}\delta^{ij} \Tr\left(\hat{\nabla}^{0}_{j}h (\hat{\nabla}^{0}_{i} k )^{\dagger}\right) dV\\&+\int_{B_{\rho}}\delta^{ij} \Tr\left((\hbar h^{-1}-\mathbb{I}) \hat{\nabla}^{0}_{j}h (\hat{\nabla}^{0}_{i} k )^{\dagger}\right) dV.
\end{aligned}
\end{equation}
By Lemma~\ref{lem: Ca estimate}, we have $\sup_{B_{\rho}}|h-\hbar| \leq C\rho^{\alpha}$ for constants $C, \alpha>0$ depending only on $\|h\|_{L^{\infty}}, \|h^{-1}\|_{L^{\infty}}$.  Thanks to H\"older's inequality, we obtain
\begin{equation*}
\begin{aligned}
\left|\int_{B_{\rho}}\delta^{ij} \Tr\left((\hbar h^{-1}-\mathbb{I}) \hat{\nabla}^{0}_{j}h (\hat{\nabla}^{0}_{i} k )^{\dagger}\right) dV\right| &\leq C\rho^{\alpha}\|\hat{\nabla}^{0}h\|_{L^{2}(\rho)} \|\hat{\nabla}^{0}k \|_{L^{2}(\rho)}\\
&\leq C\rho^{2\alpha} \|\hat{\nabla}^{0}h\|_{L^{2}(\rho)}^{2} + \frac{1}{10}\|\hat{\nabla}^{0} k\|^{2}_{L^{2}(\rho)},
\end{aligned}
\end{equation*}
where for simplicity we have used the symbol $L^{2}(\rho)$ to denote $L^{2}(B_{\rho})$.  Using the $L^\infty$ bound for the connection terms of $\hat{\nabla}^{0}$ yields the estimate
\begin{equation*}
\begin{aligned}
\left|\int_{B_{\rho}}\delta^{ij} \Tr\left((\hbar h^{-1}-\mathbb{I}) \hat{\nabla}^{0}_{j}h (\hat{\nabla}^{0}_{i} k )^{\dagger}\right) dV\right| &\leq C\rho^{2\alpha}\|dh\|^{2}_{L^{2}(\rho)} + \frac{1}{10} \|dk\|^{2}_{L^{2}(\rho)}\\ &+ C\rho^{2}(1 + \|k\|_{L^{\infty}(B_{\rho})}).
\end{aligned}
\end{equation*}
We now turn to the task of estimate the right hand side of equation~\eqref{eq: C1a ineq 1}.  The second term on the right hand side of~\eqref{eq: C1a ineq 1} is easily seen to be bounded by $C\rho^{2}\|k\|_{L^{\infty}(\rho)}$.  For the first term, we let $\eta_{i} := \Gamma_{i} - \hbar \Gamma_{i} \hbar^{-1}$, so that
\begin{equation*}
 \int_{B_{\rho}} \delta^{ij} \Tr\left([\Gamma_{i}, \hbar]h^{-1} (\del_{j}h)k\right)dV =  \int_{B_{\rho}} \delta^{ij} \Tr\left(\hbar h^{-1} (\del_{j}h)k \eta_{i}\right)dV.
 \end{equation*} 
Again writing $\hbar h^{-1} = \mathbb{I} +(\hbar h^{-1} - \mathbb{I})$ we obtain
\begin{equation}\label{eq: C1a ineq2}
\begin{aligned}
\left| \int_{B_{\rho}} \delta^{ij} \Tr\left(\hbar h^{-1} (\del_{j}h)k \eta_{i}\right)dV\right| &\leq \left| \int_{B_{\rho}} \delta^{ij} \Tr\left((\del_{j}h)k \eta_{i}\right)dV\right|\\ &+\left| \int_{B_{\rho}} \delta^{ij} \Tr\left((\hbar h^{-1}-\mathbb{I}) (\del_{j}h)k \eta_{i}\right)dV\right|.
\end{aligned}
\end{equation}
The first term on the right hand side of~\eqref{eq: C1a ineq2} is estimated as follows.  Integrate by parts and apply H\"older's inequality to obtain
\begin{equation*}
\left| \int_{B_{\rho}} \delta^{ij} \Tr\left((\del_{j}h)k \eta_{i}\right)dV\right| \leq \frac{1}{10}\|dk\|^{2}_{L^{2}(\rho)} + C(1+\|k\|_{L^{\infty}(\rho)}) \rho^{2}.
\end{equation*}
The second term on the right hand side of~\eqref{eq: C1a ineq2} is easily seen to be bounded by $C\rho^{2\alpha}\|dh\|^{2}_{L^{2}(\rho)} + C\rho^{2} \|k\|_{L^{\infty}(\rho)}^{2}$.  Finally, we consider the first term on the right hand side of equation~\eqref{eq: C1a ineq 3}.  We write
\begin{equation}\label{eq: C1a ineq 4}
\begin{aligned}
\int_{B_{\rho}}\delta^{ij} \Tr\left(\hat{\nabla}^{0}_{j}h (\hat{\nabla}^{0}_{i} k )^{\dagger}\right) dV =& \int_{B_{\rho}} \delta^{ij} \Tr\left(\del_{j}h (\del_{i} k )^{\dagger}\right) dV \\
&- \int_{B_{\rho}} \delta^{ij} \Tr\left(h \del_{j}([\hat{\Gamma}^{0}_{i}, k] )^{\dagger}\right) dV\\
&+\int_{B_{\rho}}\delta^{ij} \Tr\left([h, \hat{\Gamma}^{0}_{j}] (\hat{\nabla}_{i}^{0} k  )^{\dagger}\right) dV.
\end{aligned}
\end{equation}
We thus obtain the estimate
\begin{equation*}
\begin{aligned}
\left|\int_{B_{\rho}} \delta^{ij} \Tr\left(\del_{j}h (\del_{i} k )^{\dagger}\right) dV\right| \leq& \left|\int_{B_{\rho}}\delta^{ij} \Tr\left(\hat{\nabla}^{0}_{j}h (\hat{\nabla}^{0}_{i} k )^{\dagger}\right) dV\right|\\
 &+ \frac{1}{10} \|dk\|^{2}_{L^{2}(\rho)}+ C\rho^{2}(1+\|k\|_{L^{\infty}(\rho)}).
\end{aligned}
\end{equation*}
Combining all of the above estimates with equation~\eqref{eq: C1a ineq 3}, we have
\begin{equation}\label{eq: C1a ineq 5}
\begin{aligned}
\left|\int_{B_{\rho}} \delta^{ij} \Tr\left(\del_{j}h (\del_{i} k )^{\dagger}\right) dV\right| \leq& C\rho^{2\alpha}\|dh\|^{2}_{L^{2}(\rho)}+ \frac{1}{2} \|dk\|^{2}_{L^{2}(\rho)}\\&+ C\rho^{2}(1 + \|k\|_{L^{\infty}(B_{\rho})}).
\end{aligned}
\end{equation}
This inequality holds for any choice of compactly supported, hermitian matrix valued function $k \in L^{\infty}(B_{\rho}) \cap W^{1,2}(B_{\rho})$.  Define a smooth hermitian matrix valued function $w$ by
\begin{equation*}
\delta^{ij} \del_{i}\del_{j} w= 0, \qquad w |_{\del B_{\rho}} = h.
\end{equation*}
By the usual estimates for the Laplace equation (see e.g. \cite[Lemma 1.35]{HL}) we have
\begin{equation*}
\|w\|_{C^{\frac{\alpha}{2}}(\bar{B}_{\rho})} \leq C \|h\|_{C^{\alpha}(\del B_{\rho})}, \qquad \sup_{B_{\rho}}|w| \leq n^{2}\sup_{\del B_{\rho}}|h|,
\end{equation*}
for a constant $C$ depending only on $\alpha$, and hence only on $\|h\|_{L^{\infty}(B_{\rho})}, \|h^{-1}\|_{L^{\infty}(B_{\rho})}$, and $n $ denotes the rank of $E$.  Thanks to \cite[Lemma 3.10]{HL}, for any $r\in (0,\rho)$, we have
\begin{equation*}
\begin{aligned}
\int_{B_{r}(x_{0})} |dw|^{2}dV &\leq c \left(\frac{r}{\rho}\right)^{2} \int_{B_{\rho}} |dw|^{2}dV\\
\int_{B_{r}(x_{0})} |dw- dw_{r}|^{2}dV &\leq c \left(\frac{r}{\rho}\right)^{4} \int_{B_{\rho}}|dw -dw_{\rho}|^{2}dV
\end{aligned}
\end{equation*}
where we have used the symbol $dw_{r}$ to denote the average of $dw$ on $B_{r}(x_{0})$.  Set $v= h-w$. Since $w$ is harmonic, the estimate~\eqref{eq: C1a ineq 5} implies
\begin{equation}
\begin{aligned}
\left|\int_{B_{\rho}} \delta^{ij} \Tr\left(\del_{j}v (\del_{i} k )^{\dagger}\right) dV\right| \leq& C\rho^{2\alpha}\|dh\|^{2}_{L^{2}(\rho)}+ \frac{1}{2} \|dk\|^{2}_{L^{2}(\rho)}\\&+ C\rho^{2}(1 + \|k\|_{L^{\infty}(B_{\rho})}).
\end{aligned}
\end{equation}
for any choice of compactly supported, hermitian matrix valued function $k \in L^{\infty}(B_{\rho}) \cap W^{1,2}(B_{\rho})$.  Since $v$ is compactly supported, and $|v| \leq 2\|h\|_{L^{\infty}}$, we can take $k =v$ in the above estimate to obtain
\begin{equation}\label{eq: C1a ineq 6}
\|dv \|_{L^{2}(\rho)}^{2} \leq C\rho^{2\alpha}\|dh\|^{2}_{L^{2}(\rho)}+ C\rho^{2}.
\end{equation}
Now, standard estimates from the elliptic theory \cite[Corollary 3.1]{HL} imply that
\begin{equation}\label{eq: harmonic est 1}
\begin{aligned}
\int_{B_{r}(x_{0})} |dh|^{2}dV &\leq c \left(\frac{r}{\rho}\right)^{2} \int_{B_{\rho}} |dh|^{2}dV + c\int_{B_{\rho}} |dv|^{2}dV\\
\int_{B_{r}(x_{0})} |dh- dh_{r}|^{2}dV &\leq c \left(\frac{r}{\rho}\right)^{4} \int_{B_{\rho}}|dh -dh_{\rho}|^{2}dV + c\int_{B_{\rho}}|dv|^{2}dV.
\end{aligned}
\end{equation}
Combining the above with estimate~\eqref{eq: C1a ineq 6} we obtain
\begin{equation*}
\int_{B_{r}(x_{0})} |dh|^{2} dV \leq C\left(\left(\frac{r}{\rho}\right)^{2} + \rho^{2\alpha}\right) \|dh\|^{2}_{L^{2}(\rho)} + C\rho^{2}
\end{equation*}
By \cite[Lemma 3.4]{HL} there is a $\rho_{0}$ small, depending only on $C, \alpha$, which in turn depend only on the given data, such that, for any $r\in(0,\rho_{0}]$ there holds
\begin{equation*}
\int_{B_{r}(x_{0})}|dh|^{2} \leq C \left(\frac{r}{\rho_{0}}\right)^{2-\frac{\alpha}{2}} \int_{B_{\rho_{0}}} |dh|^{2} dV + C r^{2-\alpha}.
\end{equation*}
Combining this estimate with~\eqref{eq: C1a ineq 6} and the second equation in~\eqref{eq: harmonic est 1} implies that, for any $0 <r \leq \rho \leq \rho_{0}$
\begin{equation*}
\begin{aligned}
\int_{B_{r}(x_{0})} |dh- dh_{r}|^{2}dV \leq& c \left(\frac{r}{\rho}\right)^{4} \int_{B_{\rho}}|dh -dh_{\rho}|^{2}dV\\
&+C\frac{\rho^{2+\frac{3\alpha}{2}}}{\rho_{0}^{2-\frac{\alpha}{2}}}\left( \|dh\|^{2}_{L^{2}(B_{\rho_{0}})} +1\right) + C\rho^{2}.
\end{aligned}
\end{equation*}
Another application of \cite[Lemma 3.4]{HL} implies there is a constant $C$ depending only on the given data so that
\begin{equation*}
\int_{B_{\rho}(x_{0})} |dh- dh_{\rho}|^{2}dV \leq C \rho^{2+2\alpha},
\end{equation*}
for all $\rho \in(0, \rho_{0}]$.  The proposition easily follows.
\end{proof}

\end{document}